\documentclass[12pt]{amsart}

\usepackage{geometry}
\usepackage{booktabs}
\usepackage{amsmath}
\usepackage{amssymb}
\usepackage{amsthm}
\usepackage{amsfonts, dsfont}
\usepackage{mathrsfs} 
\usepackage{algorithm}
\usepackage{algorithmic}
\usepackage{paralist}
\usepackage{graphics} 
\usepackage{epsfig} 
\usepackage{graphicx}  
\usepackage{epstopdf}
\usepackage{epstopdf}
\usepackage{verbatim}
\usepackage{mathrsfs}
\usepackage{mathtools}
\usepackage{pstricks}
\usepackage{relsize}
\usepackage{tikz}
\usetikzlibrary{matrix}
\usepackage{subcaption}
\usepackage{pgfplots}
\usepackage{fixltx2e}
\usepackage{enumitem}
\usepackage{bm}
\usepackage{upgreek}
\usepackage{url}
\usepackage[colorlinks=true]{hyperref}
\hypersetup{urlcolor=blue, linkcolor=blue, citecolor=red}
\usepackage{cleveref}
\usepackage{bm}
\usepackage{appendix}

\usepackage{tcolorbox} 

\allowdisplaybreaks

\newcommand{\D}[1]{\mbox{\rm #1}} 
\newcommand{\dd}{\D{d}}

\definecolor{ao(english)}{rgb}{0.0, 0.5, 0.0}

\DeclareMathOperator*{\argmin}{argmin}

\usepackage[abbrev]{amsrefs}

\usepackage{amssymb}

\numberwithin{equation}{section}

\newtheorem{theorem}{Theorem}[section]
\newtheorem{lemma}[theorem]{Lemma}
\newtheorem{assumption}[theorem]{Assumption}

\newtheorem{proposition}[theorem]{Proposition}
\newtheorem{remark}[theorem]{Remark}
\newtheorem{definition}[theorem]{Definition}

\newtheorem{ex}{Example}[section]

\definecolor{ForestGreen}{RGB}{34,139,34}
\definecolor{ao(english)}{rgb}{0.0, 0.5, 0.0}

\begin{document}

\title[Particle method for optimization in Hilbert spaces]
{A derivative-free particle method for   optimization in Hilbert spaces}
\thanks{
}

\author{Hui Huang}
\address{Hui Huang \newline
Hunan University, 
School of Mathematics, \newline
Changsha, China
}
\email{\texttt{huihuang1@hnu.edu.cn}}

\author{Hicham Kouhkouh}
\address{Hicham Kouhkouh \newline 
University of Graz, 
Department of Mathematics and Scientific Computing -- NAWI, \newline
Graz, Austria
}
\email{\texttt{hicham.kouhkouh@uni-graz.at}}

\date{\today}

\begin{abstract}
We introduce a stochastic interacting particle system in separable Hilbert spaces together with its associated mean-field formulation. The model is shown to retain the characteristic consensus-driven structure of classical Consensus-Based Optimization, while accounting for the analytical challenges of infinite-dimensional dynamics. We establish well-posedness of the proposed dynamics and analyze the associated consensus mechanism. Furthermore, we derive convergence guarantees under suitable assumptions on the objective functional, showing concentration of the dynamics toward the minimizer in the long-time regime. This extends the applicability of the method to a broad class of infinite-dimensional optimization problems. 
In addition, we study the corresponding finite-particle system relevant for numerical implementation and propose a practical algorithm.
\end{abstract}



\maketitle

\section{Introduction}

In recent years, derivative-free optimization methods have attracted increasing attention in the context of large-scale and non-convex optimization problems arising in machine learning, physics, and applied mathematics. Among these approaches, Consensus-Based Optimization (CBO) \cite{pinnau2017consensus,carrillo2018analytical} has emerged as a particularly successful particle-based method. Motivated by collective interaction mechanisms observed in social and biological systems, CBO employs an ensemble of interacting agents to explore the search landscape and progressively concentrate around the global minimizer. Its robustness in the absence of gradient information has led to a rapidly growing literature addressing both its theoretical foundations and a wide range of extensions and applications \cite{ha2022stochastic,borghi2022consensus,carrillo2022consensus,carrillo2021consensus,fornasier2022anisotropic,fornasier2020consensus,fornasier2024truncated,byeon2025consensus,wei2025consensus,bonandin2025consensus,cipriani2022zero,huang2024fast,huang2023global,roith2025consensus}.

The CBO algorithm is distinguished by its derivative-free nature and its strong amenability to rigorous mathematical analysis \cite{fornasier2024consensus,huang2025faithful}. Consequently, the method has been extensively developed and adapted to a wide array of complex settings. Notable extensions include multiple-minimizer problems \cite{bungert2025polarized}, saddle point problems \cite{huang2024consensus,borghi2024particle}, multiplayer games \cite{chenchene2025consensus}, 
stochastic optimization \cite{bellavia2025discrete}, multi-level \cite{herty2025multiscale,garcia2025defending} and multi-objective optimization \cite{borghi2023adaptive}, clustered federated learning \cite{carrillo2023fedcbo}, as well as constrained optimization \cite{bungert2025mirrorcbo,beddrich2026constrained,borghi2023constrained}. Further advancements encompass jump diffusions \cite{kalise2023consensus,aceves2026consensus}, momentum acceleration \cite{chen2022consensus}, discrete-time formulations \cite{ha2024time,ha2020convergence,ko2022convergence}, and memory-based dynamics \cite{riedl2024leveraging,huang2024self,totzeck2020consensus}. Rather than providing an exhaustive account of this rapidly expanding field, we refer the interested reader to the survey by Totzeck \cite{totzeck2021trends} and more recent comprehensive reviews in \cite{fornasier2026consensus} for a broader perspective.

While the finite-dimensional Euclidean setting has been extensively studied, many problems arising in modern applications are inherently infinite-dimensional. Examples include variational problems, optimal control, and optimization constrained by partial differential equations. Extending particle-based optimization methods such as CBO to such settings is therefore both natural and necessary, but introduces substantial analytical challenges due to the lack of compactness and the infinite-dimensional structure of the underlying state space. Recent work has nevertheless generalized CBO beyond classical Euclidean spaces to genuinely infinite-dimensional settings. For instance, \cite{khatab2026consensus} employs Gaussian process representations within a CBO framework for optimization in Sobolev spaces, whereas \cite{borghi2026variational} formulates CBO on the manifold of Gaussian probability measures endowed with a Bures-Wasserstein geometry for variational inference. These approaches extend the optimization framework itself to non-classical and potentially infinite-dimensional spaces. However, none of these works directly formulate the CBO dynamics as a continuous-time Stochastic Differential Equation (SDE) in an infinite-dimensional setting. While these approaches provide a highly flexible algorithmic framework, a functional-analytic foundation for infinite-dimensional CBO remains an open challenge, specifically regarding the well-posedness of the underlying dynamics and the global convergence guarantees. To bridge this gap, our paper introduces and establishes a general SDE framework for CBO in separable Hilbert spaces.

Besides that, CBO has also been applied to problems whose underlying formulation is infinite-dimensional in nature, while the numerical optimization is still performed in a finite-dimensional parameter space. Examples include stochastic control problems addressed in \cite{lyu2025consensus} through gradient-free optimization of parameterized policies, as well as model predictive control formulations in \cite{borghi2025model} where CBO serves as an optimization engine. Other developments have further broadened the applicability of CBO as in \cite{beddrich2026constrained}.

In this manuscript, we consider the following optimization problem
\begin{equation*}
    \text{Find } x^* \in \argmin_{x \in H} \mathscr{E}(x),
\end{equation*}
where $H$ is a separable Hilbert space and $\mathscr{E}:H\to\mathbb{R}$ is a possibly non-convex and non-smooth functional.

To this end, we develop a derivative-free optimization framework in a separable Hilbert space setting, extending the classical CBO methodology from finite-dimensional Euclidean spaces to infinite-dimensional function spaces. We introduce an appropriate particle system and corresponding mean-field formulation in this context, and analyze its well-posedness and asymptotic behavior. In particular, we show that the proposed dynamics retain the characteristic concentration mechanism of CBO and establish convergence properties under suitable assumptions on the objective functional.

\subsection{Motivation}

The Hilbert space formulation introduced above naturally encompasses a broad class of problems arising in applications. We now present several classical examples that illustrate this infinite-dimensional setting and motivate the development of the proposed method. We refer to Appendix \ref{app:examples of H} for several examples of separable Hilbert spaces. 

\paragraph{\textbf{Example 1} (Stochastic Optimization)}

Let $U$ be a separable Hilbert space and $(\Omega,\mathcal{F},\mathbb{P})$ a probability space. We consider a stochastic optimization problem of the form
\[
    \min_{u \in U} \; \mathbb{E}\big[ J(u, \xi) \big],
\]
where $\xi$ is a random input and $U$ is a set of admissible controls.

A natural Hilbert space setting is obtained by lifting the control variable $u \in U$ to a random variable $\mathbf{u} : \Omega \longrightarrow U$, and identifying admissible controls with the Bochner space
\[
    \mathscr{U} := L^2(\Omega; U) = \left\{ \mathbf{u} : \Omega \longrightarrow U  \text{measurable} \;:\; \mathbb{E}[\,\|\mathbf{u}\|_U^2\,] < \infty \right\}.
\]
In this setting, the stochastic optimization problem can be written as
\[
    \min_{\mathbf{u} \in \mathscr{U}} \; \mathbb{E}\big[ J(\mathbf{u}(\omega), \xi(\omega)) \big] = \min_{\mathbf{u} \in L^2(\Omega;U)} \int_{\Omega} J(\mathbf{u}(\omega), \xi(\omega)) \, \mathrm{d}\mathbb{P}(\omega).
\]
Equivalently, $L^2(\Omega;U)$ is a Hilbert space equipped with inner product
\[
    \langle \mathbf{u}, \mathbf{v} \rangle_{L^2(\Omega;U)} = \mathbb{E}\big[\langle \mathbf{u}(\omega), \mathbf{v}(\omega)\rangle_U\big].
\]
In this formulation, randomness is absorbed into the state space, and the original stochastic problem becomes a variational problem on the infinite-dimensional Hilbert space $L^2(\Omega;U)$.
This allows to rewrite the stochastic optimization problem in the form
\[
    \min_{\mathbf{u} \in H} \; \mathscr{E}(\mathbf{u}), \quad \text{ where } H := L^2(\Omega; U), 
\]
and the energy functional $\mathscr{E}: H \to \mathbb{R}$ is defined by
\[
    \mathscr{E}(\mathbf{u}) := \mathbb{E}\big[ J(\mathbf{u}(\omega), \xi(\omega)) \big].
\]
In this sense, stochastic optimization can be interpreted as a deterministic variational problem posed on an infinite-dimensional Hilbert space of random variables.

\paragraph{\textbf{Example 2} (Optimal Control)}

We consider the following optimal control problem for an ODE
\begin{equation*}
\begin{aligned}
    & \min\limits_{u \in L^2(0,T)} J(y,u) := \frac{1}{2} \int_0^T \big(|y(t)|^2 + \alpha |u(t)|^2 \big)\,\dd t, \quad  \text{ for some } \alpha > 0,\\
    & \text{subject to: }\; y'(t) = y(t) + u(t), \quad t \in (0,T), \text{ and } y(0) = y_0.
\end{aligned}
\end{equation*}
We choose the Hilbert space $H := L^2(0,T)$, with inner product
\[
    \langle u,v \rangle_H := \int_0^T u(t)v(t)\,\dd t.
\]
For each control $u \in L^2(0,T)$, the ODE admits a unique solution $y = S(u)$, where
\[
    S: L^2(0,T) \longrightarrow H^1(0,T) \subset L^2(0,T)
\]
is the control-to-state operator. 
In this example, the solution is explicitly given by
\[
    y(t) = e^t y_0 + \int_0^t e^{t-s} u(s)\,\dd s.
\]
Substituting $y = S(u)$ into the cost functional yields the reduced functional
\[
    \mathscr{E}(u) := \frac{1}{2} \int_0^T \Big(|S(u)(t)|^2 + \alpha |u(t)|^2 \Big)\,\dd t.
\]
Finally, the optimal control problem is equivalent to the minimization problem
\[
    \min_{u \in H} \mathscr{E}(u), \quad \text{where } H = L^2(0,T).
\]

An analogous formulation is available in the setting of stochastic optimal control, provided that the admissible control set is chosen appropriately to encode not only integrability requirements, but also the relevant measurability and filtration-adaptedness constraints.

\paragraph{\textbf{Example 3} (Optimal Control of PDEs)}

Let $\Omega \subset \mathbb{R}^d$ be a bounded domain. We consider the optimal control problem
\begin{equation*}
\begin{aligned}
    & \min\limits_{u \in L^2(\Omega)} J(y,u) := \frac{1}{2}\|y - y_{\text{d}}\|_{L^2(\Omega)}^2 + \dfrac{\alpha}{2}\|u\|_{L^2(\Omega)}^2, \qquad  \text{ for some } \alpha > 0,\\
    & \text{subject to: }\; -\Delta y = u \quad \text{in } \Omega,\quad \text{ and } \; y = 0 \quad \text{on } \partial\Omega.
\end{aligned}
\end{equation*}
We define the Hilbert space $H := L^2(\Omega)$, with inner product
\[
    \langle u,v \rangle_H := \int_\Omega u(x)v(x)\,\dd x.
\]

For each $u \in L^2(\Omega)$, there exists a unique weak solution $y \in H_0^1(\Omega)$ of the Poisson equation. We define the solution operator
\[
    S : L^2(\Omega) \longrightarrow H_0^1(\Omega), \quad u \longmapsto S(u):=y,
\]
where $y$ is such as
\[
    -\Delta y = u \quad \text{in } \Omega, \qquad \text{ and } \; y|_{\partial\Omega} = 0.
\]
Substituting $y = S(u)$ into the cost functional $J(\cdot,\cdot)$ yields
\[
    \mathscr{E}(u) := \frac{1}{2}\|S(u) - y_{\text{d}}\|_{L^2(\Omega)}^2 + \frac{\alpha}{2}\|u\|_{L^2(\Omega)}^2.
\]
The optimal control problem is then equivalent to the minimization problem
\[
    \min_{x \in H} \mathscr{E}(x), \quad \text{ where } H = L^2(\Omega).
\]

Analogous optimization formulations also arise in the control of more general PDEs and SPDEs. In these settings, one must define suitable admissible control spaces that incorporate the appropriate spatial and temporal regularity requirements, as well as measurability and filtration-adaptedness conditions in the stochastic case.

\paragraph{\textbf{Example 4} (Inverse Problem)}

Let $H$ be a Hilbert space (e.g.,  $H = L^2(\Omega)$), and let $A : H \to H$ be a bounded linear operator. We consider the inverse problem $Ax = b$, where only noisy data $b^\delta \in H$ is available such as $b^\delta \approx b$. We aim to recover $x \in H$ from $Ax \approx b^\delta$, which is typically ill-posed. 
To stabilize the problem, a standard approach is to introduce the Tikhonov functional
\[
    \mathscr{E}(x) := \frac{1}{2}\|Ax - b^\delta\|_H^2 + \frac{\alpha}{2}\|x\|_H^2, \qquad \text{ for some } \alpha > 0.
\]
The inverse problem is then reformulated as the unconstrained minimization problem
\[
    \min_{x \in H}\, \mathscr{E}(x).
\]
where $x \in H$ is the unknown to be reconstructed.

\paragraph{\textbf{Example 5} (Calculus of Variations)} 

Many evolution equations arising in the calculus of variations admit a natural formulation in terms of energy minimization on infinite-dimensional Hilbert spaces. Classical examples are provided by the Allen--Cahn and Cahn--Hilliard equations, both of which are associated with the same Ginzburg--Landau energy functional, but endowed with different gradient-flow structures.

Let $\Omega \subset \mathbb{R}^d$ be a bounded domain and let $\varepsilon > 0$. We define the energy functional
\[
    \mathscr{E}(u) := \int_\Omega \left( \frac{\varepsilon^2}{2} |\nabla u(x)|^2 + F(u(x)) \right)\,\dd x,
\]
where $F$ is a double-well potential, for example $F(u)=\frac{1}{4}(u^2-1)^2$.

The Allen--Cahn equation is the $L^2(\Omega)$-gradient flow associated with $\mathscr{E}$. Formally, 
\[
    \partial_t u = - \frac{\delta \mathscr{E}}{\delta u},
\]
which yields \; $\partial_t u = \varepsilon^2 \Delta u - F'(u)$. A natural Hilbert space is $H := H_0^1(\Omega)$ equipped with inner product
\[
    \langle u,v\rangle_H := \int_\Omega \nabla u \cdot \nabla v\,\dd x.
\]

On the other hand, the Cahn--Hilliard equation arises as the $H^{-1}(\Omega)$-gradient flow associated with the same energy functional $\mathscr{E}$. Formally, one has
\[
    \partial_t u = \Delta \left( \frac{\delta \mathscr{E}}{\delta u}\right),
\]
which yields \; $\partial_t u = \Delta\bigl(-\varepsilon^2 \Delta u + F'(u)\bigr)$. The natural variational setting for this equation is the Hilbert space $H := H^{-1}(\Omega)$, equipped with the inner product
\[
    \langle f,g\rangle_{H^{-1}} = \int_\Omega \nabla (-\Delta)^{-1}f \cdot \nabla (-\Delta)^{-1}g\,\dd x,
\]
where $(-\Delta)^{-1}$ denotes the inverse Laplacian with suitable boundary conditions. In this framework, the Cahn--Hilliard equation is naturally interpreted as the gradient flow of $\mathscr{E}$ with respect to the $H^{-1}$-metric.

The associated stationary problem can therefore be written in the abstract form
\[
    \min_{u \in H} \mathscr{E}(u),
\]
with $H = H_0^1(\Omega)$ in the Allen--Cahn case and $H = H^{-1}(\Omega)$ in the Cahn--Hilliard case.

This shows that both models admit a unified variational formulation in infinite-dimensional Hilbert spaces. The energy functional is the same in both cases, while the choice of the underlying Hilbert-space metric determines the corresponding gradient-flow dynamics.

\subsection{Model Formulation}\label{sec:model}

\subsubsection{The dynamics}

Let us be given $(H, \langle \cdot, \cdot \rangle_H, \|\cdot\|_H)$ a separable Hilbert space. We aim at establishing a CBO model in this infinite-dimensional setting in order to find the global minimizer $x^*\in H$ of a continuous, non-convex objective function $\mathscr{E}: H\to \mathbb{R}$. This is a derivative-free method which relies on a swarm of interacting particles that explore the space $H$ and ultimately coalesce around the global minimizer $x^*$. 
In the mean field limit, the trajectory of a single representative particle $X_t$ is governed by the following McKean-Vlasov Stochastic Differential Equation (SDE) in $H$
\begin{equation}
\label{eq:mckean_vlasov_cbo}
\begin{aligned}
    & \dd X_t = -\lambda (X_t - \mathfrak{m}_{\alpha}(\mu_t))\,\dd t + \sigma\, \mathsf{D}(X_t- \mathfrak{m}_{\alpha}(\mu_t))\,\dd W_t^{Q}, \quad X_0 = \xi\\
    \text{ where }\; & \mathfrak{m}_{\alpha}(\mu_t) = \frac{\int_{H} x\, \exp(-\alpha\, \mathscr{E}(x))\,\mu_{t}(\dd x)}{\int_{H} \exp(-\alpha\,\mathscr{E}(x))\,\mu_{t}(\dd x)} \qquad \text{ and } \qquad  \mu_t = \text{Law}(X_t).
\end{aligned}
\end{equation} 
The parameter $\lambda>0$ is the intensity of the deterministic drift toward the consensus point, and $\sigma>0$ scales the stochastic exploration. The parameter $\alpha>0$ is the inverse-temperature: when $\alpha\to \infty$, the exponential weight heavily penalizes suboptimal positions leading the measure to concentrate around the global minimizers of $\mathscr{E}$. Finally, the operator $\mathsf{D}:H \to \mathscr{L}(U,H)$ modulates the noise intensity based on the distance of $X_t$ to the consensus point $\mathfrak{m}_{\alpha}(\mu_t)$. In particular, it decays as the distance decays, causing less exploration as the system converges. The driving noise $W_t^Q$ is a $Q$-Wiener process on a Hilbert space $U$ which we define in detail next. 

Since $H$ is a separable Hilbert space, it is naturally equipped with its Borel $\sigma$-algebra. Consequently, for a Borel-measurable function $f:H\to\mathbb{R}$ and a probability measure $\mu$ on $H$, the quantity
\[
    \int_H f(x)\,\mu(\dd x)
\]
denotes the usual Lebesgue integral of $f$ with respect to $\mu$, or equivalently, the expectation of $f$ under the probability law $\mu$. 

More generally, if $f:H\to H$ is an $H$-valued measurable function satisfying
\[
    \int_H \|f(x)\|_H\,\mu(\dd x)<\infty,
\]
then the integral
\[
    \int_H f(x)\,\mu(\dd x)
\]
is understood in the sense of the Bochner integral (see, e.g., \cite[Appendix A]{liu2015stochastic} or \cite[Appendix A]{prevot2007concise}), yielding an element of $H$. This provides the natural extension of Lebesgue integration to Hilbert-space-valued functions.

\subsubsection{The noise}
\label{subsec:q_wiener}

A fundamental challenge in transitioning from $\mathbb{R}^d$ to an infinite-dimensional space $H$ is that standard Brownian motion becomes ill-defined. If one were to inject an independent noise into every orthogonal direction of $H$, the total variance would diverge to infinity, driving the dynamics out of $H$.

To preserve finite energy, the stochastic exploration must be driven by a $Q$-Wiener process, which dampens the noise. 

Let $U$ be a separable Hilbert space (for example, $U=H$). We define a covariance operator $Q\in \mathscr{L}(U,U)$ that is linear, symmetric, and positive semi-definite. We assume moreover that $Q$ is a trace-class operator, that is its trace is finite
\begin{equation*}
    \mathrm{Tr}(Q) = \sum\limits_{k=1}^{\infty} \lambda_k \, < \infty
\end{equation*}
where $(\lambda_k)_{k\in \mathbb{N}}$ is the sequence of non-negative eigenvalues of $Q$. 

The Spectral theorem for compact, self-adjoint operators guarantee existence of a complete orthonormal basis $(\mathbf{e}_k)_{k\in\mathbb{N}}$ of $U$ consistence of the eigenvectors of $Q$, such that $Q\,\mathbf{e}_k = \lambda_k\,\mathbf{e}_k$. Using this eigenbasis, we can construct the $Q$-Wiener process $W_t^Q$ via its Karhunen-Lo\`{e}ve expansion (see \cite{alexanderian2015brief})
\begin{equation*}
    W_t^Q = \sum\limits_{k=1}^{\infty}\sqrt{\lambda_k}\,\beta_{k}(t)\,\mathbf{e}_k
\end{equation*}
where $(\beta_{k}(t))_{k\in\mathbb{N}}$ is a sequence of independent, standard one-dimensional Brownian motions. The following remark provides an intuitive explanation of this construction. 

\begin{remark}
Each eigenvalue $\lambda_k$ represents the variance (or stochastic energy) injected into the spatial direction $\mathbf{e}_k$. The trace-class condition $\sum\lambda_k<\infty$ forces this variance to decay rapidly for higher dimensions, which ensures the noise remains bounded in the space $U$. Indeed, using It\^{o}'s isometry, one obtains
\begin{equation*}
    \mathbb{E}\left[\|W_t^Q\|_U^2\right] = \mathbb{E}\left[\left\| \sum_{k=1}^\infty \sqrt{\lambda_k}\, \beta_k(t)\, \mathbf{e}_k \right\|_U^2\right] = \sum_{k=1}^\infty \lambda_k \, \mathbb{E}[|\beta_k(t)|^2] = t \sum_{k=1}^\infty \lambda_k = t\, \mathrm{Tr}(Q) < \infty.
\end{equation*}
\end{remark}

We refer to Appendix \ref{app:noise} for more details.

\subsection{Analytical Challenges and Main Contributions}

Transitioning from finite-dimensional Euclidean space ($\mathbb{R}^d$) to a separable Hilbert space ($H$) introduces theoretical obstacles for particle-based optimization. In $\mathbb{R}^d$, standard CBO algorithms rely on isotropic white noise to systematically explore the objective landscape. However, in an infinite-dimensional Hilbert space, standard cylindrical Brownian motion possesses infinite trace, meaning isotropic exploratory noise would immediately drive the particle swarm out of the space $H$. 

Consequently, infinite-dimensional stochastic exploration must be driven by a $Q$-Wiener process with a trace-class covariance operator ($\mathrm{Tr}(Q) < \infty$). While this guarantees that the trajectories remain in $H$, it creates an algorithmic flaw: the variance injected into higher orthogonal frequencies must decay rapidly to zero. If the global minimizer resides in these higher frequencies, the trace-class noise lacks the requisite energy to push the swarm out of local minima, leading to premature convergence and algorithmic failure.

To resolve this tension between finite energy and exploratory viability, this manuscript introduces three primary contributions:

\begin{enumerate}
    \item \textbf{The Active Subspace and Projected Diffusion:} We introduce a geometric decomposition of the Hilbert space into a finite-dimensional active subspace $V$ and an infinite-dimensional orthogonal complement $V^\perp$. By proposing a novel Projected Isotropic Diffusion operator ($\mathsf{D}(x) = \|x\|_H \mathsf{P}_V$) to be later made precise, we explicitly filter the stochastic noise, trapping the exploratory variance within $V$ and reducing the dynamics in $V^\perp$ to a purely deterministic drift toward the consensus point. This completely insulates the optimization process from the trace-class decay of the ambient space.
    
    \item \textbf{Global Well-Posedness via Wasserstein Contraction:} Due to the non-linear, non-local nature of the consensus point in $H$, standard Picard iterations fail. We establish global well-posedness by proving that the decoupled measure-mapping operator is a strict contraction on the space of continuous probability measure flows $\mathscr{C}_T = C([0,T]; \mathscr{P}_2(V))$.
    
    \item \textbf{Strict Positivity of Mass and Global Convergence:} The core of our convergence analysis relies on a novel infinite-dimensional It\^o calculus argument. By utilizing the explicit structural properties of our projected diffusion operator, we prove that the swarm maintains a strictly positive probability mass in the neighborhood of the minimizer. This circumvents the infinite-dimensional small ball problem, allowing us to invoke a quantitative Laplace principle and guarantee global exponential convergence to the minimizer.
\end{enumerate}

\subsection{Organization}

The remainder of the paper is organized as follows. 
\textbf{Section \ref{sec:wellposedness}} is devoted to the well-posedness analysis of the dynamics, where we study in particular the properties of the consensus point and establish existence and uniqueness via a fixed-point argument. In \textbf{Section \ref{sec:convergence}}, we derive convergence guarantees for the proposed method under suitable assumptions on the objective functional. \textbf{Section \ref{sec:particle}} addresses the interacting particle system with finitely many agents, which is directly relevant for numerical implementation. In \textbf{Section \ref{sec:numerics}}, we present the resulting algorithm and illustrate its performance through numerical experiments.

Finally, \textbf{Appendix \ref{app:sde in H}} collects the necessary background on stochastic differential equations in Hilbert spaces, while \textbf{Appendix \ref{app:nonsmooth}} provides elements of nonsmooth analysis required for the infinite-dimensional calculus used throughout the paper. In \textbf{Appendix \ref{app:examples of H}}, we recall several examples of seperable Hilbert spaces that arise in analysis.

\section{Well-Posedness Analysis} 
\label{sec:wellposedness}

Let $\mathscr{P}_{2}(H)$ denote the space of probability measures on $H$ with finite second moments, equipped with the $2$-Wasserstein distance $\mathcal{W}_2$. Let $U=H$ and $Q\in \mathscr{L}(H,H)$ be the trace-class covariance operator (which is positive semi-definite and self-adjoint) of the Wiener process $W_t^Q$.

We define $\mathscr{L}_Q(H, H)$ 
as the space of linear operators $B: H_0 \to H$ where $H_0 = Q^{1/2}H$, such that $B Q^{1/2}$ is a Hilbert-Schmidt operator from $H$ to $H$. This space is equipped with the norm
\begin{equation}
    \|B\|_{\mathscr{L}_Q}^2:= \|B Q^{1/2}\|_{\text{HS}}^2=\sum_{i=1}^\infty\|B Q^{1/2}\,\mathbf{e}_i\|_H^2 = \mathrm{Tr}(B Q B^*)
\end{equation}
where we recall for $T$ a Hilbert-Schmidt operator, we have $\|T\|_{\text{HS}}^{2} := \mathrm{Tr}(TT^*)$.

To establish well-posedness of the infinite-dimensional McKean-Vlasov equation \eqref{eq:mckean_vlasov_cbo}, we require the objective function $\mathscr{E}:H\to \mathbb{R}$ to satisfy the following assumptions.

\begin{assumption}[Regularity of the Objective Function]
\label{assum:objective}
The function $\mathscr{E}$ satisfies the following properties:
\begin{enumerate}
    \item \textbf{Bounded from below:} Without loss of generality, we assume 
    \begin{equation*}
        \inf_{x \in H} \mathscr{E}(x) \geq 0 .
    \end{equation*}
    If bounded below by a negative constant, then a simple shift would leave the consensus point $\mathfrak{m}_\alpha$ invariant. 
    \item \textbf{Polynomial growth:} There exist constants $C_{\mathscr{E}} > 0$ and $p \geq 1$ such that 
    \begin{equation*}
        \mathscr{E}(x) \leq C_{\mathscr{E}} \big(1 + \|x\|_H^p\big) \qquad \forall\, x\in H.
    \end{equation*}
    \item  \textbf{Coercive with sufficient growth:} There exist constants $c, R_0 > 0$ and $p \geq 1$ such that 
       \begin{equation*}
       \mathscr{E}(x) \geq c\|x\|_H^p \quad \mbox{ for all }\quad \|x\|_H > R_0.
    \end{equation*}
    \item \textbf{Local Lipschitz continuity:} For any radius $R > 0$, there exists a constant $L > 0$ such that for all $x, y \in H$ and $p\geq 1$,
    \begin{equation*}
    \begin{aligned}
        |\mathscr{E}(x) - \mathscr{E}(y)|
        & \leq L(1+\|x\|_H^{p-1}+\|y\|_H^{p-1})\|x - y\|_H.
    \end{aligned}
    \end{equation*}
\end{enumerate}
\end{assumption}

The following result establishes global Lipschitz continuity of the main terms in the dynamics \eqref{eq:mckean_vlasov_cbo}.

\begin{proposition}
\label{prop:global_lipschitz}
Under Assumption \ref{assum:objective}, the mappings 
\begin{equation*}
    w(x) = \exp(-\alpha\, \mathscr{E}(x)) \in \mathbb{R}, \qquad \text{ and } \qquad f(x) = x\,\exp(-\alpha\, \mathscr{E}(x)) \,\in H
\end{equation*}
are globally Lipschitz continuous on $H$.
\end{proposition}

\begin{remark}
 Due to the lack of a translation-invariant Lebesgue measure in infinite-dimensional Hilbert spaces, classical ``almost everywhere'' arguments are not available. This leads us to adopt tools from non-smooth analysis to handle non-differentiability in the study of Hilbert space-valued functions. These are recalled in Appendix \ref{app:nonsmooth}.    
\end{remark}
\begin{proof}

We provide the proof for $f(x)$, as the proof for $w(x)$ follows similarly.

We shall apply results in Appendix \ref{app:nonsmooth} where $X=H$ is a Hilbert space. In particular, the duality product $\langle\,\cdot,\cdot\,\rangle$ reduces to a scalar product after using Riesz identification of $H^*$ with $H$. 

We consider the function 
\begin{equation*}
\begin{aligned}
    f:H\longrightarrow H, \qquad x\longmapsto f(x) := x\,\exp(-\alpha\, \mathscr{E}(x))
\end{aligned}
\end{equation*}
which can be seen as the product of the identity map $H\to H$ with the map $g\circ h : H\longrightarrow \mathbb{R}$ where 
\begin{equation*}
\begin{aligned}
    h:H & \longrightarrow \mathbb{R},\;
    x  \longmapsto h(x):=\mathscr{E}(x)
\end{aligned}
\qquad \text{ and } \qquad
\begin{aligned}
    g:\mathbb{R} & \longrightarrow \mathbb{R},\;
    t  \longmapsto g(t):=e^{-\alpha \,t}.
\end{aligned}
\end{equation*}
All these functions are Lipschitz near any element of their domains. Moreover, $g$ is strictly differentiable (because it is continuously differentiable)\footnote{See Definition \ref{def: cont diff} for continuously differentiable, Definition \ref{def: stric deriv} for strictly differentiable, and Proposition \ref{prop:diff} for the connection between the two.}. Therefore one obtains the generalized gradient (Definition \ref{def: general grad}) of $f$ using the chain rule (Theorem \ref{thm:Chain rule}) and the product rule (Proposition \ref{prop:product}) as follows
\begin{equation}\label{subdiff f}
    \partial f(x) \subset \left\{ e^{-\alpha \mathscr{E}(x)}I - \alpha\,e^{-\alpha\mathscr{E}(x)}x\otimes \xi\,:\; \xi \in \partial \mathscr{E}(x)\right\},
\end{equation}
where $x\otimes \xi: v\mapsto \langle \xi, v \rangle x$ is a rank-$1$ operator, and $I$ is the identity operator in $H$. 
Note that in Appendix \ref{app:nonsmooth} we have recalled the generalized gradient for functions valued in $\mathbb{R}$, whereas in our case, we have $f(x) = x\,\exp({-\alpha \,\mathscr{E}(x)}) \in H$. Those results can be extended directly, by reasoning coordinate-wise: 
\begin{equation*}
\begin{aligned}
    f(x) = \sum\limits_{i=1}^{\infty} \,e^{-\alpha \,\mathscr{E}(x)} \, \langle x, \mathbf{e}_i \rangle\, \mathbf{e}_i
\end{aligned}
\end{equation*}
where $(\mathbf{e}_i)_{i=1}^{\infty}$ is an orthonormal basis of $H$, and   $x\mapsto e^{-\alpha \,\mathscr{E}(x)} \, \langle x, \mathbf{e}_i \rangle$ is now a function from $H\to \mathbb{R}$.

Using the Mean-Value Theorem (Theorem \ref{thm:MV}), one has
\begin{equation*}
    f(x) - f(y)  \in \langle \partial f(u), x-y \rangle \quad \text{ for some } u\in [x,y],
\end{equation*}
that is
\begin{equation*}
    f(x) - f(y)  = \langle \zeta, x-y \rangle \quad \text{ for some } \zeta \in \partial f(u).
\end{equation*}
Therefore, using \eqref{subdiff f} we can express $\partial f$ with $\partial \mathscr{E}$, and get 
\begin{align*}
    \|f(x) - f(y)\| & \leq \|x-y\|\,\sup\limits_{\substack{\zeta \in \partial f(u),\\ u\in [x,y]}} \|\zeta\|\\
    & \leq \|x-y\|\,\sup\limits_{\substack{\xi \in \partial \mathscr{E}(u),\\ u\in [x,y]}} \bigg(e^{-\alpha\,\mathscr{E}(u)} + \alpha\,\|x\|\,\|\xi\|\,e^{-\alpha\,\mathscr{E}(u)}\bigg)\\
    & \leq \|x-y\|\, \left(e^{-\alpha\,\mathscr{E}(u)} + \alpha\,\|x\|\,e^{-\alpha\,\mathscr{E}(u)}\,\sup\limits_{\substack{\xi \in \partial \mathscr{E}(u),\\ u\in [x,y]}}\,\|\xi\|\,\right).
\end{align*}
Since $\mathscr{E}$ is locally Lipschitz, we have
\begin{equation*}
    \|\xi\| \leq L(1+\|x\|_H^{p-1}+\|y\|_H^{p-1}),\quad \forall \, \xi \in \partial \mathscr{E}(u),\, u\in [x,y]
\end{equation*}
hence
\begin{equation*}
    \sup\limits_{\substack{\xi \in \partial \mathscr{E}(u),\\ u\in [x,y]}}\,\|\xi\| \leq L(1+\|x\|_H^{p-1}+\|y\|_H^{p-1}),
\end{equation*}
and
\begin{equation*}
\begin{aligned}
    \|f(x) - f(y)\| & \leq \|x-y\|\, \left(e^{-\alpha\,\mathscr{E}(u)} + \alpha\,\|x\|\,e^{-\alpha\,\mathscr{E}(u)}L\,(1+\|x\|_H^{p-1}+\|y\|_H^{p-1})\,\right)\\
    & \leq \|x-y\|\, \left(1 + \alpha\,\|x\|L\,(1+\|x\|_H^{p-1}+\|y\|_H^{p-1})\,\right)e^{-\alpha\,\mathscr{E}(u)}.
\end{aligned}
\end{equation*}
It suffices now to observe that $u=t\,x+(1-t)\,y$ for some $t\in [0,1]$, and that the function 
\begin{equation*}
    (x,y)\mapsto \left(1 + \alpha\,\|x\|L\,(1+\|x\|_H^{p-1}+\|y\|_H^{p-1})\,\right)e^{-\alpha\,\mathscr{E}(t\,x + (1-t)\,y)}
\end{equation*}
is bounded for all $(x,y)\in H^{2}$ thanks to the polynomial growth of $\mathscr{E}$ and its non-negativity. This provides a global Lipschitz bound for $f$, and concludes the proof. 
\end{proof}

Our next assumptions concern the covariance operator  of the Wiener process $Q$ and the space $H$. 
We recall $Q:H\to H$ is a symmetric, positive semi-definite, and trace-class operator over $H$ (i.e., $\mathrm{Tr}(Q) < \infty$), and $H$ is an infinite-dimensional separable Hilbert space with inner product $\langle \cdot, \cdot \rangle_H$ and induced norm $\|\cdot\|_H$.

\begin{assumption}[Finite-Dimensional Active Subspace] \label{assump:active_subspace}
We assume there exists a linear subspace $V \subset H$ satisfying the following conditions:

\begin{enumerate}[label=(\roman*)]
    \item 
    The subspace $V$ has a finite dimension $d < \infty$, and the ambient space admits the orthogonal decomposition $H = V \oplus V^\perp$. Furthermore, the global minimizer $x^*$ of $\mathscr{E}$ resides in $V$, and the initial probability measure $\mu_0 = \mathrm{Law}(\xi)$ is supported entirely on $V$. 
    
    \item 
    The restriction of $Q$ to the active subspace $V$ is strictly positive-definite, that is there exists a constant $\lambda_{\min} > 0$ such that
    \begin{equation*}
        \langle v, Q v \rangle_H \geq \lambda_{\min} \|v\|_H^2 \qquad \forall\,v\in V.
    \end{equation*}
\end{enumerate}
\end{assumption}

From now on, we define the diffusion operator as 
\begin{equation}\label{def:diffusion}
\begin{aligned}
    \mathsf{D}: H & \longrightarrow \mathscr{L}(H,H) \\
    x & \longmapsto \mathsf{D}(x) := \|x\|_H \mathsf{P}_V\,,
\end{aligned}
\end{equation}
where $\mathsf{P}_V$ is the orthogonal projection onto the active subspace $V$. We recall for each $x\in H$, $\mathsf{D}(x) \in \mathscr{L}(H,H)$ is a linear operator from $H\to H$. As defined in \eqref{def:diffusion}, it is simply the projection of each element $H$ onto $V\subset H$, scaled with $\|x\|_{H}$. 

Thus we can rewrite \eqref{eq:mckean_vlasov_cbo} as  a projected isotropic diffusion model
\begin{equation}
\label{eq:mckean_vlasov_cbo_P}
  \dd X_t = -\lambda \big(X_t - \mathfrak{m}_\alpha(\mu_t)\big)\,\dd t + \sigma \|X_t - \mathfrak{m}_\alpha(\mu_t)\|\mathsf{P}_V\,\dd W_t^Q, \qquad X_0 = \xi\,. 
\end{equation}

The diffusion operator $\mathsf{D}$ defined above, which scales the noise based on the particle's distance to the consensus point, enjoys the following regularity and structure: Lipschitz continuity, linear growth, and coercivity on the the active subspace $V$. These are summarized next.

\begin{proposition}
\label{prop:diffusion}
Let Assumption \ref{assump:active_subspace} hold, and $V\subset H$ be as defined therein.  
Then the mapping $\mathsf{D}: H \to \mathscr{L}_Q(H, H)$ defined in \eqref{def:diffusion} satisfies the following properties:

\begin{enumerate}
    \item 
    For all $x, y \in H$, it holds that
    \begin{equation*}
        \|\mathsf{D}(x) - \mathsf{D}(y)\|_{\mathscr{L}_Q} \leq   L_{\mathsf{D}} \|x - y\|_H,
    \end{equation*}
    where $L_{\mathsf{D}}:=\|\mathsf{P}_V\|_{\mathscr{L}_Q}=\mathrm{Tr}(\mathsf{P}_V Q \mathsf{P}_V)^{1/2}$.
    
    \item 
    For all $x\in H$
    \begin{equation*}
        \|\mathsf{D}(x)\|_{\mathscr{L}_Q} \leq  L_{\mathsf{D}}\|x\|_H \qquad \text{ and } \qquad \mathsf{D}(0) = 0.
    \end{equation*}

    \item  
    For all $x \in H$ and all $v \in V$
    \begin{equation*}
        \|\mathsf{D}(x)^* v\|_H^2 =  \|x\|_H^2 \|v\|_H^2.
    \end{equation*}
\end{enumerate}
\end{proposition}

\begin{proof} 
We prove each statement separately. 

\textbf{Step 1.} \textit{(Lipschitz continuity)} 
Let $x, y \in H$ be arbitrary. We evaluate the norm
\begin{align*}
    \|\mathsf{D}(x) - \mathsf{D}(y)\|_{\mathscr{L}_Q} = \big\| \|x\|_H \mathsf{P}_V - \|y\|_H \mathsf{P}_V \big\|_{\mathscr{L}_Q}\leq \|\mathsf{P}_V\|_{\mathscr{L}_Q}\, \|x - y\|_H
\end{align*}
The Lipschitz constant being $\|\mathsf{P}_V\|_{\mathscr{L}_Q}= \mathrm{Tr}(\mathsf{P}_V Q \mathsf{P}_V)^{1/2}$ completes the proof of the first property.

\textbf{Step 2.} \textit{(Linear growth)} 
The linear growth follows directly from the global Lipschitz continuity proven in Step 1 by setting $y = 0$. The operator vanishing at the origin is straightforward: $\mathsf{D}(0) = \|0\|_H \mathsf{P}_V =0.$

\textbf{Step 3.} \textit{(Coercivity in $V$)} 
We first determine the adjoint of the diffusion operator, $\mathsf{D}(x)^*$. Because $\mathsf{P}_V$ is an orthogonal projection onto a closed subspace, it is self-adjoint by definition ($\mathsf{P}_V^* = \mathsf{P}_V$). Since $\|x\|_H$ is a scalar, the entire operator is self-adjoint
\[
    \mathsf{D}(x)^* = \big( \|x\|_H \mathsf{P}_V \big)^* = \|x\|_H \mathsf{P}_V^* = \|x\|_H \mathsf{P}_V.
\]
Let $v \in V$ be an arbitrary direction vector in the active subspace. By the definition of an orthogonal projection, $\mathsf{P}_V$ acts as the identity on elements already inside $V$, so $\mathsf{P}_V v = v$. Applying the adjoint operator to $v$ gives
\[
    \mathsf{D}(x)^* v = \|x\|_H (\mathsf{P}_V v) = \|x\|_H v.
\]
Finally, we compute the squared norm
\[
    \|\mathsf{D}(x)^* v\|_H^2 = \big\| \|x\|_H v \big\|_H^2 = \|x\|_H^2 \|v\|_H^2\,.
\]
\end{proof}

Based on the previous setting, we can now state the main existence and uniqueness result for the infinite-dimensional CBO dynamics. 

By a strong solution, we mean an $H$-valued, predictable stochastic process $X_t$ with continuous paths almost surely, which satisfies \eqref{eq:mckean_vlasov_cbo_P} $\mathbb{P}$-almost surely for all $t \in [0,T]$. See Definition \ref{def:solution in H} and Remark \ref{rmk:uniq} in Appendix \ref{app:sde in H}. 

\begin{theorem}[Well-Posedness in $H$]
\label{thm:well_posedness}
Let $T > 0$ be a fixed time horizon. Suppose Assumptions \ref{assum:objective} and \ref{assump:active_subspace} hold. Let the initial condition $\xi$ be an $\mathcal{F}_0$-measurable $H$-valued random variable such that $\mathbb{E}[\|\xi\|_H^2] < \infty$. 
Then, the McKean-Vlasov SDE \eqref{eq:mckean_vlasov_cbo_P} admits a unique strong solution $X \in L^2\big(\Omega; C([0,T]; H)\big)$. 
\end{theorem}

The proof of Theorem \ref{thm:well_posedness} relies on a measure-freezing (decoupling) argument on the Wasserstein space $C([0,T]; \mathscr{P}_2(V))$. However, owing to the non-linear, non-local nature of the consensus point $\mathfrak{m}_\alpha(\mu)$, standard Picard iterations fail. The critical step is establishing a localized Lipschitz bound for the map $\mu \mapsto \mathfrak{m}_\alpha(\mu)$ in the $\mathcal{W}_2$ metric, which we tackle in the next subsections.

\subsection{The Consensus Point}
\label{sec:consensus_lipschitz}

The central difficulty in establishing well-posedness for the infinite-dimensional McKean-Vlasov SDE \eqref{eq:mckean_vlasov_cbo} is the highly non-linear nature of the consensus point mapping $\mu \mapsto \mathfrak{m}_\alpha(\mu)$. Because the objective function $\mathscr{E}$ satisfies only local Lipschitz and polynomial growth conditions, standard global contraction arguments fail. 

To proceed, we establish a localized Lipschitz bound for $\mathfrak{m}_\alpha$ with respect to the $2$-Wasserstein metric. Let  $\mathscr{P}_{2,R}(H)$ denote the space of probability measures on $H$ with second moments bounded by $R^2$ for some $R>0$, i.e.,
\begin{equation*}
    \mathscr{P}_{2,R}(H) = \left\{ \mu \in \mathscr{P}_2(H) \; : \; \int_H \|x\|_H^2 \,\mu(\dd x) \leq R^2 \right\}.
\end{equation*}

\begin{lemma}[Local Lipschitz Bound in $\mathcal{W}_2$]
\label{lem:local_lipschitz_m}
Under Assumption \ref{assum:objective}, for any $R > 0$, there exists a constant $L_{\mathfrak{m}}(R) > 0$ such that for all
\begin{equation*}
    \|\mathfrak{m}_\alpha(\mu) - \mathfrak{m}_\alpha(\nu)\|_H \leq L_{\mathfrak{m}}(R)\, \mathcal{W}_2(\mu, \nu) \qquad \mu, \nu \in \mathscr{P}_{2,R}(H).
\end{equation*}
\end{lemma}

\begin{proof}
Let $w(x) = \exp(-\alpha \, \mathscr{E}(x))$. We decompose the consensus point into a numerator operator $N(\mu) = \int_H x\, w(x)\, \mu(\dd x)$ and a denominator functional $Z(\mu) = \int_H w(x)\, \mu(\dd x)$. 

For any two measures $\mu, \nu \in \mathscr{P}_{2,R}(H)$, the difference is bounded such as
\begin{align}
    \label{eq:m_diff_bound_R}
    \|\mathfrak{m}_{\alpha}(\mu) - \mathfrak{m}_{\alpha}(\nu)\|_H &= \left\| \frac{N(\mu)Z(\nu) - N(\nu)Z(\mu)}{Z(\mu)Z(\nu)} \right\|_H \nonumber \\
    &\leq \frac{1}{|Z(\mu)|} \|N(\mu) - N(\nu)\|_H + \frac{\|N(\nu)\|_H}{|Z(\mu)Z(\nu)|} |Z(\mu) - Z(\nu)|.
\end{align}

\textbf{Step 1.} \textit{(Useful estimates)}

First, we bound the denominator $Z(\mu)$ strictly away from zero. By the definition of $\mathscr{P}_{2,R}(H)$, the second moment is bounded as $\int_H \|x\|_H^2\, \mu(\dd x) \leq R^2$. By Chebyshev's inequality, the probability mass outside a ball of radius $\sqrt{2}R$ is strictly bounded
\begin{equation*}
    \mu\big(\{x \in H : \|x\|_H \geq \sqrt{2}R\}\big) \leq \frac{1}{(\sqrt{2}R)^2} \int_H \|x\|_H^2 \,\mu(\dd x) \leq \frac{R^2}{2R^2} = \frac{1}{2}.
\end{equation*}
Consequently, the measure of the closed ball $B_{\sqrt{2}R}$ is at least $1/2$. By Assumption \ref{assum:objective}, the objective function $\mathscr{E}$ has a polynomial growth and is locally bounded, so on $B_{\sqrt{2}R}$, there exists a finite supremum $M_R = \sup\limits_{\|x\|_H \leq \sqrt{2}R} \mathscr{E}(x) < \infty$. Therefore
\begin{equation}
\label{eq:lowerbound Z}
\begin{aligned}
    Z(\mu) &= \int_H \exp(-\alpha \,\mathscr{E}(x))\, \mu(\dd x) \geq \int_{B_{\sqrt{2}R}} \exp(-\alpha\, \mathscr{E}(x))\, \mu(\dd x) \\
    &\geq \exp(-\alpha\, M_R) \mu(B_{\sqrt{2}R}) \geq \frac{1}{2} \exp(-\alpha\, M_R) \, =: \, Z_{\min}(R)>0.
\end{aligned}
\end{equation} 
Thus, $Z(\mu) \geq Z_{\min}(R)$ and, by symmetry, $Z(\nu) \geq Z_{\min}(R)$.

Next, we bound the numerator $\|N(\nu)\|_H$. Since $\mathscr{E}(x) \geq 0$, we have $w(x) \leq 1$. Applying Jensen's inequality (or the Cauchy-Schwarz inequality for probability measures)
\begin{equation*}
    \|N(\nu)\|_H \leq \int_H \|x\|_H\, \nu(\dd x) \leq \left( \int_H \|x\|_H^2 \,\nu(\dd x) \right)^{1/2} \leq R.
\end{equation*}

Finally, we verify the Lipschitz continuity of the integrals. Due to the rapid decay of the exponential weight $w(x)$ dominating the polynomial growth of $\mathscr{E}(x)$ (along with its local Lipschitz property), the mappings $x \mapsto w(x)$ and $x \mapsto x \,w(x)$ are globally Lipschitz on $H$ with constants $L_w$ and $L_f$, respectively. See  Proposition \ref{prop:global_lipschitz}.

\textbf{Step 2.} \textit{(Conclusion)}  

Let $\pi \in \Pi(\mu, \nu)$ be the optimal coupling achieving the 2-Wasserstein distance. We bound the difference in the denominators
\begin{align*}
    |Z(\mu) - Z(\nu)| &\leq \int_{H \times H} |w(x) - w(y)| \pi(\dd x, \dd y) \\
    &\leq L_w \int_{H \times H} \|x - y\|_H \pi(\dd x, \dd y) \\
    &\leq L_w \left( \int_{H \times H} \|x - y\|_H^2 \pi(\dd x, \dd y) \right)^{1/2} = L_w\, \mathcal{W}_2(\mu, \nu).
\end{align*}
With a similar argument for the numerator, one gets
\begin{align*}
    \|N(\mu) - N(\nu)\|_H &\leq \int_{H \times H} \|x \,w(x) - y\, w(y)\|_H \pi(\dd x, \dd y) \\
    &\leq L_f \left( \int_{H \times H} \|x - y\|_H^2 \pi(\dd x, \dd y) \right)^{1/2} = L_f \,\mathcal{W}_2(\mu, \nu).
\end{align*}
Substituting these bounds back into \eqref{eq:m_diff_bound_R} yields
\begin{equation*}
    \|\mathfrak{m}_{\alpha}(\mu) - \mathfrak{m}_{\alpha}(\nu)\|_H \leq \left( \frac{L_f}{Z_{\min}(R)} + \frac{R L_w}{Z_{\min}(R)^2} \right) \mathcal{W}_2(\mu, \nu).
\end{equation*}
Setting $L_m(R) = \frac{L_f}{Z_{\min}(R)} + \frac{R L_w}{Z_{\min}(R)^2}$ completes the proof.
\end{proof}

This localized Lipschitz bound guarantees that the mean-field drift is sufficiently regular, provided the particle distribution does not instantaneously escape to infinity. This sets the stage for a truncation-based fixed-point argument.

\begin{lemma}[Moment Bound for the Consensus Point]
\label{lem:m_alpha_moment}
Suppose Assumption \ref{assum:objective} holds (in particular, $\mathscr{E}(x) \geq 0$ for all $x \in H$). Then for any probability measure $\mu \in \mathscr{P}_2(V)$
\begin{equation*}
    \|\mathfrak{m}_\alpha(\mu)\|_H^2 \leq \frac{1}{Z(\mu)} \int_V \|x\|_H^2\, \mu(\dd x),
\end{equation*}
where $Z(\mu) = \int_V \exp(-\alpha\, \mathscr{E}(x)) \,\mu(\dd x)$. That is, the squared norm of the consensus point is bounded by the second moment of $\mu$ scaled by the inverse of the partition function.
\end{lemma}

\begin{proof}
We can rewrite the consensus point as an expectation with respect to a new, exponentially weighted probability measure. Define the measure $\tilde{\mu}$ on $H$ such that
\begin{equation*}
    \dd \tilde{\mu}(x) = \frac{\exp(-\alpha \, \mathscr{E}(x))}{Z(\mu)} \,\mu(\dd x).
\end{equation*}
By definition, $\int_V \dd\tilde{\mu}(x) = 1$, hence $\tilde{\mu}$ is a probability measure. The consensus point can thus be written as $\mathfrak{m}_\alpha(\mu) = \int_V x \, \tilde{\mu}(\dd x)$.

Because $F(x) = \|x\|_H^2$ is a convex function on the Hilbert space $H$, we can apply Jensen's inequality for Bochner integrals
\begin{equation*}
    \|\mathfrak{m}_\alpha(\mu)\|_H^2 = \left\| \int_V x \, \tilde{\mu}(\dd x) \right\|_H^2 \leq \int_V \|x\|_H^2 \, \tilde{\mu}(\dd x).
\end{equation*}

Substituting the definition of $\tilde{\mu}$ back into the integral, we obtain
\begin{equation*}
    \int_V \|x\|_H^2 \, \tilde{\mu}(\dd x) = \int_V \|x\|_H^2 \,\frac{\exp(-\alpha\, \mathscr{E}(x))}{Z(\mu)} \mu(\dd x).
\end{equation*}

By Assumption \ref{assum:objective}, the objective function is non-negative ($\mathscr{E}(x) \geq 0$), which implies $\exp(-\alpha \,\mathscr{E}(x)) \leq 1$ for all $x \in H$. Applying this upper bound to the numerator of the integrand yields
\begin{equation*}
    \int_V \|x\|_H^2 \frac{\exp(-\alpha\, \mathscr{E}(x))}{Z(\mu)} \mu(\dd x) \leq \frac{1}{Z(\mu)} \int_V \|x\|_H^2\, \mu(\dd x)
\end{equation*}
which completes the proof.
\end{proof}

\subsection{A Priori Moment Estimates}
\label{sec:moment_estimates}

Having established the local regularity of the consensus point $\mathfrak{m}_{\alpha}$, we now employ the measure-freezing method to decouple the McKean-Vlasov dynamics. By fixing the probability measure flow, we transition from a non-linear McKean-Vlasov equation to a standard It\^o SDE in the Hilbert space $H$.

Let $T > 0$ be a fixed time horizon. We define the space of continuous probability measure flows strictly supported on the active subspace $V$ with bounded second moments as $\mathscr{C}_T := C([0,T]; \mathscr{P}_2(V))$. 
We equip this space with the uniform Wasserstein metric 
\begin{equation*}
    \mathsf{d}_{\mathscr{C}_T}(\bm{\mu}, \bm{\nu}) = \sup_{0 \leq t \leq T} \mathcal{W}_2(\mu_t, \nu_t), \qquad \forall\, \bm{\mu}=(\mu_t)_{t\geq 0}, \bm{\nu} = (\nu_{t})_{t\geq 0} \in \mathscr{C}_T
\end{equation*}

Fix an arbitrary deterministic probability measure flow $\bm{\mu} = (\mu_t)_{t \in [0,T]} \in \mathscr{C}_T$. Substituting $\mu_t$ into \eqref{eq:mckean_vlasov_cbo} yields the following decoupled SDE for a process $X^{\bm{\mu}}_t$
\begin{equation}
    \label{eq:decoupled_sde}
    \dd X^{\bm{\mu}}_t = -\lambda \big(X^{\bm{\mu}}_t - \mathfrak{m}_\alpha(\mu_t)\big)\,\dd t + \sigma\, \mathsf{D}\big(X^{\bm{\mu}}_t - \mathfrak{m}_\alpha(\mu_t)\big)\, \dd W_t^Q, \quad X^{\bm{\mu}}_0 = \xi.
\end{equation}
Because $\mu_t$ is now a fixed parameter, the mapping $t \mapsto \mathfrak{m}_\alpha(\mu_t)$ is a deterministic function of time. Furthermore, since $\bm{\mu} \in \mathscr{C}_T$, the second moment is uniformly bounded on $[0,T]$, which implies that $\sup_{t \in [0,T]} \|\mathfrak{m}_\alpha(\mu_t)\|_H \leq M_\mu < \infty$. Existence and uniqueness of a strong solution to \eqref{eq:decoupled_sde} are guaranteed by Theorem \ref{thm:solution} in Appendix \ref{app:sde in H}.
\begin{lemma}[Strict Invariance of the Decoupled Trajectory] \label{lem:subspace_invariance}
Let the initial condition $\xi \in V$ almost surely. For any fixed measure flow $\bm{\mu} \in \mathscr{C}_{T}$, the consensus point satisfies $\mathfrak{m}_\alpha(\mu_t) \in V$. Consequently, the unique strong solution $X^{\bm{\mu}}_t$ to the decoupled SDE \eqref{eq:decoupled_sde} remains strictly inside $V$ for all $t \geq 0$ almost surely.
\end{lemma}
\begin{proof}
Because the law $\mu_t$ is entirely supported on the closed linear subspace $V$, any integral with respect to this measure must also reside in $V$. Therefore, the consensus point lies strictly in the active subspace
\begin{equation*}
    \mathfrak{m}_\alpha(\mu_t) = \frac{\int_V x \exp(-\alpha\, \mathscr{E}(x)) \,\dd\mu_t(x)}{\int_V \exp(-\alpha \, \mathscr{E}(x))\, \dd\mu_t(x)} \in V.
\end{equation*}
Because $\mathfrak{m}_\alpha(\mu_t) \in V$, its orthogonal projection vanishes, that is $\mathsf{P}_{V^\perp} \mathfrak{m}_\alpha(\mu_t) = 0$.

To isolate the behavior outside the active subspace, we apply the continuous linear operator $\mathsf{P}_{V^\perp}$ to both sides of \eqref{eq:decoupled_sde}. Because $\mathsf{P}_{V^\perp}$ is a bounded linear operator, it commutes with the stochastic differential. 

First, we evaluate the diffusion term. Because $\mathsf{P}_V$ maps entirely into $V$, and $\mathsf{P}_{V^\perp}$ annihilates anything in $V$, their composition is identically zero, that is
\begin{equation*}
    \mathsf{P}_{V^\perp} \Big( \mathsf{D}\big(X^{\bm{\mu}}_t - \mathfrak{m}_\alpha(\mu_t)\big) \dd W_t^Q \Big) =\|X^{\bm{\mu}}_t - \mathfrak{m}_\alpha(\mu_t)\| (\mathsf{P}_{V^\perp} \mathsf{P}_V) \dd W_t^Q = 0.
\end{equation*}
Hence, the stochastic noise in the orthogonal complement is null. 

Next, we evaluate the drift term. Using the linearity of $\mathsf{P}_{V^\perp}$ and the fact that $\mathsf{P}_{V^\perp} \mathfrak{m}_\alpha(\mu_t) = 0$, we obtain
\begin{equation*}
    \dd \mathsf{P}_{V^\perp}X^{\bm{\mu}}_t = -\lambda \Big(\mathsf{P}_{V^\perp} X^{\bm{\mu}}_t - \mathsf{P}_{V^\perp} \mathfrak{m}_\alpha(\mu_t\Big)\dd t = -\lambda \mathsf{P}_{V^\perp} X^{\bm{\mu}}_t \,\dd t.
\end{equation*}

Let $Y_t = \mathsf{P}_{V^\perp} X^{\bm{\mu}}_t$. The SDE for the orthogonal component has reduced to
\begin{equation*}
    \dd Y_t = -\lambda Y_t \,\dd t
\end{equation*}
whose unique solution is $Y_t = Y_0 e^{-\lambda t}$. 
Because the initial state $X_0 \in V$ almost surely, its orthogonal component is null, that is $Y_0 = \mathsf{P}_{V^\perp} X^{\bm{\mu}}_0 = 0$. Consequently, $Y_t = 0$ for all $t \in[0,T]$, which implies $\mathsf{P}_{V^\perp}X^{\bm{\mu}}_t = 0$. Therefore, $X^{\bm{\mu}}_t\in V$ almost surely.
\end{proof}

In the following lemma, we verify existence of its second moment.

\begin{lemma}[Uniform Second Moment Estimate]
\label{lem:moment_estimate}
Suppose Assumption \ref{assum:objective} holds, and let us consider the setting in Proposition \ref{prop:diffusion} such  that $\lambda > 2\sigma^2 L_{\mathsf{D}}^2$. Let $\bm{\mu} \in \mathscr{C}_T$ be a fixed probability measure flow, and let $X^{\bm{\mu}}_t$ be the strong solution to the decoupled SDE \eqref{eq:decoupled_sde} with an initial condition satisfying $\mathbb{E}[\,\|\xi\|_H^2\,] < \infty$. Then, there exists a constant $C > 0$, dependent only on $\lambda, \sigma, L_{\mathsf{D}}$, and the flow $\bm{\mu}$, such that
\begin{equation*}
    \sup_{0\leq t \leq T} \,\mathbb{E}\big[\,\|X^{\bm{\mu}}_t\|_H^2\,\big] \leq C.
\end{equation*}
\end{lemma}

\begin{proof}
Because the input measure flow $\bm{\mu}$ belongs to $\mathscr{C}_T = C([0,T]; \mathscr{P}_2(V))$, its second moment is continuous and uniformly bounded over the time horizon $[0,T]$. Let us denote this bound as 
\begin{equation}\label{K_mu}
    \sup_{s \in [0,T]} \int_V \|x\|_H^2\, \mu_s(\dd x) \; =:\; K_\mu < \infty.
\end{equation}
Similarly, because $\mathscr{E}(x) \geq 0$ and the mapping $s \mapsto \mu_s$ is continuous, the denominator of the consensus point $Z(\mu_s) = \int_{V} w(x)\,\mu_{s}(\dd x)$ is bounded away from zero; see \eqref{eq:lowerbound Z}. By the Extreme Value Theorem, there exists a constant $Z_{\min} > 0$ such that
\begin{equation}\label{z_min}
    \inf_{s \in [0,T]} Z(\mu_s) = Z_{\min}.
\end{equation}

Applying the infinite-dimensional It\^o formula (see Theorem \ref{thm:Ito in H} and Example \ref{ex:Ito in H}) to $\|X^{\bm{\mu}}_t\|_H^2$ yields the following integral equation 
\begin{equation}
\label{eq:ito_norm_squared}
\begin{aligned}
    \|X^{\bm{\mu}}_t\|_H^2 &= \|\xi\|_H^2 + \int_0^t 2\, \langle X^{\bm{\mu}}_s, -\lambda(X^{\bm{\mu}}_s - \mathfrak{m}_\alpha(\mu_s)) \rangle_H \,\dd s \\
    &\qquad \qquad + \int_0^t 2 \, \langle X^{\bm{\mu}}_s, \sigma \mathsf{D}(X^{\bm{\mu}}_s - \mathfrak{m}_\alpha(\mu_s)) \dd W_s^{Q} \rangle_H\\
    &\qquad \qquad \qquad \qquad + \int_0^t \sigma^2 \|\mathsf{D}(X^{\bm{\mu}}_s - \mathfrak{m}_\alpha(\mu_s))\|_{\mathscr{L}_Q}^2 \dd s.
\end{aligned}
\end{equation}

We bound the drift using the Cauchy-Schwarz and Young's inequalities, and the trace term using the Lipschitz continuity of $\mathsf{D}$ from Proposition \ref{prop:diffusion}. Thus we have
\[
    2 \langle X^{\bm{\mu}}_s, -\lambda(X^{\bm{\mu}}_s - \mathfrak{m}_\alpha(\mu_s)) \rangle_H \leq -\lambda \|X^{\bm{\mu}}_s\|_H^2 + \lambda \|\mathfrak{m}_\alpha(\mu_s)\|_H^2, 
\]
and
\[
    \sigma^2 \|\mathsf{D}(X^{\bm{\mu}}_s - \mathfrak{m}_\alpha(\mu_s))\|_{\mathscr{L}_Q}^2 \leq 2\sigma^2 L_{\mathsf{D}}^2 \big(\|X^{\bm{\mu}}_s\|_H^2 + \|\mathfrak{m}_\alpha(\mu_s)\|_H^2\big).
\]

Taking the expectation of \eqref{eq:ito_norm_squared}, the stochastic integral vanishes because its integrand is a true martingale. Substituting our bounds into the expectation, we obtain
\begin{equation}
    \label{eq:expectation_bound_pre_lemma}
    \mathbb{E}\big[\,\|X^{\bm{\mu}}_t\|_H^2\,\big] \leq \mathbb{E}[\,\|\xi\|_H^2\,] + \int_0^t  C_1\, \mathbb{E}\big[\|X^{\bm{\mu}}_s\|_H^2\big] + C_2\, \|\mathfrak{m}_\alpha(\mu_s)\|_H^2  \;\dd s,
\end{equation}
where 
\begin{equation*}
    C_1 = -\lambda + 2\sigma^2 L_{\mathsf{D}}^2 \qquad \text{ and } \qquad C_2 = \lambda + 2\sigma^2 L_{\mathsf{D}}^2.
\end{equation*}

Now, we invoke Lemma \ref{lem:m_alpha_moment} to bound the consensus point directly by the second moment of the input measure $\mu_s$
\begin{equation*}
    \|\mathfrak{m}_\alpha(\mu_s)\|_H^2 \leq \frac{1}{Z(\mu_s)} \int_V \|x\|_H^2 \,\mu_s(\dd x) \leq \frac{K_\mu}{Z_{\min}}.
\end{equation*}
Substituting this bound back into \eqref{eq:expectation_bound_pre_lemma} yields
\begin{equation*}
    \mathbb{E}\big[\,\|X^{\bm{\mu}}_t\|_H^2\,\big] \leq \mathbb{E}[\,\|\xi\|_H^2\,] + C_2 \,\frac{K_\mu}{Z_{\min}} \,T + \int_0^t C_1 \, \mathbb{E}\big[\|X^{\bm{\mu}}_s\|_H^2\big] \dd s.
\end{equation*}
Applying Gr\"onwall's Lemma, we find
\begin{equation}\label{eq:second moment apriori}
    \mathbb{E}\big[\,\|X^{\bm{\mu}}_t\|_H^2\,\big] \leq \left(\mathbb{E}[\,\|\xi\|_H^2\,] + C_2\, \frac{K_\mu}{Z_{\min}}\, T\right) \exp(C_1\, T).
\end{equation}
Moreover, when $\lambda > 2\sigma^2 L_{\mathsf{D}}^2$, one has $C_{1}<0$. Hence, the right-hand side, as a function of $T\geq0$, is bounded\footnote{Given $a,b,c>0$ constants, consider $f(t)=(a+b\,t)e^{-c\,t}$. Its maximum is attained at $t^* = \frac{b-ca}{bc}$. Two cases arise. If $b\leq ca$, then $f$ is decreasing on $[0,\infty)$, so we have $f(t)\leq f(0) =a$ for all $t\geq0$. If $b>ca$, then $f(t)\leq f(t^*)\leq b/c$. Conclusion: $f(t)\leq a+b/c$ for all $t\geq 0$.} by
\[
    K:= \mathbb{E}[\,\|\xi\|_H^2\,] - \frac{C_2}{C_1}\, \frac{K_\mu}{Z_{\min}} \qquad \text{ where } K_\mu \text{ is } \eqref{K_mu} \text{ and } Z_{\min} \text{ is } \eqref{z_min}
\]
and one gets
\begin{equation}\label{eq:second moment apriori 2}
    \mathbb{E}\big[\,\|X^{\bm{\mu}}_t\|_H^2\,\big] \leq K \qquad \forall\,t\geq 0.
\end{equation} 
\end{proof}

\begin{lemma}[Time Continuity in the Wasserstein Space]
\label{lem:time_continuity}
Under the assumptions of Lemma \ref{lem:moment_estimate}, the probability law of the decoupled process, denoted by $\nu_t = \mathrm{Law}(X^{\bm{\mu}}_t)$, is continuous in time with respect to the 2-Wasserstein metric. Therefore, the probability measure flow belongs to the space of continuous flows: $\bm{\nu}=(\nu_t)_{t\geq0} \in C([0,T]; \mathscr{P}_2(V))$. 
In particular, the map 
\begin{equation*}
\begin{aligned}
    \Phi: \mathscr{C}_T & \longrightarrow \mathscr{C}_T,\quad \bm{\mu}  \longmapsto \Phi(\bm{\mu})=\big(\mathrm{Law}(X^{\bm{\mu}}_t)\big)_{t\geq 0}
\end{aligned}
\end{equation*}
is well-defined.
\end{lemma}

\begin{proof}
Fix $0 \leq s < t \leq T$. We can bound the Wasserstein distance between the probability laws at time $s$ and time $t$ by the expected squared distance of the process itself, since the joint distribution of $(X^{\bm{\mu}}_s, X^{\bm{\mu}}_t)$ is an admissible coupling
\begin{equation*}
    \mathcal{W}_2^2(\nu_s, \nu_t) \leq \mathbb{E}\big[\,\|X^\mu_t - X^\mu_s\|_H^2\,\big].
\end{equation*}

Using the integral form of the decoupled SDE \eqref{eq:decoupled_sde} and the elementary inequality $(a+b)^2 \leq 2a^2 + 2b^2$, we obtain 
\begin{align*}
    \mathbb{E}\big[\,\|X^{\bm{\mu}}_t - X^{\bm{\mu}}_s\|_H^2\,\big] &\leq 2\, \mathbb{E}\left[ \left\| \int_s^t -\lambda\big(\,X^{\bm{\mu}}_r - \mathfrak{m}_\alpha(\mu_r)\,\big)\, \dd r \right\|_H^2 \right] \nonumber \\
    &\qquad \qquad \qquad + 2\, \mathbb{E}\left[ \,\left\| \int_s^t \sigma \,\mathsf{D}\big(X^{\bm{\mu}}_r - \mathfrak{m}_\alpha(\mu_r)\big)\, \dd W_r^{Q} \right\|_H^2\, \right].
\end{align*}

For the drift term, we apply the Cauchy-Schwarz inequality to the time integral. For the stochastic integral, we apply the infinite-dimensional It\^o isometry (see \eqref{eq: ito isom}) and Proposition \ref{prop:diffusion}. thus we obtain
\begin{align*}
    \mathbb{E}\big[\,\|X^{\bm{\mu}}_t - X^{\bm{\mu}}_s\|_H^2\,\big] &\leq 2\,\lambda^2 (t-s) \int_s^t \mathbb{E}\big[\,\|X^{\bm{\mu}}_r - \mathfrak{m}_\alpha(\mu_r)\|_H^2\,\big] \,\dd r  \\
    &\qquad \qquad \qquad + 2\,\sigma^2 \int_s^t \mathbb{E}\big[\,\|\mathsf{D}\big(X^{\bm{\mu}}_r - \mathfrak{m}_\alpha(\mu_r)\big)\|_{\mathscr{L}_Q}^2\,\big]\, \dd r  \\
    &\leq 2\int_s^t \big( \,\lambda^2(t-s) + \sigma^2 L_{\mathsf{D}}^2 \,\big)\, \mathbb{E}\big[\,\|X^{\bm{\mu}}_r - \mathfrak{m}_\alpha(\mu_r)\|_H^2\,\big] \, \dd r.
\end{align*}

From the uniform bounds established in Lemma \ref{lem:moment_estimate} and the bound on the consensus point from Lemma \ref{lem:m_alpha_moment}, we know there exists a global constant $M > 0$ such that $\mathbb{E}[\,\|X^{\bm{\mu}}_r - \mathfrak{m}_\alpha(\mu_r)\|_H^2\,] \leq M$ for all $r \in [0,T]$. 

Substituting this uniform bound into the integral yields
\begin{equation*}
\begin{aligned}
    \mathcal{W}_2^2(\nu_s, \nu_t) 
    & \;\leq\; \mathbb{E}\big[\|X^{\bm{\mu}}_t - X^{\bm{\mu}}_s\|_H^2\big]\; \leq\; 2\,M \big( \lambda^2 T + \sigma^2 L_{\mathsf{D}}^2 \big) |t-s|. 
\end{aligned}
\end{equation*}
Therefore, the map $t \mapsto \nu_t=\mathrm{Law}(X^{\bm{\mu}}_t)$ is continuous in the $\mathcal{W}_2$ metric. Combined with the uniform second moment bounds, this ensures the map $\bm{\mu} \mapsto \Phi({\bm{\mu}}) = \big(\mathrm{Law}(X^{\bm{\mu}}_t)\big)_{t\geq 0}$ is well-defined and maps into $\mathscr{C}_T$. 
\end{proof}

\subsection{Fixed-Point Argument and the Proof of Theorem \ref{thm:well_posedness}}
\label{sec:contraction_proof}

With the a priori moment estimates secured, we are now in a position to prove Theorem \ref{thm:well_posedness}. We will define a solution map on the Wasserstein space of probability measure flows and prove that it is a strict contraction on a sufficiently small time interval.

Let $T_0 \in (0, T]$ be a small time horizon to be determined. We consider the complete metric space $\mathscr{C}_{T_0} = C([0,T_0]; \mathscr{P}_2(V))$ equipped with the uniform Wasserstein metric 
\begin{equation}
    \mathsf{d}_{\mathscr{C}_{T_0}}(\bm{\mu}, \bm{\nu}) = \sup_{0 \leq t \leq T_0} \mathcal{W}_2(\mu_t, \nu_t), \qquad \forall\, \bm{\mu}=(\mu_t)_{t\geq 0}, \;\bm{\nu} = (\nu_{t})_{t\geq 0} \in \mathscr{C}_{T_{0}}
\end{equation}

We define the mapping $\Phi: \mathscr{C}_{T_0} \longrightarrow \mathscr{C}_{T_0}$ such that $\Phi(\bm{\mu})$ is the flow of the probability laws for the process $X^\mu$ solving the decoupled SDE \eqref{eq:decoupled_sde}, that is,
\begin{equation*}
    \Phi({\bm{\mu}})_t = \mathrm{Law}(X^{\bm{\mu}}_t), \quad \text{for all } t \in [0,T_0].
\end{equation*}
By Lemma \ref{lem:moment_estimate} and Lemma \ref{lem:time_continuity}, the decoupled process $X^{\bm{\mu}}_t$ possesses uniformly bounded second moments and its law is continuous in time. This ensures the mapping $\Phi$ is well-defined, satisfying $\Phi({\bm{\mu}}) \in \mathscr{C}_{T_0}$. A fixed point of this mapping, i.e., a probability measure flow ${\bm{\mu}}^*$ such that $\Phi({\bm{\mu}}^*) = {\bm{\mu}}^*$, corresponds exactly to the strong solution of the original McKean-Vlasov SDE \eqref{eq:mckean_vlasov_cbo}.

Let $M_0 = \mathbb{E}[\,\|\xi\|_H^2\,] < \infty$ be the initial second moment. We fix an arbitrary radius $R > 0$ such that $R^2 > M_0$. Let $T_0 \in (0, T]$ be a small time horizon to be determined. 

We define $\mathscr{C}_{T_0, R}$ as the space of continuous probability measure flows whose second moments are uniformly bounded by $R^2$ 
\begin{equation}\label{invariant ball}
    \mathscr{C}_{T_0, R} = \left\{ {\bm{\mu}} \in C([0,T_0]; \mathscr{P}_2(V))\; : \; \sup_{0 \leq t \leq T_0} \int_V \|x\|_H^2 \,\mu_t(\dd x) \leq R^2 \right\}.
\end{equation}
Equipped with the uniform Wasserstein metric $\mathsf{d}_{\mathscr{C}_{T_0}}({\bm{\mu}}, {\bm{\nu}}) = \sup_{0 \leq t \leq T_0} \mathcal{W}_2(\mu_t, \nu_t)$, the space $\mathscr{C}_{T_0, R}$ is a complete metric space. We define the measure-mapping operator $\Phi({\bm{\mu}})_t = \mathrm{Law}(X^{\bm{\mu}}_t)$, where $X^{\bm{\mu}}_t$ is the solution to the decoupled SDE \eqref{eq:decoupled_sde}.

\begin{lemma}[Invariance of $\mathscr{C}_{T_0, R}$]
\label{lem:invariance}
In the situation of Lemma \ref{lem:moment_estimate}, 
there exists a time horizon $T_1 > 0$ such that for all $T_0 \leq T_1$, the mapping $\Phi$ maps $\mathscr{C}_{T_0, R}$ into itself.
\end{lemma}

\begin{proof}
Let ${\bm{\mu}} \in \mathscr{C}_{T_0, R}$. By definition, $\mu_t \in \mathscr{P}_{2,R}(H)$ for all $t \in [0, T_0]$. Following the same approach as for the a priori moment estimate in Lemma \ref{lem:moment_estimate}, one obtains a bound similar to \eqref{eq:second moment apriori}, that is
\begin{equation*}
    \mathbb{E}\big[\|X^{\bm{\mu}}_t\|_H^2\big] \leq \left( M_0 + C_2 \frac{R^2}{Z_{\min}(R)} \, t \right) \exp(C_1 \,t), \quad \text{for all } t \in [0, T_0],
\end{equation*}
where $C_1$ and $C_2$ are constants depending on $\lambda, \sigma$, and $L_{\mathsf{D}}$, but independent of $T_0$. 
As $t \to 0$, the right-hand side converges  to $M_0$. Since we chose $R^2 > M_0$, there exists a sufficiently small time horizon $T_1 > 0$ such that for all $t \leq T_1$
\begin{equation*}
    \left( M_0 + C_2 \frac{R^2}{Z_{\min}(R)}\, t \right) \exp(C_1 \,t) \leq R^2.
\end{equation*}
Therefore, if we choose any $T_0 \leq T_1$, we guarantee $\sup_{t \in [0, T_0]} \mathbb{E}[\|X^{\bm{\mu}}_t\|_H^2] \leq R^2$, meaning $\Phi({\bm{\mu}}) \in \mathscr{C}_{T_0, R}$.
\end{proof}

We are now ready to prove contraction using synchronous coupling. 

\begin{lemma}[Strict Contraction of $\Phi$]
\label{lem:contraction}
Suppose $\lambda > 2\sigma^2\, L_{\mathsf{D}}^2$.
There exists a time horizon $T_0 \in (0, T_1]$ such that the map $\Phi$ is a strict contraction on $\mathscr{C}_{T_0, R}$.
\end{lemma}

\begin{proof}
Let ${\bm{\mu}}, {\bm{\nu}} \in \mathscr{C}_{T_0, R}$. Because both flows remain inside the invariant ball \eqref{invariant ball}, $\mu_t, \nu_t \in \mathscr{P}_{2,R}(V)$ for all $t \in [0, T_0]$. Consequently, we can apply the localized Lipschitz bound from Lemma \ref{lem:local_lipschitz_m} and obtain 
\begin{equation*}
    \|\mathfrak{m}_\alpha(\mu_t) - \mathfrak{m}_\alpha(\nu_t)\|_H \leq L_\mathfrak{m}(R) \, \mathcal{W}_2(\mu_t, \nu_t).
\end{equation*}

Next, we couple the processes $X^{\bm{\mu}}_t$ and $X^{\bm{\nu}}_t$ synchronously, driving both with the identical initial condition $\xi$ and identical $Q$-Wiener process $W_t^Q$. The Wasserstein distance is bounded by this coupling
\begin{equation*}
    \mathcal{W}_2^2(\Phi(\mu)_t, \Phi(\nu)_t) \leq \mathbb{E}\big[\|X^{\bm{\mu}}_t - X^{\bm{\nu}}_t\|_H^2\big].
\end{equation*}

Applying the infinite-dimensional It\^o formula (Theorem \ref{thm:Ito in H}) to $\|X^{\bm{\mu}}_t - X^{\bm{\nu}}_t\|_H^2$, taking the expectation to eliminate the martingale term, and applying standard Cauchy-Schwarz and Lipschitz bounds (as derived in Section \ref{sec:moment_estimates}), we obtain 
\begin{equation*}
    \mathbb{E}\big[\,\|X^{\bm{\mu}}_t - X^{\bm{\nu}}_t\|_H^2\,\big] \leq \int_0^t \tilde{C}_1 \,\mathbb{E}\big[\,\|X^{\bm{\mu}}_s - X^{\bm{\nu}}_s\|_H^2\,\big]\, \dd s + \int_0^t \tilde{C}_2\, \|\mathfrak{m}_\alpha(\mu_s) - \mathfrak{m}_\alpha(\nu_s)\|_H^2\, \dd s,
\end{equation*}
where 
\[
    \tilde{C}_1 = -\lambda + 2\sigma^2\, L_{\mathsf{D}}^2 \qquad \text{ and } \qquad \tilde{C}_2 = \lambda + 2\sigma^2\, L_{\mathsf{D}}^2.
\]

Substituting the local Lipschitz bound $L_\mathfrak{m}(R)$, we get 
\begin{equation*}
    \mathbb{E}\big[\|X^{\bm{\mu}}_t - X^{\bm{\nu}}_t\|_H^2\big] \leq \int_0^t \tilde{C}_1 \mathbb{E}\big[\|X^{\bm{\mu}}_s - X^{\bm{\nu}}_s\|_H^2\big] \dd s + \int_0^t \tilde{C}_2 L_m(R)^2 \mathcal{W}_2^2(\mu_s, \nu_s)\, \dd s.
\end{equation*}

By Gr\"onwall's Lemma, this yields 
\begin{equation*}
    \mathbb{E}\big[\|X^{\bm{\mu}}_t - X^{\bm{\nu}}_t\|_H^2\big] \leq \tilde{C}_2 L_m(R)^2 \exp(\tilde{C}_1 \,t) \int_0^t \mathcal{W}_2^2(\mu_s, \nu_s)\, \dd s.
\end{equation*}

Taking the supremum over $[0, T_0]$ gives 
\begin{equation*}
    \sup_{0 \leq t \leq T_0} \mathcal{W}_2^2(\Phi({\bm{\mu}})_t, \Phi({\bm{\nu}})_t) \leq \big[ \tilde{C}_2 L_m(R)^2 T_0 \exp(\tilde{C}_1\, T_0) \big] \sup_{0 \leq t \leq T_0} \mathcal{W}_2^2(\mu_t, \nu_t).
\end{equation*}

Because $\tilde{C}_1, \tilde{C}_2$, and $L_m(R)$ are fixed constants independent of $T_0$, we can choose $T_0$ such that the prefactor satisfies 
\begin{equation*}
    \tilde{C}_2 L_m(R)^2 T_0 \exp(\tilde{C}_1 T_0) < 1.
\end{equation*}
For this choice of $T_0$, the mapping $\Phi$ is a strict contraction on the complete metric space $\mathscr{C}_{T_0, R}$. By Banach's Fixed-Point Theorem, there exists a unique probability measure flow ${\bm{\mu}}^* \in \mathscr{C}_{T_0, R}$ such that $\Phi({\bm{\mu}}^*) = {\bm{\mu}}^*$.
\end{proof}

\subsection*{Conclusion of the Proof of Theorem \ref{thm:well_posedness}}

To recover the strong solution to the original McKean-Vlasov dynamics, we substitute this unique fixed point back into the decoupled equation
\begin{equation*}
    \dd X_t^* = -\lambda \big(X_t^* - \mathfrak{m}_\alpha(\mu_t^*)\big)\dd t + \sigma \|X_t^* - \mathfrak{m}_\alpha(\mu_t^*)\|_H \mathsf{P}_V \,\dd W_t^Q, \quad X_0^* = \xi.
\end{equation*}
Because the mapping $t \mapsto \mathfrak{m}_\alpha(\mu_t^*)$ is now a fixed, deterministic, and continuous function of time, this reduces to a standard stochastic differential equation with globally Lipschitz coefficients. By Theorem \ref{thm:solution}, this SDE admits a unique strong solution $X^* \in L^2\big(\Omega; C([0,T_0]; H)\big)$. By construction, the law of $X^*$ is exactly $\bm{\mu}^*$, thereby satisfying the non-linear McKean-Vlasov equation \eqref{eq:mckean_vlasov_cbo_P}. Pathwise uniqueness is guaranteed: any other strong solution $Y$ must possess a law that is a fixed point of $\Phi$. By the uniqueness of $\bm{\mu}^*$, $Y$ must solve the identical decoupled SDE, and thus $Y = X^*$ almost surely.

The local existence and uniqueness on $[0, T_0]$ can be extended to the arbitrary time horizon $[0, T]$. Note that the choice of $T_0$ in Lemma \ref{lem:contraction} depends on $R$, which in turn depends on the second moment of the initial condition. By the uniform moment estimate in Lemma \ref{lem:moment_estimate}, the second moment of the solution $\mathbb{E}[\|X_t\|^2]$ remains bounded on any finite interval. Thus, one can iterate the fixed-point argument on $[T_0, 2T_0]$, $[2T_0, 3T_0]$, and so on. Since the second moment cannot blow up in finite time, the sequence of time intervals will cover $[0, T]$ in a finite number of steps, establishing global well-posedness.

\section{Convergence Guarantees}
\label{sec:convergence}

We analyze the standard mean-field CBO dynamics \eqref{eq:mckean_vlasov_cbo_P} in a separable Hilbert space $H$, which is given by the McKean-Vlasov SDE 
\begin{equation} \label{eq:mckean_vlasov_cbo_degenerate}
    \dd X_t = -\lambda \big(X_t - \mathfrak{m}_\alpha(\mu_t)\big)\,\dd t + \sigma \,\mathsf{D}\big(X_t - \mathfrak{m}_{\alpha}(\mu_t)\big)\, \dd W_t^Q, \quad X_0 = \xi,
\end{equation}
where $\mathsf{D}$ given by \eqref{def:diffusion}, is a Lipschitz continuous multiplicative diffusion operator with Lipschitz constant $L_{\mathsf{D}}$, and $W_t^Q$ is a $Q$-Wiener process with a strictly positive, trace-class covariance operator $Q$ (see \S\ref{subsec:q_wiener} and Appendix \ref{app:noise}).
 
To establish global convergence to the unique minimizer $x^*$ in the infinite-dimensional setting, we must carefully decouple the temperature limit ($\alpha \to \infty$) from the time evolution to avoid the exponential mass decay associated with the small ball problem in $H$. We achieve this through a sequence of three interconnected steps 
\begin{enumerate}[label = \arabic*.]
    \item \textbf{Strict Positivity of Mass:} We adapt a smooth mollifier technique to the infinite-dimensional It\^o calculus. By carefully bounding the trace of the composition operator, we prove the swarm maintains a strictly positive probability mass in any neighborhood of $x^*$.
    \item \textbf{Quantitative Laplace Contraction:} With the strictly positive mass, we can apply Jensen's inequality to demonstrate that for a sufficiently large $\alpha$, the consensus point $\mathfrak{m}_{\alpha}(\mu_t)$ is drawn strictly closer to the minimizer than the average variance of the swarm. 
    \item \textbf{Finite-Time Exponential Convergence:} Finally, we analyze the time evolution of the swarm's expected squared distance. By fixing a target accuracy $\varepsilon$ and evaluating the dynamics over a finite time horizon $T_\varepsilon$, we control the lower-bound of the mass and demonstrate that the error decays exponentially fast before reaching the $\varepsilon$-neighborhood.
\end{enumerate}

\begin{lemma}[Active-Subspace Bounds for Degenerate Diffusion] \label{lem:degenerate_covariance_bounds}
Suppose the finite-dimensional active subspace $V \subset H$ and the trace-class covariance operator $Q$ satisfy Assumption \ref{assump:active_subspace}. 
Furthermore, assume the state-dependent multiplicative diffusion operator $\mathsf{D} : H \to \mathscr{L}(H)$ satisfies the regularity and structural conditions obtained in Proposition \ref{prop:diffusion}. 

Then, the degenerate diffusion covariance operator $\Sigma(x) = \sigma^2 \mathsf{D}(x - \mathfrak{m}_{\alpha}) Q \mathsf{D}(x - \mathfrak{m}_{\alpha})^*$ satisfies the following active-subspace bounds for all $x, \mathfrak{m}_{\alpha} \in H$ and $v\in V$ 
\begin{align}
    \langle v, \Sigma(x) v \rangle_H &\geq c_1 \|x - \mathfrak{m}_{\alpha}\|_H^2 \|v\|_H^2, 
    \label{eq:sigma_lower_deg} \\
    \mathrm{Tr}(\Sigma(x)) &\leq C_2 \|x - \mathfrak{m}_{\alpha}\|_H^2, 
    \label{eq:sigma_trace_deg}
\end{align}
where $c_1 = \sigma^2 \lambda_{\min}  > 0$, $\lambda_{\min}$ being as defined in statement (iii) of Assumption \ref{assump:active_subspace}, and $C_2 =\sigma^2 \mathrm{Tr}(\mathsf{P}_V Q \mathsf{P}_V) > 0$.
\end{lemma}
\begin{remark}
    It should be noted that incorporating $\mathsf{P}_V$ in $\mathsf{D}$ acts as a filter against the infinite-dimensional tail $V^\perp$. Consequently, the bounding constants are those obtained at the finite-dimensional setting: the exploration lower bound is driven by the minimum active eigenvalue ($c_1 = \sigma^2 \lambda_{\min}$), while the  upper bound is driven by the noise energy isolated to the active subspace ($C_2 = \sigma^2 \mathrm{Tr}(\mathsf{P}_V Q \mathsf{P}_V)$).
\end{remark}

\begin{proof}
We prove each inequality separately. 

\textbf{Step 1.} \textit{(The lower bound \eqref{eq:sigma_lower_deg})} 
Let $v \in H$ be an arbitrary vector in the active subspace. Expanding the quadratic form of the covariance operator and moving the left-most diffusion operator via its adjoint yields
\begin{equation} \label{eq:adj_move}
    \langle v, \Sigma(x) v \rangle_H = \sigma^2 \,\langle \mathsf{D}(x - \mathfrak{m}_{\alpha})^* v, \, Q \mathsf{D}(x - \mathfrak{m}_{\alpha})^* v \rangle_H.
\end{equation}
Let $w = \mathsf{D}(x - \mathfrak{m}_{\alpha})^* v=\mathsf{D}(x - \mathfrak{m}_{\alpha}) v$. By the definition of the operator $ \mathsf{D}$,  the transformed vector $w \in V$. Therefore, we may invoke the local strict positivity of $Q$ from Assumption \ref{assump:active_subspace}, yielding 
\begin{equation*}
    \langle w, Q w \rangle_H \geq \lambda_{\min}\,\|w\|_H^2.
\end{equation*}
Applying this to \eqref{eq:adj_move} gives 
\begin{equation*}
    \langle v, \Sigma(x) v \rangle_H \geq \sigma^2 \lambda_{\min} \|\mathsf{D}(x - \mathfrak{m}_{\alpha})^* v\|_H^2.
\end{equation*}
Finally, we apply the coercivity property of Proposition \ref{prop:diffusion}(3) to lower bound the energy injected into the active subspace by the distance to the consensus point
\begin{equation*}
    \langle v, \Sigma(x) v \rangle_H \geq \sigma^2 \lambda_{\min} \Big( \|x - \mathfrak{m}_{\alpha}\|_H^2 \, \|v\|_H^2 \Big).
\end{equation*}
Letting $c_1 = \sigma^2 \lambda_{\min}>0 $ completes the proof of the lower bound.

\textbf{Step 2.} \textit{(The upper bound \eqref{eq:sigma_trace_deg})} 
For any bounded linear operator $A$ and positive trace-class operator $Q$, standard functional analysis (see Appendix \ref{app:norms}) yields $\mathrm{Tr}(A Q A^*) \leq \|A\|_{\mathrm{op}}^2 \mathrm{Tr}(Q)$. Applying this to the diffusion covariance $\Sigma(x)$ with $A = \mathsf{D}(x - \mathfrak{m}_{\alpha})$ yields
\begin{equation*}
    \mathrm{Tr}(\Sigma(x)) \leq \sigma^2\, \|\mathsf{D}(x - \mathfrak{m}_{\alpha})\|_{\mathrm{op}}^2\, \mathrm{Tr}(Q).
\end{equation*}
Using the linear Growth and vanishing-at-zero property of Proposition \ref{prop:diffusion}(2), the operator norm is bounded by the distance
\begin{equation*}
    \|\mathsf{D}(x - \mathfrak{m}_{\alpha})\|_{\mathrm{op}}^2 \leq L_{\mathsf{D}}^2 \,\|x - \mathfrak{m}_{\alpha}\|_H^2. 
\end{equation*}
Substituting this into the trace inequality yields
\begin{equation}
    \mathrm{Tr}(\Sigma(x)) \leq \sigma^2 \Big( L_{\mathsf{D}}^2 \|x - \mathfrak{m}_{\alpha}\|_H^2 \Big) \,\mathrm{Tr}(Q).
\end{equation}
Because $Q$ is globally trace-class, $\mathrm{Tr}(Q) < \infty$. Letting the positive constant $C_2 = \sigma^2 L_{\mathsf{D}}^2 \mathrm{Tr}(Q)$  completes the proof of the trace upper bound. 
\end{proof}

We can now establish a positive probability mass in any neighborhood of $x^*$. To this end, we define $\phi_r: H \to [0,1]$, a smooth mollifier compactly supported on $B_r(x^*)$ such as
\begin{equation*}
    \phi_r(x) = \begin{cases} 
    \exp\left( 1 - \frac{r^2}{r^2 - \|x - x^*\|_H^2} \right), & \text{if } \|x - x^*\|_H < r \\ 
    0, & \text{otherwise.}
    \end{cases}
\end{equation*}
Let $y(x) = \|x - x^*\|_H^2$. Inside the ball, the Fr\'echet derivatives of $\phi_r(x)$ in $H$ are 
\begin{align*}
    D\phi_r(x) &= -\frac{2r^2}{(r^2 - y(x))^2} \phi_r(x) (x - x^*), \\
    D^2\phi_r(x) &= \phi_r(x) \left[ \frac{4r^2 (2y(x)-r^{2})}{(r^2 - y(x))^4} (x - x^*) \otimes (x - x^*) - \frac{2r^2}{(r^2 - y(x))^2} I \right]. 
\end{align*}
Here, $I$ is the identity operator in $H$. 

Applying the infinite-dimensional It\^o formula (Theorem \ref{thm:Ito in H}) to $\mathbb{E}[\phi_r(X_t)]$ over the interval $t \in [0,T]$ yields 
\begin{equation*} 
\begin{aligned}
    \frac{\dd}{\dd t} \mathbb{E}[\phi_r(X_t)] 
    & = \mathbb{E}\Big[  \langle D\phi_r(X_t), -\lambda(X_t - \mathfrak{m}_{\alpha}) \rangle_H +  \frac{\sigma^2}{2} \mathrm{Tr}\Big( \mathsf{D}(X_t - \mathfrak{m}_{\alpha}) Q \mathsf{D}(X_t - \mathfrak{m}_{\alpha})^* D^2\phi_r(X_t) \Big)  \Big]\\
    & = \mathbb{E}[\,T_1 \, + \, T_2\,]
\end{aligned}
\end{equation*}
where
\begin{align*}
    & T_1 := \langle D\phi_r(X_t), -\lambda(X_t - \mathfrak{m}_{\alpha}) \rangle_H,\\
    & T_2 := \frac{\sigma^2}{2} \,\mathrm{Tr}\Big( \mathsf{D}(X_t - \mathfrak{m}_{\alpha}) Q \mathsf{D}(X_t - \mathfrak{m}_{\alpha})^* D^2\phi_r(X_t) \Big).
\end{align*}

\begin{lemma}[Strict Positivity of Mass] \label{lem:positive_mass_degenerate}
Let the initial probability measure $\mu_0$ assign strictly positive mass to any neighborhood of the global minimizer $x^*$. For a fixed radius $r > 0$, let the local neighborhood be defined as 
\[
    B_r(x^*) = \{x \in H \;:\; \|x - x^*\|_H < r\}.
\]
Assume that for any finite time horizon $T > 0$, the consensus point remains uniformly bounded by some constant $B>0$ 
\[
    \sup\limits_{t \in [0,T]} \|x^* - \mathfrak{m}_{\alpha}(\mu_t)\|_H \leq B.
\]
Furthermore, assume the ambient space $H$, the trace-class covariance $Q$, and the state-dependent diffusion operator $\mathsf{D}$ satisfy the active-subspace and regularity conditions outlined in Assumption \ref{assump:active_subspace} and in Proposition \ref{prop:diffusion}. Then there exists a strictly positive constant $p > 0$ such that for all $t \in [0,T]$, the law $\mu_t$ of the CBO dynamics 
\[
    \dd X_t = -\lambda (X_t - \mathfrak{m}_{\alpha}(\mu_t))\,\dd t + \sigma\, \mathsf{D}(X_t - \mathfrak{m}_{\alpha}(\mu_t))\,\dd W_t^Q,
\] 
satisfies
\begin{equation*}
    \mu_t(B_r(x^*)) \geq \mathbb{E}[\phi_r(X_t)] \geq \mathbb{E}[\phi_r(X_0)] e^{-pt} > 0.
\end{equation*}
\end{lemma}

\begin{proof}

Let us introduce $v = X_t - x^*$,  $y = \|v\|_H^2$, and  $h = X_t - \mathfrak{m}_{\alpha}(\mu_t)$. Recall the smooth mollifier $\phi_r(x) = \exp\big(1 - \frac{r^2}{r^2 - y}\big)$, which is compactly supported on $B_r(x^*)$.  

Applying the infinite-dimensional It\^o formula (see Theorem \ref{thm:Ito in H}), the time evolution of the expected mollifier is $\frac{\dd}{\dd t} \mathbb{E}[\phi_r(X_t)] = \mathbb{E}[T_1 + T_2]$. Factoring out $\phi_r(X_t)$, we evaluate the sum of the drift ($T_1$) and diffusion ($T_2$) operators strictly inside the ball ($y < r^2$), that is
\begin{equation} \label{eq:full_operator_deg}
    \frac{T_1 + T_2}{\phi_r(X_t)} = \frac{2\lambda r^2}{(r^2-y)^2} \langle v, h \rangle_H + \frac{2r^2(2y - r^2)}{(r^2-y)^4} \langle v, \Sigma(X_t) v \rangle_H - \frac{r^2}{(r^2-y)^2} \mathrm{Tr}(\Sigma(X_t)).
\end{equation}

Because the initial measure and the projected dynamics strictly preserve the active subspace $V$ (as established in Lemma \ref{lem:subspace_invariance}), we are mathematically guaranteed that $X_t, x^*, \mathfrak{m}_\alpha \in V$, and thus the direction vector $v \in V$. We may therefore invoke Lemma \ref{lem:degenerate_covariance_bounds} to bound the diffusion terms using the constants $c_1$ and $C_2$ therein defined. Applying these alongside the Cauchy-Schwarz bound for the drift $\langle v, h \rangle_H \geq -\sqrt{y}\|h\|_H$, we obtain the following lower-bound
\begin{equation*}
    \frac{T_1 + T_2}{\phi_r(X_t)} \geq \|h\|_H^2 \underbrace{ \left[ \frac{2 c_1 r^2 (2y - r^2) y}{(r^2-y)^4} - \frac{C_2 r^2}{(r^2-y)^2} \right] }_{A(y)} - \|h\|_H \underbrace{ \left[ \frac{2\lambda r^2 \sqrt{y}}{(r^2-y)^2} \right] }_{B(y)}.
\end{equation*}

Let us define $z:= \|h\|$. To establish a uniform global lower bound, we need to study the function
\begin{equation*}
    H(z;y)
    := z^{2} \left[ \frac{2c_{1}\, r^{2} (2y - r^{2})y}{(r^{2} -y)^4} - \frac{C_{2}\, r^{2}}{(r^{2} - y)^2} \right]
    - z \left[\frac{2 \lambda r^{2} \sqrt{y}}{(r^{2} - y)^2}\right],
\end{equation*}
for $z\in[0, r+B]$, and $y\in[0,r^2)$. Note moreover that $c_1\leq C_2$ since $Q$ is positive definite on $V$.

Set $\delta := r^{2}-y$, so we have $\delta \in (0,r^2]$. Then
\begin{equation*}
    (2y-r^{2})y = (r^{2}-\delta)(r^{2}-2\delta)
\end{equation*}
and
\begin{equation*}
    H(z;y) = A(y)\, z^2 - B(y)\, z,
\end{equation*}
with
\begin{equation*}
    A(y) = \frac{2c_{1}r^{2}(r^{2}-\delta)(r^{2}-2\delta)}{\delta^4} - \frac{C_{2}r^{2}}{\delta^2},
    \qquad
    B(y) = \frac{2\lambda r^{2} \sqrt{y}}{\delta^2}.
\end{equation*}

\textbf{Step 1.} \textit{(A lower bound for $A(y)$)} 
We have
\begin{equation*}
    A(y) = \frac{2c_{1}r^{2}(r^{2}-\delta)(r^{2}-2\delta)}{\delta^4} - \frac{C_{2}r^{2}}{\delta^2}.
\end{equation*}
\textbullet\; When $\delta \leq r^{2}/2$, we have
\begin{equation*}
    r^{2}-\delta \geq \frac{r^{2}}{2}.
\end{equation*}
Hence, for such $\delta$,
\begin{equation*}
    A(y) \geq \frac{c_{1}\,r^4(r^{2}-2\delta)}{\delta^4} - \frac{C_{2}r^{2}}{\delta^2}.
\end{equation*}
\textbullet\; When $\delta \leq r^{2}/4$, we have 
\begin{equation*}
    r^{2}-2\delta \geq r^{2}/2.
\end{equation*}
Hence, for such $\delta$,
\begin{equation*}
    A(y) \geq \frac{c_{1}\, r^6}{2\delta^4} - \frac{C_{2}\,r^{2}}{\delta^2} = \frac{r^{2}}{\delta^4} \left( \frac{c_{1}\, r^4}{2} - C_{2}\delta^2 \right).
\end{equation*}
\textbullet\; When $\delta$ satisfies
\begin{equation*}
    \delta \leq \frac{r^{2}}{2} \sqrt{\frac{c_{1}}{C_{2}}},
\end{equation*}
we have
\begin{equation*}
    \frac{c_{1} r^4}{2} - C_{2}\delta^2 \geq \frac{c_{1} r^4}{4}.
\end{equation*}

To combine these constraints, we define
\begin{equation*}
    \delta_* := \min\left( \frac{r^{2}}{4},\; \frac{r^{2}}{2}\sqrt{\frac{c_{1}}{C_{2}}} \right),
    \qquad
    y_* := r^{2} - \delta_*.
\end{equation*}
Then for all $y \in [y_*, r^{2})$, i.e. $\delta \leq \delta_*$,
\begin{equation*}
    A(y) \geq \frac{c_{1}\, r^6}{4\, \delta^4} > 0.
\end{equation*}
Therefore, the quadratic function $t\mapsto H(z;y)=A(y)\, z^2 - B(y)\, z\;$ is convex.

\textbf{Step 2.} \textit{(A first lower bound for $H(z;y)$)} 
We focus in this step on the region $y \in [y_*, r^{2})$ where $A(y)>0$. Because $H(z;y)$ is convex in $z$, we can compute its minimum attained in 
\begin{equation*}
    z^{*} = \frac{B(y)}{2 A(y)}.
\end{equation*}
and that is
\begin{equation*}
    \min_{z\geq 0} H(z;y) = -\frac{B(y)^2}{4A(y)}.
\end{equation*}
Next, we estimate
\begin{equation*}
    B(y)^2 = \frac{4\,\lambda^2\, r^4\, y}{\delta^4} \leq \frac{4\lambda^2\, r^6}{\delta^4},
\end{equation*}
and then 
\begin{equation*}
    \frac{B(y)^2}{4A(y)} \leq \frac{4\,\lambda^2\, r^6 / \delta^4}{4 (c_{1} r^6 / (4\delta^4))}= \frac{4\lambda^2}{c_{1}}.
\end{equation*}
Therefore,
\begin{equation*}
    H(z;y) \geq -\frac{4\lambda^2}{c_{1}} \quad \forall \, t\in[0,r+B],\; y\in[y_*,r^{2}).
\end{equation*}

\textbf{Step 3.} \textit{(A second lower bound for $H(z;y)$)} 
The previous lower bound is valid for $y\in [y_*,r^{2})$. We now provide a lower bound when $y\in [0,y_{*})$. Recall from the end of Step 1, the notation $y_{*} = r^{2} - \delta_{*}$. Then we have $(r^{2}-y) \geq \delta_*>0$,  hence all denominators in $H(z;y)$ are bounded 
\begin{equation*}
    \frac{1}{(r^{2}-y)^2} \leq \frac{1}{\delta_*^2},
    \qquad
    \frac{1}{(r^{2}-y)^4} \leq \frac{1}{\delta_*^4}.
\end{equation*}
Thus there exist constants $K_1,K_2>0$, depending only on
$r^{2},c_{1},C_{2},\lambda$, such that
\begin{equation*}
    |A(y)| \leq K_1, \qquad |B(y)| \leq K_2,
\end{equation*}
and therefore
\begin{equation*}
    H(z;y) \geq -(r+B)^2 K_1 - (r+B) K_2.
\end{equation*}

\textbf{Step 4.} \textit{(Conclusion)} 
There exists a constant $p>0$, depending only on
$r,c_{1},C_{2},\lambda,B$, such that
\begin{equation*}
    H(z;y) \geq -p \quad \forall z\in[0,r+B],\; y\in[0,r^{2}).
\end{equation*}

\textbf{Step 5.} \textit{(Synthesis and Gr\"onwall Integration)} 
We define the global lower bound as $p = \max(|K_1|, |K_2|)$. For all $X_t \in B_r(x^*)$, the operator satisfies 
\begin{equation*}
    T_1 + T_2 \geq -p \, \phi_r(X_t).
\end{equation*}
Taking expectations yields the linear differential inequality $\frac{\dd}{\dd t} \mathbb{E}[\phi_r(X_t)] \geq -p \mathbb{E}[\phi_r(X_t)]$. Integrating this inequality via Gr\"onwall's lemma yields 
\begin{equation*}
    \mathbb{E}[\phi_r(X_t)] \geq \mathbb{E}[\phi_r(X_0)] e^{-pt}.
\end{equation*}
Because the initial measure $\mu_0$ assigns strictly positive mass to the neighborhood of the minimizer, $\mathbb{E}[\phi_r(X_0)] > 0$. Consequently, the probability mass inside $B_r(x^*)$ remains strictly positive for all finite times $t \in [0,T]$, completing the proof.
\end{proof}

\begin{assumption}[Regularity of the Objective Function (bis)]
\label{assum:objective bis}
The function $\mathscr{E}$ satisfies the following properties:
\begin{enumerate}
    \item It is continuous with a unique global minimizer $x^*$ satisfying $\mathscr{E}(x^*) = 0$.
    \item It has a local quadratic growth near the minimizer, that is, there exists constants $\ell_1,\ell_2 > 0$ and a radius $R > 0$ such that 
    \begin{equation*}
    \begin{aligned}
        \mathscr{E}(x) &\geq \ell_1 \|x - x^*\|_H^2 \qquad \forall\, x\in B_{R}(x^*)\\
        \text{and } \; \mathscr{E}(x) &\geq \ell_2>0 \qquad \forall\, x\in B_{R}(x^*)^c.
    \end{aligned}
    \end{equation*}
\end{enumerate}
\end{assumption}

\begin{lemma}[Quantitative Laplace Principle and Consensus Contraction] \label{lem:laplace_contraction_degenerate}
Let the objective functional $\mathscr{E}: H \to \mathbb{R}$ satisfy Assumption \ref{assum:objective bis}. 
Fix any target radius $r \in (0, R)$ and a final time $T > 0$. Let $X_t$ be a solution to the degenerate CBO dynamics \eqref{eq:mckean_vlasov_cbo_degenerate} over $t \in [0,T]$, and $\mu_t = \mathrm{Law}(X_t)$. 
For any $B > 0$, assume that 
\[
\sup_{t \in [0,T]} \|x^* - \mathfrak{m}_{\alpha}(\mu_t)\|_H \leq B.
\]
Then, the consensus point $\mathfrak{m}_{\alpha}(\mu_t)$ satisfies 
\begin{equation} \label{eq:laplace_bound_statement}
    \|\mathfrak{m}_{\alpha}(\mu_t) - x^*\|_H^2 \leq r + \rho(\alpha, t) \mathbb{E}\big[\|X_t - x^*\|_H^2\big],
\end{equation}
where the contraction factor $\rho(\alpha, t)$ depends on the strictly positive mass of the swarm, and satisfies $\lim_{\alpha \to \infty} \rho(\alpha, t) = 0$ for any fixed $t \in [0,T]$.
\end{lemma}

\begin{proof}
We proceed in several steps. 

\textbf{Step 1.} \textit{(Jensen's Inequality and Domain Splitting)} 
Recall that the consensus point $\mathfrak{m}_{\alpha}(\mu_t)$ is the expectation of $x$ under the Gibbs measure $\mu_\alpha(\dd x) = \frac{1}{Z_\alpha} e^{-\alpha \mathscr{E}(x)} \mu_t(\dd x)$. The squared norm being convex, we can apply Jensen's inequality
\begin{equation} \label{eq:jensen_split_deg}
    \|\mathfrak{m}_{\alpha}(\mu_t) - x^*\|_H^2 = \left\| \int_H (x - x^*)\, \dd \mu_\alpha(x) \right\|_H^2 \leq \int_H \|x - x^*\|_H^2 \, \dd \mu_\alpha(x).
\end{equation}
We split the integration domain into the fixed local basin 
\[
B_r(x^*) = \{x \in H \;:\; \|x - x^*\|_H^2 \leq r\}
\]
and the far-field region $B_r^c$. 

Inside the local basin $B_r(x^*)$, the integrand is strictly bounded by $r$. Because $\mu_\alpha$ is a probability measure, this local integral is also bounded by $r$.

\textbf{Step 2.} \textit{(Bounding the Far-Field Integral)} 
In the far-field region $B_r^c$, the local quadratic growth assumption guarantees that the objective value is bounded from below $\mathscr{E}(x) \geq c_1 r$ or $c_2$. Consequently, the exponential Gibbs weight is strictly bounded from above. Factoring this maximum weight out of the integral gives 
\begin{equation}
 \label{eq:far_field_bound_deg}
\begin{aligned}
    &\int_{B_r^c} \|x - x^*\|_H^2 \frac{e^{-\alpha \mathscr{E}(x)}}{Z_\alpha} \,\dd\mu_t(x) \\
    & \qquad \qquad \leq \frac{e^{-\alpha \ell_1 r}}{Z_\alpha} \int_{B_r^c\cap B_R} \|x - x^*\|_H^2 \,\dd\mu_t(x)+ \frac{e^{-\alpha \ell_2}}{Z_\alpha} \int_{B_r^c\cap B_R^c} \|x - x^*\|_H^2 \,\dd\mu_t(x) \\
    & \qquad \qquad \leq \frac{e^{-\alpha \ell_r}}{Z_\alpha} \mathbb{E}\big[\|X_t - x^*\|_H^2\big] \qquad \text{ with } \ell_r:=\min\{\ell_{1}r,\,\ell_2\}.
\end{aligned}
\end{equation}

\textbf{Step 3.} \textit{(Lower Bounding the Partition Function $Z_\alpha$)} 
To control the denominator $Z_\alpha$, we must evaluate it over a region where the energy is arbitrarily small. Because $\mathscr{E}$ is continuous and $\mathscr{E}(x^*) = 0$, there exists a smaller radius $r_\varepsilon \in (0, r)$ such that for all $x \in B_{r_\varepsilon}(x^*)$, we have $\mathscr{E}(x) \leq \frac{\ell_r}{2}$.
Restricting the integration of $Z_\alpha$ to this smaller ball 
\begin{equation*}
    Z_\alpha = \int_H e^{-\alpha \mathscr{E}(x)} \,\dd\mu_t(x) \geq \int_{B_{r_\varepsilon}(x^*)} e^{-\alpha \mathscr{E}(x)} \,\dd\mu_t(x).
\end{equation*}
Inside $B_{r_\varepsilon}(x^*)$, the maximum possible energy is $\frac{\ell_r}{2}$, meaning the minimum possible weight is $e^{-\alpha \ell_r / 2}$. Thus
\begin{equation} \label{eq:z_alpha_bound}
    Z_\alpha \geq e^{-\alpha \ell_r/ 2} \,\mu_t\big(B_{r_\varepsilon}(x^*)\big).
\end{equation}

\textbf{Step 4.} \textit{(Incorporating the Strict Positivity Lemma)} 
We must ensure that $\mu_t\big(B_{r_\varepsilon}(x^*)\big)$ is strictly bounded away from zero. We invoke Lemma \ref{lem:positive_mass_degenerate} on the ball $B_{r_\varepsilon}(x^*)$. From the evolution of the smooth mollifier $t\mapsto\phi_{r_\varepsilon}(X_t)$, there exists a finite constant $p_\varepsilon > 0$ such that the mass is strictly lower-bounded for all $t \in [0,T]$
\begin{equation*}
    \mu_t\big(B_{r_\varepsilon}(x^*)\big) \geq \mathbb{E}[\phi_{r_\varepsilon}(\xi)]\, e^{-p_\varepsilon t} \;:=\; \eta(r_\varepsilon, t) > 0.
\end{equation*}
Substituting this positive mass bound back into \eqref{eq:z_alpha_bound} yields $Z_\alpha \geq e^{-\alpha \ell_1 \frac{r}{2}}\, \eta(r_\varepsilon, t)$.

\textbf{Step 5.} \textit{(The Contraction Factor and the $\alpha \to \infty$ Limit)} 
Substituting the lower bound for $Z_\alpha$ into our far-field bound \eqref{eq:far_field_bound_deg}, and assembling the domain split from Step 1, we obtain the differential inequality
\begin{equation*}
    \|\mathfrak{m}_{\alpha}(\mu_t) - x^*\|_H^2 \leq r + \rho(\alpha, t) \mathbb{E}\big[\|X_t - x^*\|_H^2\big],
\end{equation*}
where the contraction factor is defined as 
\begin{equation*}
    \rho(\alpha, t) = \frac{e^{-\alpha \ell_r}}{e^{-\alpha \ell_r / 2} \eta(r_\varepsilon, t)} = \frac{e^{-\alpha \ell_r / 2}}{\eta(r_\varepsilon, t)}.
\end{equation*}
Because we fixed the target radius $r$ (and thereby $r_\varepsilon$) before considering $\alpha$, the mass bound $\eta(r_\varepsilon, t)$ acts as a strictly positive, $\alpha$-independent constant. Taking the limit as the Laplace parameter $\alpha \to \infty$, the exponential decay in the numerator dominates, yielding
\begin{equation*}
    \lim_{\alpha \to \infty} \rho(\alpha, t) = \lim_{\alpha \to \infty} \frac{e^{-\alpha \ell_r / 2}}{\eta(r_\varepsilon, t)} = 0.
\end{equation*}
Thus, for any sufficiently large $\alpha$, we guarantee that the consensus point is strictly localized, completing the proof. 
\end{proof}

\begin{theorem}[Global Convergence] \label{thm:global_convergence_degenerate}
Let $X_t$ be a solution to the degenerate CBO dynamics \eqref{eq:mckean_vlasov_cbo_degenerate} in the Hilbert space $H$. Assume the parameters satisfy  $\lambda > \sigma^2 L_{\mathsf{D}}^2$. Suppose the objective function $\mathscr{E}$ satisfies Assumption \ref{assum:objective} and Assumption \ref{assum:objective bis}. Suppose moreover the ambient space $H$, the trace-class covariance $Q$, and the state-dependent diffusion operator $\mathsf{D}$ satisfy the active-subspace and regularity conditions in Assumption \ref{assump:active_subspace} and in Proposition \ref{prop:diffusion}.

For any arbitrarily small target accuracy $\varepsilon \in (0, \mathbb{E}[\|X_0 - x^*\|_H^2])$, there exist a finite time horizon $0<T_*\leq T_\varepsilon:= \frac{2}{\lambda} \ln \left( \frac{E(0)}{\varepsilon} \right)$ and a minimal Laplace parameter $\alpha_0(\varepsilon) > 0$ such that for all $\alpha \geq \alpha_0(\varepsilon)$, the expected squared distance of the swarm decays exponentially fast 
\begin{equation*} 
    \mathbb{E}\big[\|X_t - x^*\|_H^2\big] \leq \mathbb{E}\big[\|X_0 - x^*\|_H^2\big] e^{-\frac{\lambda}{2} t}, \quad \text{for all } t \in [0, T_*].
\end{equation*}
Consequently, both the swarm variance and the consensus point error achieve the target accuracy $\varepsilon$ at time $T_*$, that is
\begin{equation*}
    \mathbb{E}\big[\|X_{T_*} - x^*\|_H^2\big] \leq \varepsilon, \quad \text{and} \quad \|\mathfrak{m}_{\alpha}(\mu_{T_*}) - x^*\|_H^2 \leq \varepsilon.
\end{equation*}
\end{theorem}

\begin{proof}
We proceed in several steps. 
Let us define 
\[
    E(t):= \mathbb{E}\big[\|X_t - x^*\|_H^2\big].
\]

\textbf{Step 1.} \textit{(The Differential Inequality)} 
Applying the infinite-dimensional It\^o formula (see Theorem \ref{thm:Ito in H} and Example \ref{ex:Ito in H}) to $x\mapsto \|x - x^*\|_H^2$ and using the polarization identity\footnote{We use $\;2\langle a,b\rangle = \|a\|_{H}^{2} + \|b\|_{H}^{2} - \|a-b\|_{H}^{2}$, where $a = X_t - x^{*}$, and $b = X_t - \mathfrak{m}_{\alpha}(\mu_t)$.}, and Proposition \ref{prop:diffusion}(2), we obtain 
\begin{equation*}
    \frac{\dd}{\dd t} E(t) \leq -\lambda E(t) + \lambda \|\mathfrak{m}_{\alpha}(\mu_t) - x^*\|_H^2 - \left( \lambda - \sigma^2 L_{\mathsf{D}}^2 \right) \mathbb{E}\big[\|X_t - \mathfrak{m}_{\alpha}(\mu_t)\|_H^2\big].
\end{equation*}
The parameter condition $\lambda \geq \sigma^2 L_{\mathsf{D}}^2$ ensures the final term is non-positive. Discarding it yields 
\begin{equation} \label{eq:diff_ineq_deg}
    \frac{\dd}{\dd t} E(t) \leq -\lambda E(t) + \lambda \|\mathfrak{m}_{\alpha}(\mu_t) - x^*\|_H^2.
\end{equation}

\textbf{Step 2.} \textit{(Defining the Stopping Times and Target Horizon)} 
We fix the target local radius as $r = \frac{\varepsilon}{4}$. The finite time horizon $T_\varepsilon$ required for the exponential decay to hit the target accuracy is defined as
\begin{equation*}
    T_\varepsilon := \frac{2}{\lambda} \ln \left( \frac{E(0)}{\varepsilon} \right).
\end{equation*}
To handle the dependency between the mass lower bound and the consensus convergence, we introduce two stopping times. Let $T^\alpha_1$ be the time the variance hits the target accuracy, and $T^\alpha_2$ be the time the consensus point escapes a loose bounding threshold
\begin{align*}
    T^\alpha_1 &:= \inf\big\{t > 0 \mid E(t) \leq \varepsilon \big\}, \\
    T^\alpha_2 &:= \inf\big\{t > 0 \mid \|\mathfrak{m}_{\alpha}(\mu_t) - x^*\|_H^2 \geq E(0) + 1 \big\}.
\end{align*}
We define the effective time horizon as $T_\alpha = \min\{T^\alpha_1, T^\alpha_2, T_\varepsilon\}$.

\textbf{Step 3.} \textit{(Uniform Bounding on the Localized Interval)} 
For all $t \in [0, T_\alpha)$, the consensus point is strictly bounded by definition
\[
    \|\mathfrak{m}_{\alpha}(\mu_t) - x^*\|_H \leq \sqrt{E(0) + 1} := B.
\]
Because the consensus point is bounded, we may apply Lemma \ref{lem:positive_mass_degenerate} to obtain a strictly positive mass bound on the  set $B_r(x^*)$ with an escape rate $p$ that depends only on $B$ and $r$, independent of $\alpha$, that is
\begin{equation*}
    \mu_t(B_r(x^*)) \,\geq\, \eta(r, t) \,\geq\, \eta(r, T_\varepsilon)\, =: C_0 > 0.
\end{equation*}
By the Laplace Contraction (Lemma \ref{lem:laplace_contraction_degenerate}), this strictly positive mass enforces a bound on the consensus point for all $t \in [0, T_\alpha)$, and one gets
\begin{equation} \label{eq:laplace_sub}
    \|\mathfrak{m}_{\alpha}(\mu_t) - x^*\|_H^2 \leq r + \frac{e^{-\alpha \ell_r / 2}}{C_0} E(t) =: r + \rho(\alpha) E(t).
\end{equation}
We choose $\alpha_0$ sufficiently large such that for all $\alpha \geq \alpha_0$, we guarantee that $\rho(\alpha) \leq \frac{1}{4}$.

\textbf{Step 4.} \textit{(The Exponential Decay Phase)} 
Substituting the contraction bound \eqref{eq:laplace_sub} into the differential inequality \eqref{eq:diff_ineq_deg} yields
\begin{equation*}
    \frac{\dd}{\dd t} E(t) \leq -\lambda E(t) + \lambda \big( r + \rho(\alpha) E(t) \big).
\end{equation*}
For all $t \in [0, T_\alpha)$, because $t \leq T^\alpha_1$, the variance has not yet reached the target, so we have $E(t) \geq \varepsilon$. Consequently, our chosen radius $r = \frac{\varepsilon}{4}$  satisfies $r \leq \frac{E(t)}{4}$. Using this relation alongside $\rho(\alpha) \leq \frac{1}{4}$, the previous inequality becomes
\begin{equation*}
    \frac{\dd}{\dd t} E(t) \leq -\lambda E(t) + \lambda \left( \frac{E(t)}{4} + \frac{E(t)}{4} \right) = -\frac{\lambda}{2} \, E(t).
\end{equation*}
Applying Gr\"onwall's inequality yields an exponential decay
\begin{equation} \label{eq:gronwall_strict}
    E(t) \leq E(0) \, e^{-\frac{\lambda}{2} t} \qquad \text{for all } t \in [0, T_\alpha).
\end{equation}

\textbf{Step 5.} \textit{(Ruling out $T^\alpha_2$ via Continuous Induction)} 
We must now verify that the consensus point never triggers the threshold $T^\alpha_2$. Because $E(t)$ is strictly decreasing from \eqref{eq:gronwall_strict}, we have $E(t) \leq E(0)$. Substituting this into our consensus bound \eqref{eq:laplace_sub}, we get
\begin{equation*}
    \|\mathfrak{m}_{\alpha}(\mu_t) - x^*\|_H^2 \leq \frac{\varepsilon}{4} + \frac{1}{4} E(t) \leq \frac{\varepsilon}{4} + \frac{E(0)}{4}.
\end{equation*}
Because we assumed $\varepsilon < E(0)$, the right-hand side is strictly bounded by $\frac{E(0)}{2}$, which is less than the threshold $E(0) + 1$. By the continuity of the trajectories, this implies that $T^\alpha_2$ is never reached. Consequently, $T_\alpha = \min\{T^\alpha_1, T_\varepsilon\}$.

\textbf{Step 6.} \textit{(Target Accuracy at $T_\varepsilon$)} 
When $T^\alpha_1<T_\varepsilon$, one has $T_\alpha=T^\alpha_1$, and therefore $E(T_\alpha)=E(T^\alpha_1)=\varepsilon$ by the definition of $T^\alpha_1$.
For $T^\alpha_1\geq T_\varepsilon$, we have $T_\alpha=T_\varepsilon$.
Evaluating the exponential bound \eqref{eq:gronwall_strict} at the target horizon $T_\alpha$ yields
\begin{equation*}
      E(T_\alpha)=E(T_\varepsilon) \leq E(0) \exp\left( -\frac{\lambda}{2} \left[ \frac{2}{\lambda} \ln \left( \frac{E(0)}{\varepsilon} \right) \right] \right) = \varepsilon.
\end{equation*}
Thus, the variance reaches the $\varepsilon$ threshold exactly at or before time $T_\varepsilon$. 
Finally, substituting the achieved swarm accuracy $E(T_\alpha) \leq \varepsilon$ into the Laplace contraction inequality \eqref{eq:laplace_sub} establishes the accuracy of the algorithm's output
\begin{equation*}
    \|\mathfrak{m}_{\alpha}(\mu_{T_\alpha}) - x^*\|_H^2 \leq r + \rho(\alpha) E(T_\alpha) \leq \frac{\varepsilon}{4} + \frac{\varepsilon}{4} = \frac{\varepsilon}{2} \leq \varepsilon.
\end{equation*}
This completes the proof when we choose $T_*=T_\alpha$.
\end{proof}

\begin{remark}
In the previous proof, we can set $\varepsilon$, which sets $T_\varepsilon$. Subsequently, this defines $C_0$, and finally also defines $\alpha_0$.
\end{remark}

\begin{remark}[Galerkin Truncation and Subspace Minimization]
\label{rem:subspace_minimizer}
We emphasize a fundamental theoretical nuance regarding the projected diffusion operator $\mathsf{D}(x) = \|x\|_H \mathsf{P}_V$. Because the exploratory stochastic noise is strictly confined to the finite-dimensional active subspace $V$, the dynamics in the infinite-dimensional orthogonal complement $V^\perp$ are entirely deterministic. Consequently, the swarm lacks the stochastic energy required to perform global exploration or escape local minima residing in $V^\perp$. 

Therefore, the convergence established in our main results strictly guarantees that the consensus point exponentially tracks $x_V^*$, the global minimizer of the energy functional restricted to the subspace $V$. The total algorithmic error relative to the true global minimizer $x^* \in H$ can be decomposed via the triangle inequality as
\begin{equation*}
   \mathbb{E}\big[\|X_t - x^*\|_H^2\big]\leq \underbrace{\mathbb{E}\big[\|X_t - x_V^*\|_H^2\big]}_{\text{Optimization Error}} + \underbrace{\|x_V^* - x^*\|_H}_{\text{Approximation Error}}.
\end{equation*}
The exponential bounds derived in our theorems exclusively govern the optimization error. The residual approximation error is purely deterministic. Notably, if the global minimizer $x^*$ belongs to a relatively compact subset $\mathcal{K} \subset H$ (e.g., a specific Sobolev or Besov space), the fundamental lower limit of this approximation error is strictly governed by the Kolmogorov $M$-width \cite{pinkus2012n} of $\mathcal{K}$, where $M = \dim(V)$. Consequently, while the projected dynamics guarantee optimal convergence within the active subspace, the absolute terminal accuracy is inherently bounded by the spectral decay properties of the target space.
\end{remark}

\section{The Interacting Particle Model}
\label{sec:particle}

Let us first define the finite-particle model whose mean-field limit is \eqref{eq:mckean_vlasov_cbo}. 
We are given $N$ particles $\{X^{i}_{\cdot}\}_{i\in [N]}$, $[N]=\{1,2,...,N\}$, whose dynamics is governed by the $H$-valued SDE
\begin{equation}
\label{eq:finite_cbo}
\begin{aligned}
    & \dd X^{i}_t = -\lambda \left(X^{i}_t - \mathfrak{m}^{\alpha,N}_{t}\right)\,\dd t + \sigma\, \mathsf{D}\left(X^{i}_t - \mathfrak{m}^{\alpha,N}_{t}\right)\,\dd W_t^{Q,i}, \quad X^{i}_0 = \mathbf{x}^{i}_{\circ} 
\end{aligned}
\end{equation}
where $\mathfrak{m}^{\alpha,N}_{\cdot}$ represents the \textit{interaction} between all the particles $\{X^{i}_{\cdot}\}_{i\in [N]}$, and is given as a weighted convex combination of their position
\begin{equation}\label{consensus m}
    \mathfrak{m}^{\alpha,N}_{t} := \frac{1}{\sum\limits_{k=1}^{N}\omega_{\alpha}(X^{k}_{t})} \, \sum\limits_{j=1}^{N}\,X^{j}_{t}\; \omega_{\alpha}(X^{j}_{t}).
\end{equation}
Here, the weight coefficients $\omega_{\alpha}(\cdot)$ are as defined in Proposition \ref{prop:global_lipschitz}. They are functions evaluated in the position of the particle, and they depend on a parameter $\alpha>0$ such that $\;\omega_{\alpha}:H \to (0,\infty),\quad x\mapsto \omega_{\alpha}(x) = \exp(-\alpha\, \mathscr{E}(x)).$ The rigorous derivation of the mean-field limit estimate, bridging the particle system \eqref{eq:finite_cbo} to its mean-field counterpart \eqref{eq:mckean_vlasov_cbo}, is omitted here. We may adapt the methodologies presented in \cite{huang2022mean,huang2025uniform,bayraktar2025uniform,gerber2025uniform,gerber2023mean,choi2025modified}, which we leave for future work.

We would like to write the dynamics for the vector in $H^N$ whose entries are $X^i$, $i\in[N]$. We introduce the notation  
\begin{equation}\label{X H N}
    \mathds{X}^{N}_{t} := 
    \begin{pmatrix}
    X^{1}_{t}\\
    \vdots\\
    X^{i}_{t}\\
    \vdots\\
    X^{N}_{t}
    \end{pmatrix}
    \in \, H^{N}.
\end{equation}
Then we define the map which gives $\mathfrak{m}^{\alpha,N}$ such as
\begin{equation}\label{G}
\begin{aligned}
    \mathtt{G}:H^{N} & \longrightarrow H, \quad
    (X^{i})_{i\in [N]}=\mathds{X}  \longmapsto \mathtt{G}(\mathds{X}) := \frac{1}{\sum\limits_{k=1}^{N}\omega_{\alpha}(X^{k})} \, \sum\limits_{j=1}^{N}\,X^{j}\; \omega_{\alpha}(X^{j}).
\end{aligned}
\end{equation}
We can now write
\begin{equation}
\label{eq:mathds F}
\begin{aligned}
    X^{i} - \mathtt{G}(\mathds{X}) 
    & =  X^{i} - \frac{1}{\sum\limits_{k=1}^{N}\omega_{\alpha}(X^{k})} \, \sum\limits_{j=1}^{N}\,X^{j}\; \omega_{\alpha}(X^{j})\\ 
    & = \sum\limits_{j=1}^{N}\,\frac{\big(X^{i}-X^{j}\big)\; \omega_{\alpha}(X^{j})}{\sum\limits_{k=1}^{N}\omega_{\alpha}(X^{k})} \quad =: \; \mathbf{F}^{i}(\mathds{X}) \in \mathbb{H}.
\end{aligned}
\end{equation}
Setting $\mathds{F}(\mathds{X}):=\big( \mathbf{F}^{i}(\mathds{X}) \big)_{i\in [N]} \in H^{N}$, the vector made of $N$ dynamics each given by SDE \eqref{eq:finite_cbo} takes the form 
\begin{equation}\label{CBO in H N F}
    \dd \mathds{X}_{t} = -\lambda \, \mathds{F}(\mathds{X}_{t}) \, \dd t + \sigma\, \mathbb{D}(\mathds{F}(\mathds{X}_{t})) \, \dd \mathbb{W}^{Q}_{t},\quad t>0, \quad \mathds{X}_{0}=\mathbf{x}_{\circ}=(\mathbf{x}^{i}_{\circ})_{i\in[N]},
\end{equation}
where $\mathbb{W}^{Q}_{t} = (W^{Q,i}_{t})_{i\in [N]}$ and each $W^{Q,i}_{\cdot}$ is as defined in \S\ref{subsec:q_wiener}. Here $\mathbb{D}$ is simply
\begin{equation*}
\begin{aligned}
     \mathbb{D}: \; H^{N} & \longrightarrow \mathscr{L}(H^{N},H^{N})\\
    (Y^{i})_{i\in [N]}=\mathds{Y} \;& \longmapsto \mathbb{D}(\mathds{Y}) = (\mathsf{D}(Y^{i}))_{i\in[N]}
\end{aligned}
\end{equation*}
where $\mathsf{D}$ is as defined in \eqref{def:diffusion}. When $\mathbb{D}$ is evaluated in $\mathds{Y} = \mathds{F}(\mathds{X}_t)$, we recover
\[
\mathbb{D}(\mathds{F}(\mathds{X}_t)) 
=
\begin{pmatrix}
    \mathsf{D}(\mathbf{F}^{1}(\mathds{X}_t))\\
    \vdots\\
    \mathsf{D}(\mathbf{F}^{i}(\mathds{X}_t))\\
    \vdots\\
    \mathsf{D}(\mathbf{F}^{N}(\mathds{X}_t))
\end{pmatrix}
=
\begin{pmatrix}
    \mathsf{D}(X^1_t- \mathtt{G}(\mathds{X}_t))\\
    \vdots\\
    \mathsf{D}(X^i_t- \mathtt{G}(\mathds{X}_t))\\
    \vdots\\
    \mathsf{D}(X^N_t- \mathtt{G}(\mathds{X}_t))
\end{pmatrix}
=
\begin{pmatrix}
    \mathsf{D}(X^1_t- \mathfrak{m}^{\alpha,N}_{t})\\
    \vdots\\
    \mathsf{D}(X^i_t- \mathfrak{m}^{\alpha,N}_{t})\\
    \vdots\\
    \mathsf{D}(X^N_t- \mathfrak{m}^{\alpha,N}_{t})
\end{pmatrix}
\]
where we recall $\mathtt{G}$ is defined in \eqref{G}.

\begin{lemma}\label{lem: F}
Let Assumption \ref{assum:objective} and Assumption \ref{assump:active_subspace} be satisfied, and suppose we are in the setting of Proposition \ref{prop:diffusion}. Then the drift and diffusion coefficients of \eqref{CBO in H N F} are locally Lipschitz continuous. In particular, the map $\mathds{F}:H^N \to H^N$ is locally Lipschitz continuous.
\end{lemma}

\begin{proof}
Thanks to item (1) of Proposition \ref{prop:diffusion}, $\mathsf{D}$ is Lipschitz, hence also $\mathbb{D}$. It suffices then to show that $\mathds{F}$ is Lipschitz. 
Let us recall 
\begin{equation*}
\begin{aligned}
    \mathds{F}: H^N & \longrightarrow H^N,\;
        \mathds{X}  \longmapsto \big( \mathbf{F}^{i}(\mathds{X}) \big)_{i\in [N]} 
\end{aligned}
\end{equation*}
and for each $i\in [N]$, we have \eqref{eq:mathds F} that is
\begin{equation*}
\begin{aligned}
    \mathbf{F}^{i}: H^N & \longrightarrow H,\quad
        \mathds{X}  \longmapsto \sum\limits_{j=1}^{N}\,\frac{\big(X^{i}-X^{j}\big)\; \omega_{\alpha}(X^{j})}{\sum\limits_{k=1}^{N}\omega_{\alpha}(X^{k})}
\end{aligned}
\end{equation*}
It is therefore sufficient to verify that $\mathbf{F}^i$ is globally Lipschitz continuous.

Let $i\in[N]$ be arbitrarily fixed. 
Let $\mathds{X}, \mathds{Y} \in H^N$ satisfy $\;\|\mathds{X}\|_{H^N}, \|\mathds{Y}\|_{H^N} \le R\;$
for some fixed $R>0$. 
We shall estimate the quantity 
\begin{equation*}
    \mathbf{F}^{i}(\mathds{X}) - \mathbf{F}^{i}(\mathds{Y}) \, \in \mathbb{H},\quad \forall\, i\in [N].
\end{equation*}
We define
\begin{equation*}
    A^i(\mathds{X}) := \sum_{j=1}^N (X^i - X^j)\,\omega_\alpha(X^j), \qquad B(\mathds{X}) := \sum_{k=1}^N \omega_\alpha(X^k), \qquad \text{ so that } \; \mathbf{F}^i(\mathds{X}) = \frac{A^i(\mathds{X})}{B(\mathds{X})}.
\end{equation*}
Let us also introduce
\begin{equation*}
    M_\omega(R) := \sup_{\|x\|_H \le R} \omega_\alpha(x), \qquad m_\omega(R) := \inf_{\|x\|_H \le R} \omega_\alpha(x) > 0.
\end{equation*}
and define
\begin{equation*}
    M_B(R) := N M_\omega(R), \qquad m_B(R) := N m_\omega(R).
\end{equation*}

\textbf{Step 1.} \textit{(The quantities depending on $B$)}  
For all $\mathds{X}$ such that $\|\mathds{X}\|_{H^N} \le R$, it holds
\begin{equation}\label{eq: B bound}
m_B(R) \le B(\mathds{X}) \le M_B(R).
\end{equation}
On the other hand, recalling Proposition \ref{prop:global_lipschitz}, we have $\omega_\alpha$ is globally Lipschitz with a constant denoted $L_\omega$. Therefore
\begin{equation*}
    |B(\mathds{X}) - B(\mathds{Y})| \le \sum_{k=1}^N |\omega_\alpha(X^k) - \omega_\alpha(Y^k)| \le L_\omega \sum_{k=1}^N \|X^k - Y^k\|_H \le L_\omega \|\mathds{X} - \mathds{Y}\|_{H^N}.
\end{equation*}
Moreover, 
\begin{equation}\label{eq: B lip}
    \left|\frac{1}{B(\mathds{X})} - \frac{1}{B(\mathds{Y})}\right| = \frac{|B(\mathds{X}) - B(\mathds{Y})|}{B(\mathds{X}) B(\mathds{Y})} \le \frac{L_\omega}{m_B(R)^2} \|\mathds{X} - \mathds{Y}\|_{H^N}.
\end{equation}

\textbf{Step 2.} \textit{(The quantities depending on $A$)} 
We estimate 
\begin{equation*}
    A^i(\mathds{X}) - A^i(\mathds{Y}) = \sum_{j=1}^N \Big[(X^i - X^j)\omega_\alpha(X^j) - (Y^i - Y^j)\omega_\alpha(Y^j)\Big].
\end{equation*}
We split each term as
\begin{equation*}
\begin{aligned}
    &(X^i - X^j)\omega_\alpha(X^j) - (Y^i - Y^j)\omega_\alpha(Y^j) \\
    & = (X^i - X^j - Y^i + Y^j)\omega_\alpha(X^j) + (Y^i - Y^j)(\omega_\alpha(X^j) - \omega_\alpha(Y^j)).
\end{aligned}
\end{equation*}
\textbullet\; Since $\omega_\alpha$ is bounded on bounded sets, we have
\begin{equation*}
    \|(X^i - X^j - Y^i + Y^j)\omega_\alpha(X^j)\|_H \le M_\omega(R)\big(\|X^i - Y^i\|_H + \|X^j - Y^j\|_H\big).
\end{equation*}
\textbullet\; Since $\|\mathds{Y}\|_{H^N} \le R$, we have
\begin{equation*}
    \|Y^i - Y^j\|_H \le 2R.
\end{equation*}
\textbullet\; Using the Lipschitz property of $\omega_\alpha$ with constant $L_\omega$,
\begin{equation*}
    |\omega_\alpha(X^j) - \omega_\alpha(Y^j)| \le L_\omega \|X^j - Y^j\|_H,
\end{equation*}
thus we obtain
\begin{equation*} 
    \|(Y^i - Y^j)(\omega_\alpha(X^j) -    \omega_\alpha(Y^j))\|_H \le 2R L_\omega \|X^j - Y^j\|_H.
\end{equation*}

Therefore, for each $j$,
\begin{equation*}
\begin{aligned}
    &\|(X^i - X^j)\omega_\alpha(X^j) - (Y^i - Y^j)\omega_\alpha(Y^j)\|_H \\
    & \le M_\omega(R)\|X^i - Y^i\|_H + \big(M_\omega(R) + 2R L_\omega\big)\|X^j - Y^j\|_H.
\end{aligned}
\end{equation*}
Summing over $j=1,\dots,N$ yields
\begin{equation*}
    \|A^i(\mathds{X}) - A^i(\mathds{Y})\|_H \,\le\, N\, M_\omega(R)\|X^i - Y^i\|_H + N\big(M_\omega(R) + 2R L_\omega\big) \|\mathds{X} - \mathds{Y}\|_{H^N}.
\end{equation*}
Finally, using $\|X^i - Y^i\|_H \le \|\mathds{X} - \mathds{Y}\|_{H^N}$, we obtain
\begin{equation*}
    \|A^i(\mathds{X}) - A^i(\mathds{Y})\|_H \le N\big(2R M_\omega(R) + 2R L_\omega\big) \|\mathds{X} - \mathds{Y}\|_{H^N}.
\end{equation*}
Let us define the constant
\begin{equation*}
    C_A(R) := N\big(2R M_\omega(R) + 2R L_\omega\big).
\end{equation*}
Thus,
\begin{equation}\label{eq: A lip}
    \|A^i(\mathds{X}) - A^i(\mathds{Y})\|_{H^N} \le C_A(R)\|\mathds{X} - \mathds{Y}\|_{H^N}.
\end{equation}

On the other hand, let us recall that
\begin{equation*}
    A^i(\mathds{Y}) = \sum_{j=1}^N (Y^i - Y^j)\,\omega_\alpha(Y^j).
\end{equation*}
Taking the $H$-norm and using the triangle inequality gives
\begin{equation*}
    \|A^i(\mathds{Y})\|_{H} \le \sum_{j=1}^N \|Y^i - Y^j\|_{H}\,\omega_\alpha(Y^j).
\end{equation*}
Since $\|\mathds{Y}\|_{H^N} \le R$, we have $\;\|Y^i - Y^j\|_{H} \le \|Y^i\|_{H} + \|Y^j\|_{H} \le 2R.\;$
Moreover, $\omega_\alpha$ is bounded on bounded sets, so  $\;\omega_\alpha(Y^j) \le M_\omega(R).\;$ Combining these estimates yields
\begin{equation*}
    \|Y^i - Y^j\|_{H}\,\omega_\alpha(Y^j) \le 2R\,M_\omega(R).
\end{equation*}
Summing over $j=1,\dots,N$, we obtain
\begin{equation}\label{eq: A bound}
    \|A^i(\mathds{Y})\|_{H} \le \sum_{j=1}^N 2R\,M_\omega(R) = N \, 2R \, M_\omega(R).
\end{equation}

\textbf{Step 3.} \textit{(Conclusion)} 
We write
\begin{equation*}
    \mathbf{F}^i(\mathds{X}) - \mathbf{F}^i(\mathds{Y}) = \frac{A^i(\mathds{X}) - A^i(\mathds{Y})}{B(\mathds{X})} + A^i(\mathds{Y}) \left( \frac{1}{B(\mathds{X})} - \frac{1}{B(\mathds{Y})} \right).
\end{equation*}
\textbullet\; Using previous bounds \eqref{eq: B bound} and \eqref{eq: A lip},
\begin{equation*}
    \left\|\frac{A^i(\mathds{X}) - A^i(\mathds{Y})}{B(\mathds{X})}\right\|_{H^N} \le \frac{C_A(R)}{m_B(R)} \|\mathds{X} - \mathds{Y}\|_{H^N}.
\end{equation*}
\textbullet\; Using previous bounds \eqref{eq: B lip} and \eqref{eq: A bound},
\begin{equation*}
    \|A^i(\mathds{Y})\| \left|\frac{1}{B(\mathds{X})} - \frac{1}{B(\mathds{Y})}\right| \le \frac{2N R M_\omega(R)\, L_\omega}{m_B(R)^2} \|\mathds{X} - \mathds{Y}\|_{H^N}.
\end{equation*}
Thus we obtain
\begin{equation*}
    \|\mathbf{F}(\mathds{X}) - \mathbf{F}(\mathds{Y})\|_{H^N} \le C(R)\|\mathds{X} - \mathds{Y}\|_{H^N},
\end{equation*}
where $C(R)$ is a constant defined by
\begin{equation*}
    C(R) := \frac{C_A(R)}{m_B(R)} + \frac{2N R M_\omega(R)\, L_\omega}{m_B(R)^2} =  \frac{2R M_\omega(R) + 2R L_\omega}{m_\omega(R)} + \frac{2 R M_\omega(R)\, L_\omega}{N(m_\omega(R))^2}
\end{equation*}
Hence $\mathbf{F}$ is locally Lipschitz on $H^N$.

To show the linear growth of $\mathbf{F}^{i}$, it suffices to observe that
\begin{equation*}
    |\mathbf{F}^{i}(\mathds{X})| \leq \left|  X^{i} - \frac{1}{\sum\limits_{k=1}^{N}\omega_{\alpha}(X^{k})} \, \sum\limits_{j=1}^{N}\,X^{j}\; \omega_{\alpha}(X^{j})  \right| \leq |X^{i}| + \|\mathds{X}\|\leq 2\,\|\mathds{X}\|.
\end{equation*}
\end{proof}

\begin{theorem}[Global Well-Posedness]\label{thm: global well-posedness}
Let the standing assumptions hold, and assume the initial condition satisfies $\mathbf{x}_{\circ} \in \mathbb{H}^{N}$. Then, the SDE
\begin{equation}\label{eq: CBO main compact}
    \dd \mathds{X}_{t} = -\lambda \, \mathds{F}(\mathds{X}_{t}) \, \dd t + \sigma\, \mathbb{D}(\mathds{F}(\mathds{X}_{t})) \, \dd \mathbb{W}_{t},\quad t>0
\end{equation}
admits a unique global strong solution $\mathds{X}_{t}$ in $\mathbb{H}^{N}$.
\end{theorem}

\begin{proof}
We proceed in two main steps: we establish local existence via truncation and localization (Remark \ref{rmk:local}), then we prove non-explosion via an energy estimate applying It\^o's formula.

\textbf{Step 1.} \textit{(Local existence via truncation)} 
For each $R > 0$, we introduce a smooth cutoff function $\phi_{R}: \mathbb{H}^{N} \to [0,1]$ such that $\phi_{R}(\mathds{X}) = 1$ if $\|\mathds{X}\| \leq R$, and $\phi_{R}(\mathds{X}) = 0$ if $\|\mathds{X}\| \geq R+1$. We define the truncated coefficients 
\begin{equation*}
    \mathds{F}_{R}(\mathds{X}) := \phi_{R}(\mathds{X})\mathds{F}(\mathds{X}) \quad \text{and} \quad \mathbb{G}_{R}(\mathds{X}) := \phi_{R}(\mathds{X})\sigma\mathbb{D}(\mathds{F}(\mathds{X})).
\end{equation*}
By Lemma \ref{lem: F}, $\mathds{F}$ is locally Lipschitz. Because $\phi_{R}$ has a bounded support, the truncated drift $\mathds{F}_{R}$ and diffusion $\mathbb{G}_{R}$ are globally Lipschitz continuous and bounded on $\mathbb{H}^{N}$. 

By the standard existence and uniqueness theorem for SDEs in Hilbert spaces (see Theorem \ref{thm:solution}), the truncated SDE
\begin{equation*}
    \dd \mathds{X}_{t}^{R} = -\lambda \, \mathds{F}_{R}(\mathds{X}_{t}^{R}) \, \dd t + \mathbb{G}_{R}(\mathds{X}_{t}^{R}) \, \dd \mathbb{W}_{t},\quad \mathds{X}_{0}^{R} = \mathbf{x}_{\circ}
\end{equation*}
admits a unique global strong solution $\mathds{X}_{t}^{R}$.

We define the stopping time $\tau_{R} := \inf\{t \geq 0 : \|\mathds{X}_{t}^{R}\| \geq R\}$. By pathwise uniqueness (see Remark \ref{rmk:uniq}), for any $R' > R$, we have $\mathds{X}_{t}^{R} = \mathds{X}_{t}^{R'}$ almost surely for all $t \in [0, \tau_{R}]$. This consistency allows us to define a unique maximal local solution $\mathds{X}_{t} := \mathds{X}_{t}^{R}$ for $t \in [0, \tau_{\infty})$, where $\tau_{\infty} := \lim_{R \to \infty} \tau_{R}$ is the explosion time.

\textbf{Step 2.} \textit{(Non-explosion)} 
To prove that the solution is global, we must show that $\mathbb{P}(\tau_{\infty} = \infty) = 1$. We apply It\^o's formula to the functional $V(\mathds{X}) = \|\mathds{X}\|^{2}$ (see Theorem \ref{thm:Ito in H} and Example \ref{ex:Ito in H}). For the stopped process at $t \wedge \tau_{R}$, we obtain 
\begin{equation*}
\begin{aligned}
    \|\mathds{X}_{t \wedge \tau_{R}}\|^{2} &= \|\mathbf{x}_{\circ}\|^{2} + \int_{0}^{t \wedge \tau_{R}} \left( -2\lambda \langle \mathds{X}_{s}, \mathds{F}(\mathds{X}_{s}) \rangle + \|\sigma\mathbb{D}(\mathds{F}(\mathds{X}_{s}))\|_{\text{HS}}^{2} \right) \dd s \\
    &\quad + \int_{0}^{t \wedge \tau_{R}} 2\sigma \langle \mathds{X}_{s}, \mathbb{D}(\mathds{F}(\mathds{X}_{s})) \dd \mathbb{W}_{s} \rangle,
\end{aligned}
\end{equation*}
where $\|\cdot\|_{\text{HS}}$ denotes the Hilbert-Schmidt norm. 

Taking the expectation on both sides eliminates the stochastic integral, since the stopped integrand is bounded, making it a true martingale. We estimate the terms inside the Lebesgue integral using the Cauchy-Schwarz inequality and the linear growth property $\|\mathds{F}(\mathds{X})\| \leq 2\|\mathds{X}\|$ from Lemma \ref{lem: F}:
\begin{equation*}
    -2\lambda \langle \mathds{X}_{s}, \mathds{F}(\mathds{X}_{s}) \rangle \leq 2\lambda \|\mathds{X}_{s}\| \|\mathds{F}(\mathds{X}_{s})\| \leq 4\lambda \|\mathds{X}_{s}\|^{2}.
\end{equation*}
For the diffusion term, since $\mathbb{D}(\cdot)$ is a bounded linear operator thanks to Proposition \ref{prop:diffusion}, whose norm is controlled by the supremum norm of its argument, there exists a constant $C > 0$ such that 
\begin{equation*}
    \|\sigma\mathbb{D}(\mathds{F}(\mathds{X}_{s}))\|_{\text{HS}}^{2} \leq \sigma^{2} C \|\mathds{F}(\mathds{X}_{s})\|^{2} \leq 4\sigma^{2} C \|\mathds{X}_{s}\|^{2}.
\end{equation*}
Combining these bounds yields a constant $K = 4\lambda + 4\sigma^{2}C$ such that 
\begin{equation*}
    \mathbb{E}\big[\|\mathds{X}_{t \wedge \tau_{R}}\|^{2}\big] \leq \|\mathbf{x}_{\circ}\|^{2} + K \, \mathbb{E} \left[ \int_{0}^{t \wedge \tau_{R}} \|\mathds{X}_{s}\|^{2} \dd s \right] \leq \|\mathbf{x}_{\circ}\|^{2} + K \int_{0}^{t} \mathbb{E}\big[\|\mathds{X}_{s \wedge \tau_{R}}\|^{2}\big] \dd s.
\end{equation*}
By Gr\"onwall's lemma, we obtain the bound
\begin{equation*}
    \mathbb{E}\big[\|\mathds{X}_{t \wedge \tau_{R}}\|^{2}\big] \leq \|\mathbf{x}_{\circ}\|^{2} e^{Kt}.
\end{equation*}

To conclude non-explosion, we fix an arbitrary time $T > 0$. 
To establish non-explosion, we relate the expected value of the stopped process to the probability of exiting the ball $B_R$. Because the strong solution possesses continuous sample paths, the process precisely intersects the boundary at the stopping time, meaning $\|\mathds{X}_{\tau_R}\|^2 = R^2$ almost surely on the event $\{\tau_R \leq T\}$. By decomposing the expectation over disjoint events and discarding the non-negative contribution from $\{\tau_R > T\}$, we obtain the lower bound:
\begin{align*}
    \mathbb{E}\big[\|\mathds{X}_{T \wedge \tau_R}\|^2\big] &\geq \mathbb{E}\Big[\|\mathds{X}_{T \wedge \tau_R}\|^2 \mathbf{1}_{\{\tau_R \leq T\}}\Big] \\
    &= \mathbb{E}\Big[\|\mathds{X}_{\tau_R}\|^2 \mathbf{1}_{\{\tau_R \leq T\}}\Big] \\
    &= R^2 \, \mathbb{P}(\tau_R \leq T).
\end{align*}
Rearranging the inequality yields:
\begin{equation*}
     \mathbb{P}(\tau_R \leq T) \leq \frac{1}{R^{2}} \mathbb{E}\big[\|\mathds{X}_{T \wedge \tau_{R}}\|^{2}\big]\leq \frac{\|\mathbf{x}_{\circ}\|^{2} e^{KT}}{R^{2}}.
\end{equation*}
Taking the limit as $R \to \infty$, we find $\mathbb{P}(\tau_{\infty} \leq T) = 0$. Since $T$ was arbitrary, we conclude $\mathbb{P}(\tau_{\infty} = \infty) = 1$, meaning the solution does not explode in finite time and is globally well-posed.
\end{proof}

\section{Numerical Framework and Implementation}
\label{sec:numerics}

\subsection{A Numerical Example}

In this section, we delineate the numerical strategy for implementing CBO in an infinite-dimensional setting. To validate the theoretical results established in Section \ref{sec:convergence}, we demonstrate the model's ability to minimize functionals over the separable Hilbert space $H = L^2([0, 1])$.

The structure of our numerical illustration follows a three-step projection-realization-integration process:

\begin{enumerate}[label = \arabic*.]
    \item \textbf{Function Space Projection (The Active Subspace):} To transition from the continuous theory to a computable algorithm, we approximate the infinite-dimensional Hilbert space $H$ by a finite-dimensional subspace $H_M \subset H$. In the context of our theoretical framework (Assumption \ref{assump:active_subspace}), this discretization space $H_M$ acts exactly as our active subspace $V$. Let $\{\phi_k\}_{k=1}^M$ be an orthonormal basis of $V$. Any particle $X_t^{(n)}$ is represented purely by its active spectral coefficients:
    \begin{equation*}
        X_t^{(n)}(x) \approx \sum_{k=1}^M c_{k,t}^{(n)} \phi_k(x).
    \end{equation*}
    For our experiments, we utilize the sine basis $\phi_k(x) = \sqrt{2}\sin(k \pi x)$, which naturally enforces zero boundary conditions.

    \item \textbf{Realization of the $Q$-Wiener Process:} The driving noise must satisfy the trace-class condition $\mathrm{Tr}(Q) < \infty$. Numerically, we achieve this by assigning decaying variance to higher-frequency modes. The stochastic increment $\Delta W_t^{Q(n)}$ for each particle is computed as 
    \begin{equation*}
        \Delta W_t^{Q(n)}(x) = \sum_{k=1}^M \sqrt{\lambda_k} \, \zeta_{k,t}^{(n)} \phi_k(x), \quad \zeta_{k,t}^{(n)} \sim \mathcal{N}(0, \Delta t),
    \end{equation*}
    where $\lambda_k = k^{-2}$ ensures the summability of the eigenvalues.

    \item \textbf{Quadrature-based Consensus:} The consensus point $\mathfrak{m}_{\alpha}(\mu_t)$ requires the evaluation of the Gibbs weights $w^{(n)} = \exp(-\alpha\, \mathscr{E}(X_t^{(n)}))$. We approximate the functional $\mathscr{E}$ and the consensus integral using a high-resolution grid in the spatial domain $[0, 1]$.
\end{enumerate}

The implementation of the CBO scheme in the present setting is summarized in Algorithm \ref{alg:cbo_hilbert}.

\begin{algorithm}
\caption{Consensus-Based Optimization in Hilbert Space $H$}
\label{alg:cbo_hilbert}
\begin{algorithmic}[1]
\REQUIRE Objective functional $\mathscr{E}: H \to \mathbb{R}$, swarm size $N$, discretization resolution $M$, time step $\Delta t$, consensus parameters $\alpha, \lambda, \nu$, and trace-class covariance $Q$.
\ENSURE Approximated global minimizer $u^* \in H$.

\bigskip

\STATE \textbf{Initialize:} 
\STATE Select orthonormal basis $\{\phi_k\}_{k=1}^M$ of $H_M \subset H$.
\STATE Set eigenvalues $\lambda_k$ such that $\sum_{k=1}^M \lambda_k < \infty$.
\STATE Sample $N$ initial particles $\{X_0^{(1)}, \dots, X_0^{(N)}\}$ in $H_M$.

\bigskip

\FOR{$t = 0, \Delta t, 2\Delta t, \dots, T-\Delta t$}
    \STATE \textbf{1. Evaluate Energies:} 
    \FOR{each particle $n=1, \dots, N$}
        \STATE Compute $E_t^{(n)} = \mathscr{E}(X_t^{(n)})$.
    \ENDFOR

\bigskip

    \STATE \textbf{2. Compute Stable Gibbs Weights (Log-Sum-Exp):}
    \STATE Find $M_{\max} = \max_j (-\alpha E_t^{(j)})$.
    \FOR{each particle $n=1, \dots, N$}
        \STATE $\omega_t^{(n)} = \frac{\exp(-\alpha E_t^{(n)} - M_{\max})}{\sum_{j=1}^N \exp(-\alpha E_t^{(j)} - M_{\max})}$.
    \ENDFOR

\bigskip

    \STATE \textbf{3. Calculate Consensus Point:}
    \STATE $\mathfrak{m}_{\alpha}(\mu_t) = \sum_{n=1}^N \omega_t^{(n)} X_t^{(n)}$.

\bigskip

    \STATE \textbf{4. Update Swarm Particles:}
    \FOR{each particle $n=1, \dots, N$}
        \STATE \textbf{a. Sample $Q$-Wiener increment:}
        \STATE $\Delta W_t^{Q(n)} = \sum_{k=1}^M \sqrt{\lambda_k} \; \zeta_k\; \phi_k$, where $\zeta_k \sim \mathcal{N}(0, \Delta t)$.
        \STATE \textbf{b. Compute Hilbert distance:}
        \STATE $d_H^{(n)} = \|X_t^{(n)} - \mathfrak{m}_{\alpha}(\mu_t)\|_H$.
        \STATE \textbf{c. Update position via SDE discretization:}
        \STATE Because $\Delta W_t^{Q(n)}$ is constructed strictly from the basis of $V$, the projection $\mathsf{P}_V$ acts as the identity:
        \STATE $X_{t+\Delta t}^{(n)} = X_t^{(n)} - \lambda (X_t^{(n)} - \mathfrak{m}_{\alpha}(\mu_t)) \Delta t + \sigma\, d_H^{(n)}\, \mathsf{P}_V\, \Delta W_t^{Q(n)}$.
    \ENDFOR
\ENDFOR

\bigskip

\RETURN $\mathfrak{m}_{\alpha}(\mu_T)$
\end{algorithmic}
\end{algorithm}


\subsection{Illustration: Recovery of a Target Function}

We consider the task of recovering a hidden target function $f \in H$ by minimizing the quadratic functional 
\begin{equation*}
    \mathscr{E}(u) = \frac{1}{2} \int_0^1 |u(x) - f(x)|^2\, \dd x.
\end{equation*}
We set $f(x) = \sin(2\pi x)$ and initialize $N=2000$ particles as random smooth functions with coefficients $c_{k,0} \sim \mathcal{N}(0, 0.5)$.

For the numerical experiment, we employ $N=2000$ particles and a spectral discretization with resolution $M=1000$. The stochastic dynamics are evolved up to final time $T=6$ using a time step $\Delta t=0.01$. The model parameters are chosen as $\lambda=1.0$ for the drift coefficient and $\sigma=0.8$ for the noise intensity. The Laplace parameter appearing in the consensus weights is fixed to $\alpha=10^{15}$.

As observed in Figure \ref{fig:convergence_evolution}, the $Q$-Wiener process ensures that the particles remain smooth functions within $H$ throughout the optimization. The transition from $t=0.01$ to $t=1$ empirically  demonstrates that the weighted mean $\mathfrak{m}_{\alpha}$ effectively tracks the global minimizer even in the presence of high-dimensional stochastic noise.

\newpage

To numerically validate the theoretical guarantees established in Theorem \ref{thm:global_convergence_degenerate}, we track the evolution of the squared $L^2$ error between the empirical consensus point $m_\alpha$ and the global minimizer $f$. 
Figure \ref{fig:exponential_convergence} illustrates this trajectory on a semi-logarithmic scale. For $t \in [0, 2]$, the empirical error approximately follows a linear downward path. On a logarithmic axis, this linear decay confirms the global exponential convergence rate  of the infinite-dimensional interacting particle system. For $t > 3$, the error stabilizes at a small residual value (approximately $1.5 \times 10^{-2}$). It reflects the expected discretization limits of the fully computable numerical scheme. This terminal accuracy is governed by a combination of the finite-particle approximation ($N=2000$), the Euler-Maruyama time-stepping resolution ($\Delta t = 0.01$), and the spatial truncation of the infinite-dimensional Hilbert space onto a finite grid ($M=1000$). 

\begin{figure}[h]
     \centering
     \begin{subfigure}[b]{0.32\textwidth}
         \centering
         \includegraphics[width=\textwidth]{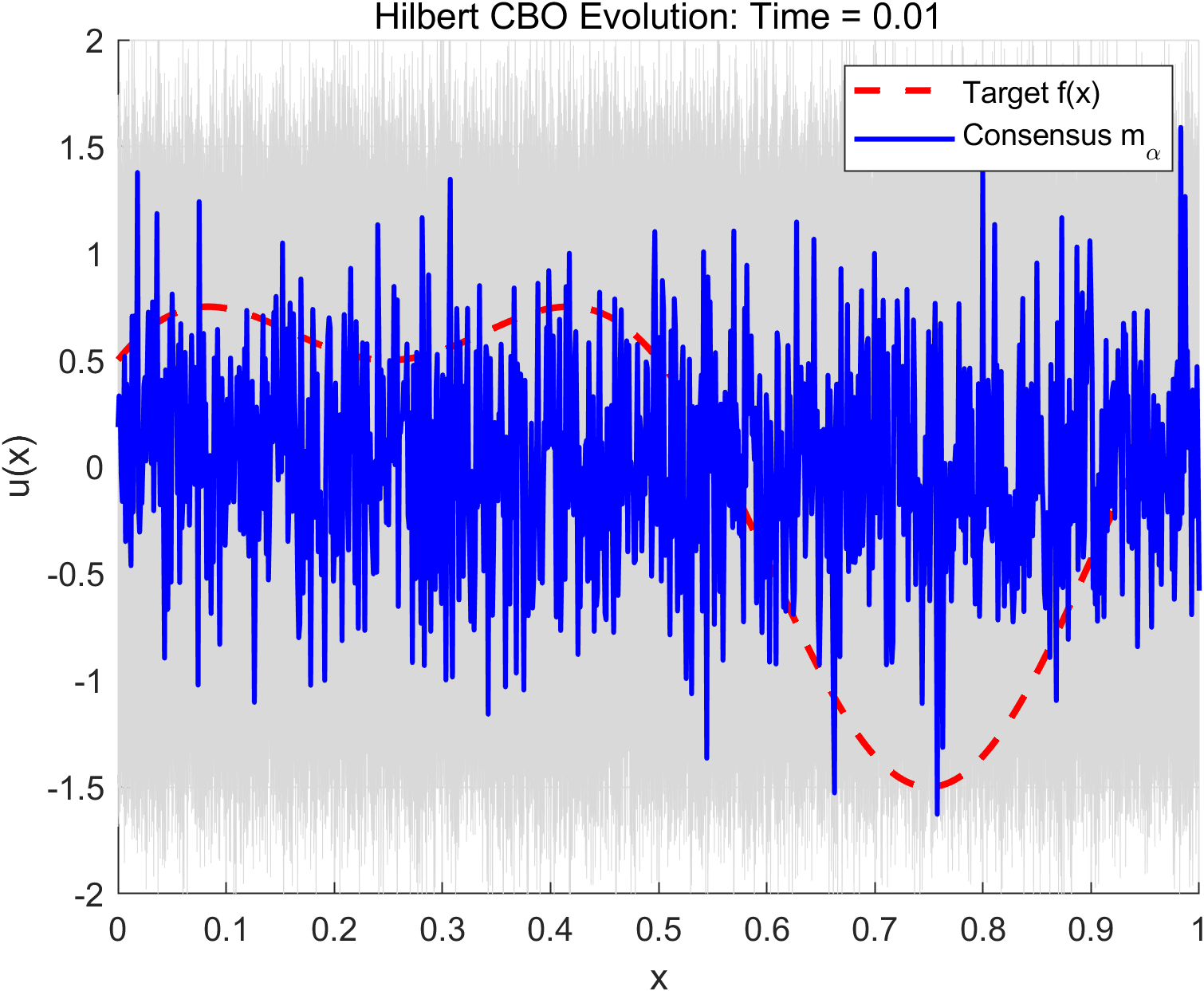}
         \caption{$t = 0.01$ (Initial)}
     \end{subfigure}
     \hfill
     \begin{subfigure}[b]{0.32\textwidth}
         \centering
         \includegraphics[width=\textwidth]{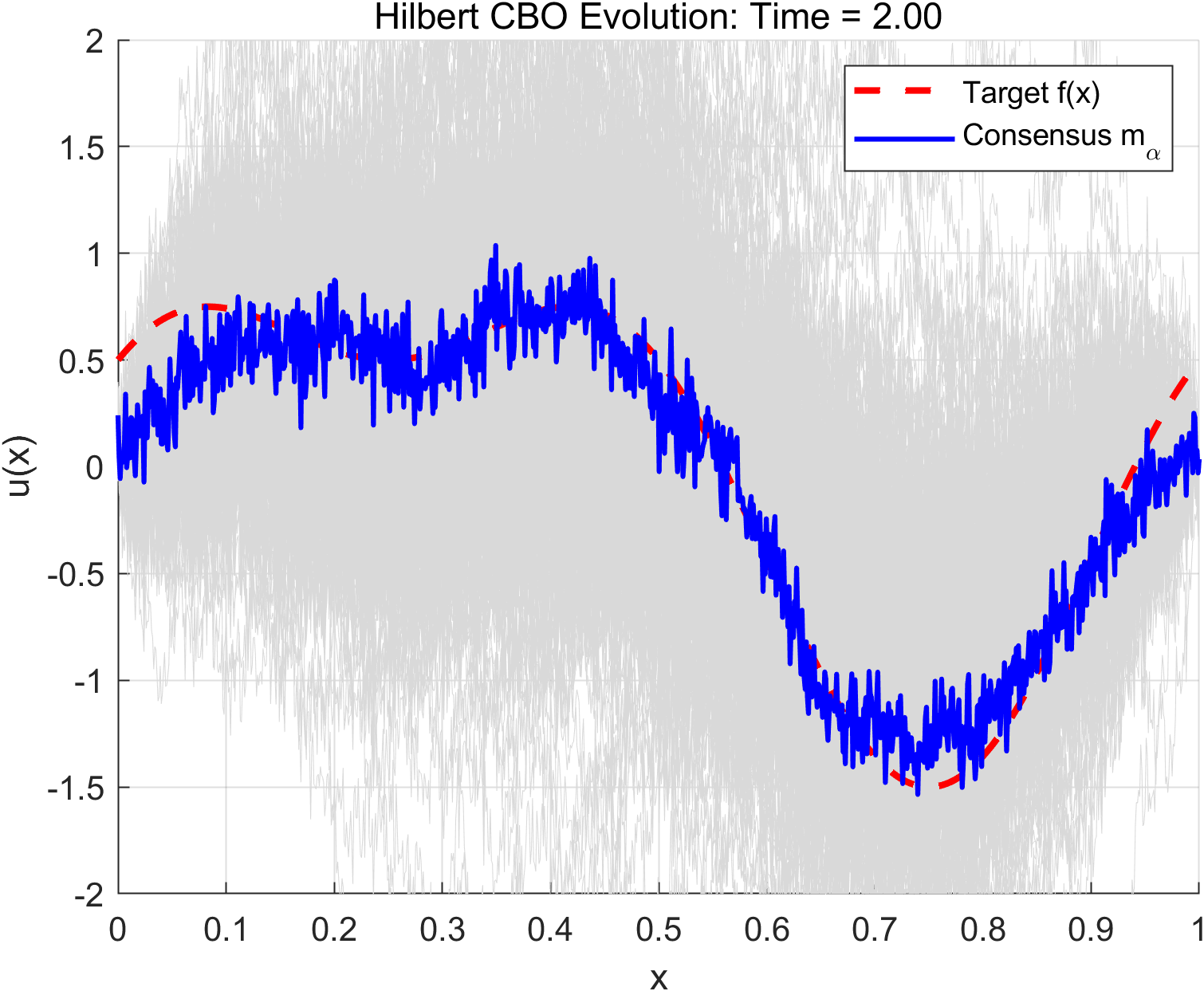}
         \caption{$t = T/2$ (Transition)}
     \end{subfigure}
     \hfill
     \begin{subfigure}[b]{0.32\textwidth}
         \centering
         \includegraphics[width=\textwidth]{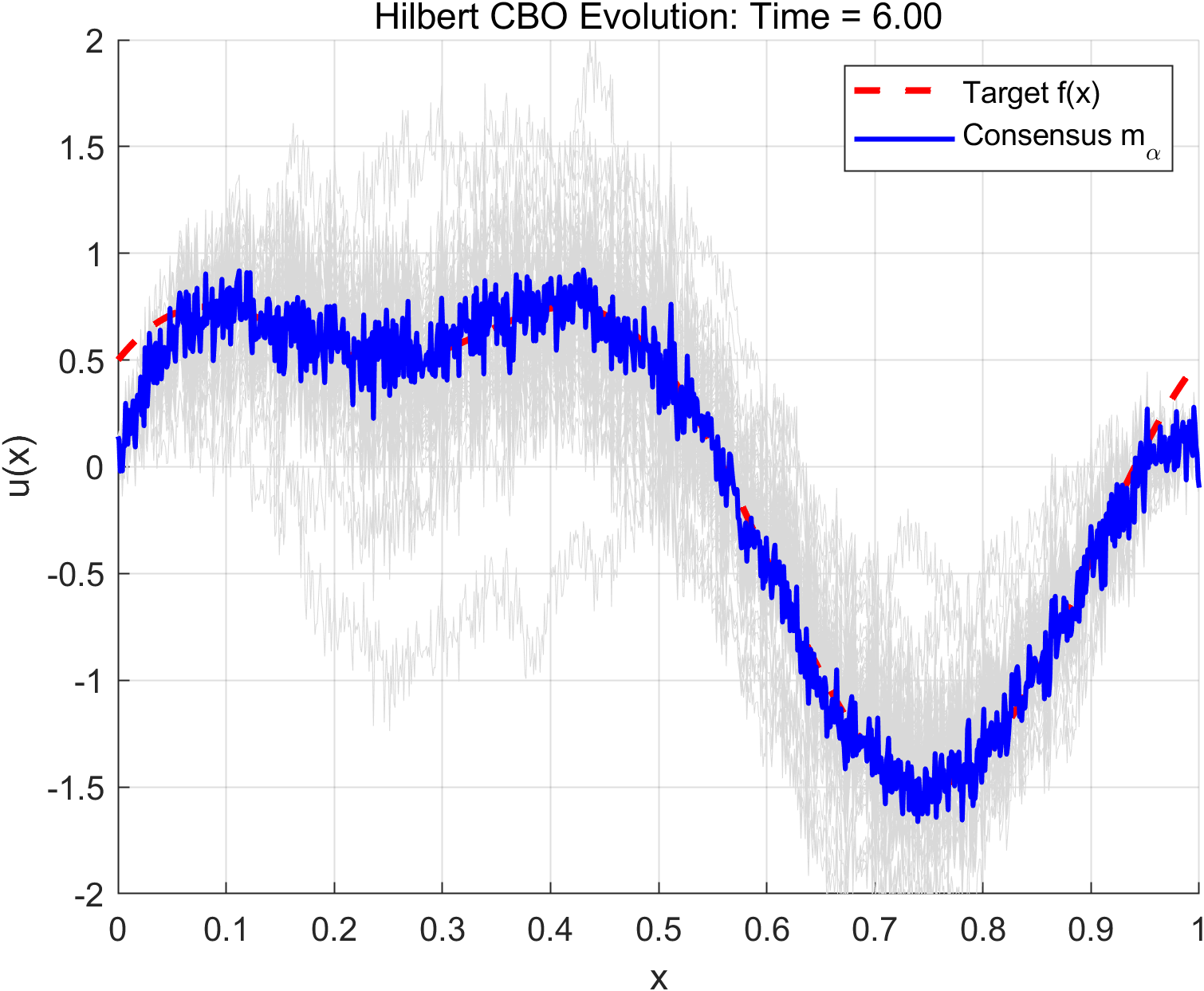}
         \caption{$t = T$ (Consensus)}
     \end{subfigure}
     \caption{Evolution of CBO particle swarm in $L^2([0,1])$. Red dashed: target minimizer $f$. Grey: individual particles. Blue: consensus $\mathfrak{m}_{\alpha}$.}
     \label{fig:convergence_evolution}
\end{figure}

\begin{figure}[H]
    \centering
    \includegraphics[width=0.55\textwidth]{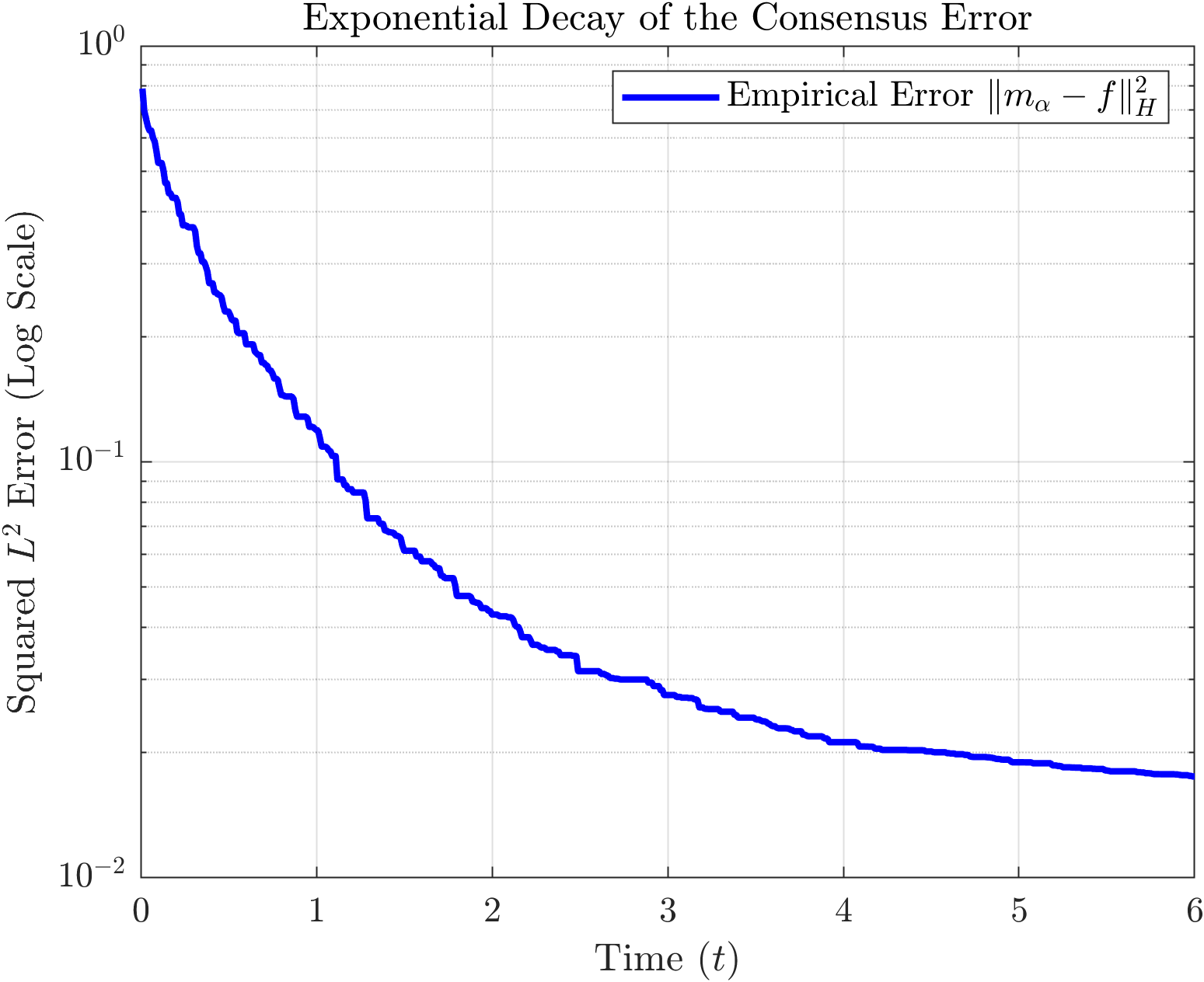}
    \caption{Empirical validation of Theorem \ref{thm:global_convergence_degenerate}. The line represents the squared $L^2$ error, plotted on a logarithmic scale. The initial linear decay confirms the theoretical global exponential convergence. The stabilization at the bottom reflects the finite-particle and time-stepping discretization errors, and spatial truncation (see Remark \ref{rem:subspace_minimizer}).}
    \label{fig:exponential_convergence}
\end{figure}

\section{Conclusion}

In this work, we have developed a Consensus-Based Optimization framework in separable Hilbert spaces, providing a rigorous extension of the classical finite-dimensional model to infinite-dimensional settings. The proposed formulation is based on a stochastic interacting particle system together with its mean-field limit, preserving the characteristic consensus-driven structure of CBO while accommodating the analytical challenges of infinite-dimensional dynamics.

We established well-posedness of the resulting dynamics via a fixed-point argument and analyzed their long-time behavior, showing convergence toward the minimum under suitable assumptions on the objective functional. In addition, we studied the corresponding finite-particle system and derived a practical numerical scheme, supported by illustrative experiments.

The present work provides a first step toward a systematic theory of derivative-free optimization methods in infinite-dimensional spaces. Several directions remain open, including sharper convergence guarantees, extensions to more general diffusion structures, and applications to constrained or structured optimization problems in function spaces.

\section*{Acknowledgment}

This work is partially supported by the starting grant from Hunan University. 
The authors also wish to thank Michael Herty for fruitful discussions, in particular during the visit of H.K. to RWTH Aachen University in Spring 2025, to which he is grateful for its hospitality and support.

\appendix

\section{SDEs in Hilbert Spaces}\label{app:sde in H}

In this appendix, we briefly recall the functional-analytic and probabilistic framework required for stochastic differential equations on separable Hilbert spaces, together with a standard well-posedness result and the corresponding infinite-dimensional It\^o formula. The presentation is restricted to the simplified setting relevant to the present work.

The appendix is organized as follows. We first recall some standard operator norms and the Hilbert--Schmidt framework, before introducing the notion of a $Q$-Wiener process on a separable Hilbert space. We then briefly discuss measurability and predictability, followed by the definition of the stochastic integral with respect to a $Q$-Wiener process. Next, we present the class of Hilbert-space-valued stochastic differential equations considered in this work, together with a standard existence and uniqueness result in the Lipschitz setting. We subsequently state the corresponding infinite-dimensional It\^o formula and conclude with a comparison to the classical finite-dimensional framework.

\subsection{Some Standard Norms}\label{app:norms}
We recall some standard norms used in infinite-dimensional analysis. 

\begin{definition}[Hilbert--Schmidt Norm]\label{def:HS norm}
Let $H$ be a separable Hilbert space with orthonormal basis $(\mathbf{e}_i)_{i \in \mathbb{N}}$. For a bounded linear operator $T \in \mathscr{L}(H)$, the Hilbert--Schmidt norm is defined as
\begin{equation}
    \|T\|_{\mathrm{HS}} := \left( \sum_{i=1}^\infty \|T \mathbf{e}_i\|_H^2 \right)^{1/2}.
\end{equation}
The space of Hilbert--Schmidt operators $\mathscr{L}_2(H) = \{T \in \mathscr{L}(H) : \|T\|_{\mathrm{HS}} < \infty\}$ is a Hilbert space equipped with the inner product
\[
    \langle T, S \rangle_{\mathrm{HS}} = \mathrm{Tr}(S^* T).
\]
\end{definition}

We next recall the relation between the Hilbert--Schmidt norm and the operator norm. The operator norm of $B \in \mathscr{L}(H)$ is defined by
\[
    \|B\|_{\mathrm{op}} := \sup_{\|x\|_H=1} \|Bx\|_H.
\]

Let $Q$ be a positive trace-class operator on $H$. We now introduce the space of $Q$-Hilbert--Schmidt operators,
\[
    \mathscr{L}_Q(H) := \{B \in \mathscr{L}(H) : BQ^{1/2} \in \mathscr{L}_2(H)\},
\]
equipped with the norm
\[
    \|B\|_{\mathscr{L}_Q} := \|BQ^{1/2}\|_{\mathrm{HS}}.
\]
Equivalently, for any orthonormal basis $(\mathbf{e}_i)_{i\in\mathbb{N}}$ of $H$, one has
\[
    \|B\|_{\mathscr{L}_Q}^2 = \sum_{i=1}^\infty \|BQ^{1/2}\mathbf{e}_i\|_H^2.
\]

Using the identity
\[
    \|T\|_{\mathrm{HS}}^2 = \mathrm{Tr}(TT^*), \qquad \text{with } T = BQ^{1/2},
\]
we obtain the equivalent trace representation
\[
    \|B\|_{\mathscr{L}_Q}^2 = \mathrm{Tr}(BQB^*).
\]

We now recall a fundamental estimate relating operator and Hilbert--Schmidt norms
\[
    \|ST\|_{\mathrm{HS}} \le \|S\|_{\mathrm{op}} \, \|T\|_{\mathrm{HS}},
\]
valid for every bounded operator $S$ and every Hilbert--Schmidt operator $T$. 
Applying this inequality with $S = B$ and $T = Q^{1/2}$ yields
\[
    \|BQ^{1/2}\|_{\mathrm{HS}} \le \|B\|_{\mathrm{op}} \, \|Q^{1/2}\|_{\mathrm{HS}}.
\]
Since $\|BQ^{1/2}\|_{\mathrm{HS}} = \|B\|_{\mathscr{L}_Q}$ and, for a positive operator $Q$,
\[
    \|Q^{1/2}\|_{\mathrm{HS}}^2 = \mathrm{Tr}(Q),
\]
we conclude that
\[
    \|B\|_{\mathscr{L}_Q}^2 = \mathrm{Tr}(BQB^*) \le \|B\|_{\mathrm{op}}^2 \, \mathrm{Tr}(Q).
\]

\subsection{The $Q$-Wiener Process}\label{app:noise}

This follows and complement the discussion in \S\ref{subsec:q_wiener}.

Recall that in a finite-dimensional space $\mathbb{R}^d$, a standard Brownian motion can be written as
\[
    W_t = \sum_{i=1}^{d} W_t^{i}\,\mathbf{e}_i,
\]
where $(W_t^{i})_{i=1}^d$ are independent standard Brownian motions and $(\mathbf{e}_i)_{i=1}^d$ is the canonical basis of $\mathbb{R}^d$. In this case, the covariance operator is the identity.

A natural infinite-dimensional analogue would formally correspond to
\[
    W_t = \sum_{i=1}^{\infty} W_t^{i}\,\mathbf{e}_i,
\]
but this expression does not define an $H$-valued random variable. Indeed, by It\^o's isometry,
\[
    \mathbb{E}\|W_t\|_H^2 = t \sum_{i=1}^{\infty} \|\mathbf{e}_i\|_H^2 = \infty.
\]
To obtain an $H$-valued process, one introduces a covariance operator $Q:H\to H$, assumed to be symmetric, positive, and trace-class. Then there exists an orthonormal basis $(\mathbf{e}_i)_{i\in\mathbb{N}}$ of $H$ and eigenvalues $(\lambda_i)_{i\in\mathbb{N}}$ such that
\[
    Q\mathbf{e}_i = \lambda_i \mathbf{e}_i, \qquad \sum_{i=1}^\infty \lambda_i < \infty.
\]
This motivates the definition of a $Q$-Wiener process.

\begin{definition}[See page 165 in \cite{chow2007stochastic}]
A stochastic process $X_t$ in a Hilbert space $H$ is said to be \textit{$H$-valued Gaussian process} if for any $g\in H$, we have $\langle X_t, g\rangle$ is a real-valued Gaussian process. Let $Q$ be a trace-class self-adjoint operator on $K$. 
An $H$-valued stochastic process $\{W_t,\,t\geq 0\}$ is called a Wiener process in $H$ with covariance operator $Q$, or a $Q$-Wiener process, if
\begin{enumerate}
    \item $W_t$ is a continuous process in $H$ for $t\geq 0$ with $W_0 = 0$ a.s.,
    \item $W_t$ has stationary, independent increments, and
    \item $W_t$ is a centered Gaussian process in $K$ with covariance operator $Q$ such that 
    \[
    \mathbb{E}[\langle W_t, g \rangle] = 0, \quad \mathbb{E}[\langle W_t,g \rangle \langle W_s, h \rangle] = (t\wedge s) \langle Qg,h \rangle
    \]
    for any $s,t\in [0,\infty)$, $g,h\in H$
\end{enumerate}
Such $W_t$ is a continuous $\mathscr{F}_t$-adapted $L^2$-martingale in $H$ with local covariation operator $Q_t = t\,Q$ a.s. $\forall\, t\geq 0$. 
\end{definition}

\begin{theorem}(\cite[Theorem 2.5 in Chapter 6, p.165]{chow2007stochastic})
Let $(\mathbf{e}_k)_{k\in\mathbb{N}}$ be an orthonormal basis of eigenvectors of $Q$ with eigenvalues $(\lambda_k)_{k\in\mathbb{N}}$. Then the $Q$-Wiener process admits the representation
\[
    W_{t}^{Q} = \sum_{k=1}^{\infty}\sqrt{\lambda_k}\, \beta_{t}^{k}\,\mathbf{e}_{k},
\]
where $(\beta_t^k)_{k\in\mathbb{N}}$ is a sequence of independent, identically distributed standard Brownian motions in one dimension. Moreover, the series converges uniformly on any finite interval $[0,T]$ with probability one. 
\end{theorem}

Using It\^o's isometry, one obtains
\[
    \mathbb{E}[\,\|W_t^{Q}\|_H^2\,] = t \sum_{i=1}^{\infty} \lambda_i = t\,\mathrm{Tr}(Q),
\]
so that $W_t^Q \in H$ almost surely.

The series representation of the $Q$-Wiener process shows that the noise acts independently along the directions $(\mathbf{e}_k)_{k\in\mathbb{N}}$, with variances $\lambda_k$. Consequently, stochastic integration against $W_t^Q$ requires that the integrand $G(t,x)$ acts in a way that is compatible with this anisotropic covariance structure.

More precisely, if $G(t,x)$ is applied to $W_t^Q$, then formally
\[
    G(t,x)\,W_t^Q = \sum_{k=1}^\infty \sqrt{\lambda_k}\, G(t,x)\,\mathbf{e}_k\, \beta_t^k.
\]
For this series to define an $H$-valued random variable in $L^2(\Omega;H)$, one requires
\[
    \sum_{k=1}^\infty \lambda_k \|G(t,x)\,\mathbf{e}_k\|_H^2 < \infty,
\]
which is equivalent to
\[
    G(t,x)Q^{1/2} \in \mathscr{L}_2(H).
\]
This observation motivates the introduction of the space
\[
    \mathscr{L}_Q(H) := \{B \in \mathscr{L}(H) : BQ^{1/2} \in \mathscr{L}_2(H)\},
\]
which characterizes precisely the class of admissible diffusion operators for stochastic integration with respect to a $Q$-Wiener process.

\subsection{Measurability and Predictability}

All stochastic processes are defined on a filtered probability space $(\Omega,\mathscr{F},(\mathscr{F}_t)_{t\in[0,T]},\mathbb{P})$ satisfying the usual conditions.

A process $X:[0,T]\times\Omega\to H$ is called \emph{adapted} if $X_t$ is $\mathscr{F}_t$-measurable for every $t\in[0,T]$.

A process $X$ is called \emph{predictable} if it is measurable with respect to the predictable $\sigma$-algebra on $[0,T]\times\Omega$, i.e. the $\sigma$-algebra generated by all left-continuous adapted processes.

In the following, all stochastic processes are assumed to be predictable whenever required for the stochastic integration with respect to the $Q$-Wiener process.

\begin{remark}[Localization]\label{rmk:local}
When necessary, we use a standard localization argument. More precisely, we consider an increasing sequence of $(\mathscr{F}_t)$-stopping times $(\tau_n)_{n\in\mathbb N}$ such that $\tau_n \uparrow T$ almost surely, and work with the stopped processes
\[
    X^{\tau_n}_t := X_{t\wedge \tau_n}.
\]
On each interval $[0,\tau_n]$, the process satisfies the required integrability or boundedness properties, and results are then extended to $[0,T]$ by letting $n\to\infty$.
\end{remark}

\subsection{Stochastic Integral}

This subsection relies on \cite[\S 6.3, page 167]{chow2007stochastic}. One can also refer to \cite[Chapter 2]{prevot2007concise} or \cite[Chapter 2]{liu2015stochastic}.

Let $H$ be a real separable Hilbert space and let $(W_t^{Q})_{t\geq 0}$ be a $Q$-Wiener process in $H$ defined in a complete probability space $(\Omega, \mathscr{F}, \mathbb{P})$ with a filtration $(\mathscr{F}_{t})_{t\geq 0}$ of increasing $\sigma$-fields of $\mathscr{F}$. 
The stochastic integral with respect to $W_t^{Q}$ is first defined for elementary predictable processes of the form (see \cite[equ.(6.17), page 168]{chow2007stochastic})
\[
    \Phi(t) = \sum_{k=1}^n \Phi_k \mathbf{1}_{(t_{k-1},t_k]}(t), \qquad \Phi_k \in \mathscr L_Q(H),
\]
by setting
\[
    \int_0^t \Phi(s)\,\dd W_s^Q := \sum_{k=1}^n \Phi_k \bigl(W_{t_k \wedge t}^Q - W_{t_{k-1} \wedge t}^Q\bigr).
\]

For a general predictable\footnote{See \cite[page 170 of \S 6.3]{chow2007stochastic}.} process  $\Phi:[0,T]\times \Omega \to \mathscr L_Q(H)$ satisfying
\[
    \mathbb{E}\left[\int_0^T \|\Phi(s)\|_{\mathscr L_Q}^2 \,\dd s\right] < \infty,
\]
the stochastic integral is defined as the $L^2(\Omega;H)$-limit of such elementary processes (see \cite[Theorem 3.4, page 171]{chow2007stochastic} and \cite[Theorem 3.6, page 173]{chow2007stochastic}).

Moreover, the Itô isometry holds
\begin{equation}\label{eq: ito isom}
    \mathbb{E}\left[\, \left\| \int_0^t \Phi(s)\,\dd W_s^Q \right\|_H^2\,\right]
    =
    \mathbb{E}\left[ \int_0^t \|\Phi(s)\|_{\mathscr L_Q}^2 \,\dd s\right].
\end{equation}

\subsection{Stochastic Differential Equations}\label{subsec: SDE H}

The family of equations we are concerned with takes the form
\[
    \dd X_t = -\lambda\,X_t\,\dd t + F(t,X_t)\,\dd t + G(t,X_t)\,\dd W_{t}^{Q}, \qquad X_0 = \xi
\]
Since the operator $-\lambda \mathrm{I}$ where $\mathrm{I}:H\to H$ is the identity map, is bounded on $H$, it may be absorbed into the drift term. 
Moreover, the distinction between \textit{mild} and \textit{strong} solutions is unnecessary in the present setting, and the stochastic evolution equation may be treated as a standard Hilbert-space-valued stochastic differential equation.

Let $H$ be a separable Hilbert space and let $(W_t^Q)_{t\geq0}$ be an $H$-valued $Q$-Wiener process with covariance operator $Q$. We consider the stochastic differential equation
\[
    \dd X_t = F(t,X_t)\,\dd t + G(t,X_t)\,\dd W_t^Q, \qquad X_0=\xi,
\]
where
\[
    F:[0,T]\times H\to H, \qquad G:[0,T]\times H\to \mathscr L_Q(H).
\]
The space $\mathscr{L}_Q(H)$ ensures that the stochastic integral with respect to the $Q$-Wiener process is well defined as an $H$-valued random variable. This additional regularity assumption on the diffusion coefficient is necessary in the infinite dimensional setting in order for the stochastic integral to be well defined. Indeed, $G(t,x)\in \mathscr L_Q(H)$ is equivalent to
\[
    G(t,x)Q^{1/2}\in \mathscr{L}_2(H), \qquad (t,x)\in [0,T]\times H,
\]
where $\mathscr{L}_2(H)$ denotes the space of Hilbert--Schmidt operators on $H$, see Definition \ref{def:HS norm}. This Hilbert--Schmidt condition ensures that the covariance of the noise remains trace class and guarantees that the stochastic integral
\[
    \int_0^t G(s,X_s)\,\dd W_s^Q
\]
is well defined as an $H$-valued random variable, and where the stochastic integral is defined as above.


\begin{definition}[Strong solution]\label{def:solution in H}
Let $H$ be a separable Hilbert space,  $(W_t^Q)_{t\geq0}$  an $H$-valued $Q$-Wiener process, and let $F:[0,T]\times H \to H,$ and $G:[0,T]\times H \to \mathscr L_Q(H).$ 
An $H$-valued stochastic process $X=(X_t)_{t\in[0,T]}$ is called a strong solution of the stochastic differential equation
\[
    \dd X_t = F(t,X_t)\,\dd t + G(t,X_t)\,\dd W_t^Q, \qquad X_0=\xi,
\]
if the following hold
\begin{itemize}
\item $X$ is adapted (equivalently predictable) and has continuous paths $\mathbb{P}$-a.s.,
\item $X$ satisfies the integral equation $\mathbb{P}$-a.s. for all $t\in[0,T]$:
\[
    X_t = \xi + \int_0^t F(s,X_s)\,\dd s + \int_0^t G(s,X_s)\,\dd W_s^Q,
\]
\item $X \in L^2(\Omega;C([0,T];H))$, i.e.
\[
    \mathbb{E}\left[\sup_{t\in[0,T]} \|X_t\|_H^2\right] < \infty.
\]
\end{itemize}
\end{definition}

Note that in the present setting, instead of ``\textit{strong solution}'', one could simply say a unique adapted process $X \in L^2(\Omega;C([0,T];H))$.

\begin{remark}
In the general framework of stochastic evolution equations (see \cite[\S 6.7, page 197]{chow2007stochastic}), the notion of strong solution is formulated via an energy equation involving duality pairings. See \cite[equ.(6.118), page 198]{chow2007stochastic}. \\
In the present setting, since all operators take values in $H$ and no unbounded generator is present, the energy formulation is equivalent to the standard integral formulation. 
Indeed, assume that $X$ satisfies the energy equation
\[
    \langle X_t, \varphi \rangle = \langle \xi, \varphi \rangle + \int_0^t \langle F(s,X_s), \varphi \rangle\,\dd s + \int_0^t \langle G(s,X_s)\,\dd W_s^Q, \varphi \rangle, \quad \forall \varphi \in H.
\]
Define
\[
    Y_t := X_t - \xi - \int_0^t F(s,X_s)\,\dd s - \int_0^t G(s,X_s)\,\dd W_s^Q.
\]
Then $\langle Y_t,\varphi\rangle = 0$ for all $\varphi \in H$, which implies $Y_t=0$ in $H$. Hence $X$ satisfies the integral formulation
\[
    X_t = \xi + \int_0^t F(s,X_s)\,\dd s + \int_0^t G(s,X_s)\,\dd W_s^Q.
\]
Therefore, we adopt the latter as the definition of strong solution.
\end{remark}

\begin{remark}\label{rmk:mild}
In the general theory of stochastic evolution equations, mild solutions are defined via the semigroup generated by an (possibly unbounded) operator $A$. \\
In the present setting, no unbounded operator is present, so that the mild formulation reduces to
\[
    X_t = \xi + \int_0^t F(s,X_s)\,\dd s + \int_0^t G(s,X_s)\,\dd W_s^Q,
\]
which coincides with the strong (integral) formulation introduced above. In this sense, the notions of mild and strong solutions are equivalent in our framework. 
For a thorough discussion on the several notions of solutions, we refer to \cite[Item 3 in Remark 4.2.2, page 90]{liu2015stochastic} or \cite[Item 3 in Remark 4.2.2, page 74]{prevot2007concise}. See also \cite[\S1.1]{liu2015stochastic}.
\end{remark}


The following well-posedness result is \cite[Theorem 7.4, page 205]{chow2007stochastic} where we only need assumptions \cite[(C1)-(C2), page 205]{chow2007stochastic}. Indeed, the assumptions \cite[(B1)-(B3), page 198]{chow2007stochastic} concern the operator $A$ which we have absorbed in $F$ (see \S \ref{subsec: SDE H}) as it is essentially the identity map. Additionally, the quadratic bound in the following theorem is a consequence of \cite[Theorem 6.4, page 186]{chow2007stochastic} which concerns \textit{mild solutions} (but these are equivalent to our \textit{strong solutions} as discussed in Remark \ref{rmk:mild}). Furthermore, the initial discussion being random is the object of the remark in \cite[page 191]{chow2007stochastic}.

\begin{theorem}[Existence and uniqueness of strong solutions]
\label{thm:solution}
Let $H$ be a separable Hilbert space and let $(W_t^Q)_{t\ge0}$ be an $H$-valued $Q$-Wiener process. Consider the stochastic differential equation
\[
    \dd X_t = F(t,X_t)\,\dd t + G(t,X_t)\,\dd W_t^Q, \qquad X_0=\xi,
\]
where $F:[0,T]\times H\to H$, and $G:[0,T]\times H\to \mathscr L_Q(H).$ 
Assume that there exists a constant $L>0$ such that, for all
$t\in[0,T]$ and $x,y\in H$,
\begin{align*}
    \|F(t,x)-F(t,y)\|_H &\le L\|x-y\|_H, \\
    \|G(t,x)-G(t,y)\|_{\mathscr L_Q} &\le L\|x-y\|_H,
\end{align*}
and
\[
    \|F(t,x)\|_H^2 + \|G(t,x)\|_{\mathscr L_Q}^2 \le L(1+\|x\|_H^2).
\]
Let the initial condition $\xi$ be $\mathscr{F}_0$-measurable and satisfy $\mathbb{E}[\|\xi\|_H^2] < \infty.$ \\
Then there exists a unique predictable process $X\in L^2(\Omega;C([0,T];H))$  with continuous sample paths almost surely such that
\[
    X_t = \xi + \int_0^t F(s,X_s)\,\dd s + \int_0^t G(s,X_s)\,\dd W_s^Q, \qquad \mathbb{P}\text{-a.s.}, \quad \text{ for all } t\in[0,T].
\]
Moreover, there exists a constant $C>0$, depending only on $T$ and $L$, such that
\[
    \mathbb{E}\Big[ \sup_{t\in[0,T]} \|X_t\|_H^2 \Big] \le C\Big( 1+\mathbb{E}[\|\xi\|_H^2] \Big).
\]
\end{theorem}

\begin{remark}[Pathwise uniqueness]\label{rmk:uniq}
Uniqueness in Theorem \ref{thm:solution} is understood in the pathwise sense: if $X$ and $Y$ are two solutions defined on the same filtered probability space, driven by the same $Q$-Wiener process and with the same initial condition $\xi$, then
\[
    \mathbb P\big(X_t = Y_t \text{ for all } t\in[0,T]\big)=1.
\]
\end{remark}

\begin{remark}
The linear growth assumption may be replaced by the weaker condition
\[
    \int_0^T \Big( \|F(t,0)\|_H^2 + \|G(t,0)\|_{\mathscr L_Q}^2 \Big)\,\dd t <\infty.
\]
Indeed, under the global Lipschitz assumption,
\[
    \|F(t,x)\|_H \le \|F(t,x)-F(t,0)\|_H + \|F(t,0)\|_H \le L\|x\|_H + \|F(t,0)\|_H,
\]
and similarly for $G$. Consequently,
\[
    \|F(t,x)\|_H^2 + \|G(t,x)\|_{\mathscr L_Q}^2 \le C\Big( \|x\|_H^2 + \|F(t,0)\|_H^2 + \|G(t,0)\|_{\mathscr L_Q}^2 \Big),
\]
for some constant $C>0$. This weaker formulation is commonly used in the literature. See \cite[page 205]{chow2007stochastic}.
\end{remark}


\subsection{It\^{o} Formula}

Let us begin by recalling first and second order Fr\'echet derivatives. More details can be found for example in \cite[\S 2.2]{bonnans2013perturbation}

\begin{definition}[First Fr\'echet derivative]
Let $H$ be a Hilbert space and let $\Phi:H\to\mathbb{R}$. We say that $\Phi$ is Fr\'echet differentiable at $x\in H$ if there exists a bounded linear functional
\[
D\Phi(x)\in \mathscr L(H,\mathbb{R})
\]
such that
\[
    \lim_{\|h\|_H\to0} \frac{|\Phi(x+h)-\Phi(x)-D\Phi(x)[h]|}{\|h\|_H} = 0.
\]
Since $\mathscr L(H,\mathbb{R}) \simeq H^*$, 
the Riesz representation theorem allows one to identify $D\Phi(x)$ with an element of $H$, still denoted by $D\Phi(x)$. Hence, for every $h\in H$,
\[
    D\Phi(x)[h] = \langle D\Phi(x),h\rangle_H.
\]
Equivalently,
\[
    D\Phi(x)[h] = \lim_{\varepsilon\to0} \frac{\Phi(x+\varepsilon h)-\Phi(x)}{\varepsilon}  = \left. \frac{\dd}{\dd\varepsilon} \Phi(x+\varepsilon h) \right|_{\varepsilon=0},
\]
which corresponds to the directional derivative of $\Phi$ at $x$ in the direction $h$.
\end{definition}

\begin{definition}[Second Fr\'echet derivative]
Let $H$ be a Hilbert space and let $\Phi:H\to\mathbb{R}$ be Fr\'echet differentiable. We say that $\Phi$ is twice Fr\'echet differentiable at $x\in H$ if the map
\[
    D\Phi:H\to \mathscr L(H,\mathbb{R})
\]
is itself Fr\'echet differentiable at $x$. The second derivative of $\Phi$ at $x$, denoted by $D^2\Phi(x)$, 
belongs to $\mathscr L(H,\mathscr L(H,\mathbb{R})).$
Using the canonical identification
\[
    \mathscr L(H,\mathscr L(H,\mathbb{R})) \simeq \mathscr B(H\times H;\mathbb{R}),
\]
the space of bounded bilinear forms\footnote{Equipped with the norm $\|B\| := \sup_{\|h\|_H\le 1,\ \|k\|_H\le 1} |B(h,k)|$.}, the second derivative may equivalently be viewed as a bounded bilinear form on $H\times H$. Thus, for $h,k\in H$,
\[
    D^2\Phi(x)(h,k) = \left. \frac{\partial^2}{\partial s\,\partial t} \Phi(x+sh+tk) \right|_{s=t=0}.
\]

Moreover, there exists a constant $C>0$ such that
\[
    |D^2\Phi(x)(h,k)| \le C\,\|h\|_H\,\|k\|_H, \qquad h,k\in H.
\]
\end{definition}

We shall now define It\^{o} formula in Hilbert spaces adapted to our setting. A more general treatment can be found in \cite[\S 6.4, page 176]{chow2007stochastic}, in particular \cite[Theorem 4.2, page 177]{chow2007stochastic}. See also \cite[Theorem 6.1.1, page 180]{liu2015stochastic}.

\begin{theorem}[It\^{o} formula in Hilbert spaces]
\label{thm:Ito in H}
Let $X$ be the strong solution of
\[
    \dd X_t = F(t,X_t)\,\dd t + G(t,X_t)\,\dd W_t^Q, \qquad X_0=\xi,
\]
and let $\Phi:H\to\mathbb{R}$ be twice Fr\'echet differentiable with bounded derivatives up to order two, i.e. $\Phi \in C_b^2(H)$. 
Then, for all $t\in[0,T]$, the following identity holds $\mathbb{P}$-almost surely:
\begin{align*}
    \Phi(X_t)
    & = \Phi(\xi)
    + \int_0^t D\Phi(X_s)\big[F(s,X_s)\big]\,\dd s \\
    &\qquad + \frac{1}{2} \int_0^t \mathrm{Tr}\Big(
    G(s,X_s)\,Q\,G(s,X_s)^* \, D^2\Phi(X_s)
    \Big)\,\dd s \\
    &\qquad \qquad + \int_0^t D\Phi(X_s)\big[G(s,X_s)\,\dd W_s^Q\big].
\end{align*}
\end{theorem}

\begin{remark}
The second-order term in the infinite-dimensional It\^o formula involves the covariance structure of the noise together with the second Fr\'echet derivative of the test functional. Since $GQ^{1/2}\in\mathscr L_2(H)$, 
the operator $GQG^*=(GQ^{1/2})(GQ^{1/2})^*$ 
is trace class. Consequently, the trace term $\mathrm{Tr}\big(GQG^*D^2\Phi(x)\big)$ is well defined. 
More explicitly, if $(\mathbf{e}_k)_{k\in\mathbb N}$ is an orthonormal basis of $H$, then
\[
    \mathrm{Tr}\big(GQG^*D^2\Phi(x)\big) = \sum_{k=1}^\infty D^2\Phi(x)\big( GQ^{1/2}\mathbf{e}_k,\, GQ^{1/2}\mathbf{e}_k \big).
\]

This is the infinite-dimensional analogue of the finite-dimensional contraction
\[
    \sum_{i,j} a_{ij}\,\partial_{ij}\Phi,
\]
where the covariance matrix $a$ is replaced by the covariance operator $GQG^*$ and the Hessian matrix by the second Fr\'echet derivative $D^2\Phi$.

The convergence of the series follows from the boundedness of $D^2\Phi(x)$ together with the Hilbert--Schmidt property $GQ^{1/2}\in\mathscr L_2(H)$.
\end{remark}

\begin{ex}\label{ex:Ito in H}
As an illustration of It\^{o} formula, let us consider the functional
\[
    \Phi:H\to\mathbb{R}, \qquad \Phi(x)=\|x\|_H^2.
\]
Its first Fr\'echet derivative is given by
\[
    D\Phi(x)[h] = 2\langle x,h\rangle_H, \qquad h\in H,
\]
while the second Fr\'echet derivative is
\[
    D^2\Phi(x)(h,k) = 2\langle h,k\rangle_H, \qquad h,k\in H.
\]
Applying the It\^o formula to a strong solution of $\dd X_t = F(t,X_t)\,\dd t + G(t,X_t)\,\dd W_t^Q,$ 
yields
\begin{align*}
\|X_t\|_H^2
    &= \|\xi\|_H^2 + 2\int_0^t \langle X_s,F(s,X_s)\rangle_H\,\dd s + \int_0^t \mathrm{Tr}\big(G(s,X_s)\,Q\,G(s,X_s)^*\big)\,\dd s \\
    &\qquad \qquad \qquad \qquad  + 2\int_0^t \langle X_s,G(s,X_s)\,\dd W_s^Q\rangle_H.
\end{align*}
Indeed, since $D^2\Phi(x)=2I$,  the trace term becomes
\[
    \frac{1}{2} \mathrm{Tr}\big( GQG^*D^2\Phi(x) \big) = \mathrm{Tr}(GQG^*).
\]
This identity is the infinite-dimensional analogue of the classical energy estimate for finite-dimensional stochastic differential equations.
Moreover, using the identity
\[
    \mathrm{Tr}(GQG^*) = \|GQ^{1/2}\|_{\mathrm{HS}}^2 = \|G\|_{\mathscr L_Q}^2,
\]
the previous formula may equivalently be written as
\begin{align*}
    \|X_t\|_H^2
    &= \|\xi\|_H^2 + 2\int_0^t \langle X_s,F(s,X_s)\rangle_H\,\dd s + \int_0^t \|G(s,X_s)\|_{\mathscr L_Q}^2\,\dd s +2\int_0^t\langle X_s,G(s,X_s)\,\dd W_s^Q\rangle_H.
\end{align*}
\end{ex}

\begin{remark}[Kolmogorov and Fokker--Planck formulation]
Let $\mu_t := \mathcal L(X_t)$ denote the law of the solution to the Hilbert-space valued SDE. The It\^o formula implies that for every sufficiently regular test function $\Phi:H\to\mathbb{R}$,
\[
    \frac{\dd}{\dd t}\mathbb{E}[\Phi(X_t)] = \mathbb{E}[\mathcal L \Phi(X_t)],
\]
where $\mathcal L$ denotes the Kolmogorov operator
\[
    \mathcal L \Phi(x) = D\Phi(x)[F(x)] + \frac{1}{2} \mathrm{Tr}\big(G(x)QG(x)^* D^2\Phi(x)\big).
\]
A more general treatment can be found in \cite[Chapter 9, page 271]{chow2007stochastic}.

This identity characterizes the evolution of the law $\mu_t$ in a weak (dual) sense, namely
\[
    \frac{\dd}{\dd t}\int_H \Phi(x)\,\mu_t(\dd x) = \int_H \mathcal L \Phi(x)\,\mu_t(\dd x),
\]
for all admissible test functions $\Phi$. Equivalently, this can be interpreted as the formal Fokker--Planck equation
\[
    \partial_t \mu_t = \mathcal L^* \mu_t,
\]
understood as an evolution equation on probability measures on $H$, dual to the Kolmogorov equation acting on observables. It follows directly from It\^o's formula applied to a (strong) solution of the SDE, provided that the solution possesses sufficient integrability to justify taking expectations, that is
\[
    \mathbb{E}\Big[\sup_{t\in[0,T]}\|X_t\|_H^2\Big] + \mathbb{E}\int_0^T \|F(t,X_t)\|_H\,\dd t + \mathbb{E}\int_0^T \|G(t,X_t)\|_{\mathscr L_Q}^2\,\dd t < \infty.
\]
We refer to \cite[Chapter 10, page 403]{bogachev2022fokker} and references therein for a detailed analysis. 
\end{remark}

\subsection{Comparison with the Finite-Dimensional Setting}

The transition from the finite-dimensional It\^o formula in $\mathbb{R}^d$ to its infinite-dimensional counterpart on a separable Hilbert space $H$ may be understood as a structural extension in which finite sums over coordinates are replaced by traces of operators. The following comparison highlights the correspondence between the classical Euclidean framework and the Hilbert-space setting considered in this work.

Let $(W_t)_{t\ge0}$ be a standard Brownian motion in $\mathbb{R}^d$, and let $(W_t^Q)_{t\ge0}$ be a $Q$-Wiener process on $H$. For a sufficiently smooth functional $\Phi$, the following analogies hold:

\begin{table}[H]
\centering
\begin{tabular}{@{}lll@{}}
\toprule
\textbf{Feature} & \textbf{Finite Dimensional ($\mathbb{R}^d$)} & \textbf{Infinite Dimensional ($H$)} \\ \midrule
State equation &
$\dd X_t = b_t\,\dd t + \sigma_t\,\dd W_t$
&
$\dd X_t = A_t\,\dd t + B_t\,\dd W_t^Q$
\\[0.3em]

First derivative &
$\nabla \Phi(x)\in\mathbb{R}^d$
&
$D\Phi(x)\in H$
\\[0.3em]

Second derivative &
$D^2\Phi(x)\in\mathbb{R}^{d\times d}$
&
$D^2\Phi(x)\in\mathscr B(H\times H;\mathbb R)$
\\[0.3em]

Diffusion coefficient \; &
Matrix $\sigma_t$
&
Operator $B_t \in \mathscr{L}_Q(H)$
\\[0.3em]

Noise covariance &
Identity matrix $I$
&
Trace-class operator $Q$
\\ \bottomrule
\end{tabular}
\end{table}

The corresponding It\^o formulas take the form

\begin{itemize}
\item \textbf{Finite-dimensional setting $(\mathbb{R}^d)$:}
\[
    \dd \Phi(X_t) = \langle \nabla \Phi(X_t), b_t\rangle\,\dd t + \langle \nabla \Phi(X_t), \sigma_t\,\dd W_t\rangle + \frac{1}{2} \sum_{i,j=1}^d (\sigma_{t} \sigma_{t}^\top)_{ij} \,\partial_{ij}\Phi(X_t)\,\dd t.
\]

\item \textbf{Infinite-dimensional setting $(H)$:}
\[
    \dd \Phi(X_t) = D\Phi(X_t)[A_t]\,\dd t + D\Phi(X_t)[B_t\,\dd W_t^Q] + \frac{1}{2} \mathrm{Tr} \big( B_tQB_t^*D^2\Phi(X_t) \big)\,\dd t.
\]
\end{itemize}

The key structural difference lies in the second-order correction term. In finite dimensions, the Hessian contribution is expressed through a finite double sum over coordinates, whereas in Hilbert spaces it is represented through a trace pairing between the covariance operator and the second Fr\'echet derivative.

Indeed, if $(\mathbf{e}_k)_{k\in\mathbb N}$ is an orthonormal basis of $H$, then
\[
    \mathrm{Tr}\big(BQB^*D^2\Phi(x)\big) = \sum_{k=1}^\infty D^2\Phi(x)\big( BQ^{1/2}\mathbf{e}_k,\, BQ^{1/2}\mathbf{e}_k \big),
\]
which is the infinite-dimensional analogue of the finite-dimensional contraction
\[
    \sum_{i,j=1}^d (\sigma\sigma^\top)_{ij}\, \partial_{ij}\Phi.
\]

In the Euclidean setting, the covariance operator is typically the identity, corresponding to isotropic white noise. In infinite dimensions, however, the identity operator is no longer trace class, and one must instead introduce a positive trace-class covariance operator $Q$. The resulting $Q$-Wiener process therefore represents a spatially correlated (colored) noise whose covariance structure ensures that the stochastic process remains $H$-valued.

This illustrates the difference between \textit{white} and \textit{colored} noise in infinite dimensions:
\begin{itemize}
    \item In finite dimensions (or formally with $Q = I$), all directions have the same variance and the noise is isotropic (white noise).
    \item In infinite dimensions, one must introduce a trace-class covariance operator $Q$, which assigns different weights $\lambda_i$ to different directions, leading to anisotropic (colored) noise and ensuring that the process is $H$-valued.
\end{itemize}

This correspondence becomes particularly transparent for the functional $\Phi(x)=\|x\|_H^2$, for which $D^2\Phi(x)=2I$.  
In finite dimensions, the It\^o correction term becomes
\[
    \mathrm{Tr}(\sigma\sigma^\top)\,\dd t = \|\sigma\|_F^2\,\dd t,
\]
where $\|\cdot\|_F$ is the Frobenius norm. In the Hilbert-space setting, this generalizes to
\[
    \mathrm{Tr}(BQB^*)\,\dd t = \|BQ^{1/2}\|_{\mathrm{HS}}^2\,\dd t,
\]
which is  the Hilbert--Schmidt norm associated with the covariance operator $Q$.

Consequently, the finite-dimensional It\^o formula may be viewed as a particular case of the Hilbert-space framework in which the underlying basis is finite and the covariance operator reduces to the identity.

%
%

\section{Elements from nonsmooth analysis}\label{app:nonsmooth}

In this section, we recall some definitions and results from nonsmooth analysis which are relevant to the present work, and more generally to the study of functions defined on infinite dimensional Hilbert spaces. 
These are borrowed from \cite[Chapter 2]{clarke1990optimization}.

Let us consider $X$ a Banach space whose elements are denoted by $x$ with a norm $\|x\|$. A function $f:X\to \mathbb{R}$ is said to be Lipschitz if there exists some nonnegative scalar $K$ such that 
\begin{equation*}
    |f(x) - f(y)| \leq K\, \|x-y\| \quad \forall\, x,y\in X.
\end{equation*}
We say that $f$ is Lipschitz (of rank $K$) \textit{near} $x$ if, for some $\varepsilon>0$, $f$ is Lipschitz within an $\varepsilon$-neighbourhood of $x$, i.e. on the set $x+\varepsilon\,B$ where $B$ is the open unit ball in $X$.

\begin{definition}
Let $f$ be Lipschitz near a given point $x$, and let $v$ be any other element in $X$. The generalized directional derivative of $f$ at $x$ in the direction $v$, denoted $f^{\circ}(x;v)$ is defined as follows
\begin{equation*}
    f^{\circ}(x;v) = \limsup\limits_{y\to x,\, t\downarrow 0} \frac{f(y + t\,v) - f(y)}{t}
\end{equation*}
where $y$ is an element in $X$ and $t$ is a positive scalar
\end{definition}

\begin{proposition}
Let $f$ be Lipschitz of rank $K$ near $x$. Then the function $v\mapsto f^{\circ}(x;v)$ is finite, positively homogeneous,  subadditive on $X$, and satisfies
\begin{equation*}
    |f^{\circ}(x;v)| \leq K\,\|v\|.
\end{equation*}
It is moreover Lipschitz of rank $K$ on $X$.
\end{proposition}

\begin{definition}\label{def: general grad}
Let $f$ be Lipschitz near a given point $x$. The generalized gradient of $f$ at $x$, denoted by $\partial f(x)$, is a subset of $X^{*}$ (the dual space of $X$) of continuous linear functionals on $X$, given by
\begin{equation*}
    \partial f(x) = \{\xi\in X^{*}\,:\, f^{\circ}(x;v) \geq \langle \xi,v \rangle \, \text{ for all } x\in X\}
\end{equation*}
where $\langle\,\cdot,\cdot\,\rangle$ is the duality product on $X$. We denote by $\|\xi\|_{*}$ the norm in $X^{*}$ given by
\begin{equation*}
    \|\xi\|_{*} := \sup\{\langle \xi, x\rangle\,:\, v\in X, \|v\|\leq 1\}.
\end{equation*}
\end{definition}

\begin{proposition}
Let $f$ be Lipschitz of rank $K$ near $x$. Then
\begin{enumerate}
    \item $\partial f(x)$ is a nonempty, convex, weak-$*$ compact subset of $X^{*}$, and $\|\xi\|_{*}\leq K$ for every $\xi \in \partial f(x)$.
    \item For every $v\in X$, one has
    \begin{equation*}
        f^{\circ}(x;v) = \max\{\langle \xi,v\rangle\,:\, \xi \in \partial f(x)\}.
    \end{equation*}
    \item $\xi \in \partial f(x)$ if and only if $f^{\circ}(x;v)\geq \langle \xi,v\rangle$ for all $v\in X$.
    \item Let $x_{i}$ and $\xi_i$ be sequences in $X$ and $X^{*}$ such that $\xi_i \in \partial f(x_i)$. Suppose that $x_i$ converges to $x$, and that $\xi$ is a cluster point of $\xi_i$ in the weak-$*$ topology. Then one has $\xi\in \partial  f(x)$. That is, the multifunction $\partial f$ is weak-$*$ closed.
\end{enumerate}
\end{proposition}

\begin{theorem}[Lebourg Mean-Value Theorem] \label{thm:MV}
Let $x$ and $y$ be points in $X$, and suppose that $f$ is Lipschitz on an open set containing the line segment $[x,y]$. Then there exists a point $u$ in $(x,y)$ such that
\begin{equation*}
    f(y) - f(x) \in \langle \partial f(u), y-x \rangle.
\end{equation*}
\end{theorem}

\begin{proposition}
Let $U$ be an open convex subset of $X$. 
When $f$ is convex on $U$ and Lipschitz near $x$, then $\partial f(x)$ coincides with the subdifferential at $x$ in the sense of convex analysis, and $f^{\circ}(x;v)$ coincides with the directional derivative $f'(x;v)$ for each $v$.
\end{proposition}

Let us now consider $F$ that maps $X$ to another Banach space $Y$. The usual one-sided directional derivative of $F$ at $x$ in the direction $v$ is 
\begin{equation*}
    F'x;v) := \lim\limits_{t\downarrow 0} \frac{F(x+tv) - F(x)}{t}
\end{equation*}
when this limits exists. 

\begin{definition}
$F$ is said to admit a Gateaux derivative at $x$, an element in the space $\mathscr{L}(X,Y)$ of continuous linear functionals from $X$ to $Y$, denoted $DF(x)$, provided that for every $v$ in $X$, $F'(x;v)$ exists and equals $\langle DF(x),v\rangle$. Equivalently, the difference quotient converges for each $v$
\begin{equation*}
    \lim\limits_{t\downarrow 0} \frac{F(x+tv) - F(x)}{t} = \langle DF(x), v\rangle,
\end{equation*}
and the convergence is uniform with respect to $v$ in finite sets. 
\end{definition}

\begin{definition}\label{def: cont diff}
We say that $F$ is continuously differentiable at $x$ when, on a neighborhood of $x$, the Gateaux derivative $Df(x)$ exists and is continuous as a mapping from $X$ to $\mathscr{L}(X,Y)$ (with its operator norm topology).
\end{definition}

\begin{definition}\label{def: stric deriv}
We say that $F$ admits a strict derivative at $x$, an element of $\mathscr{L}(X,Y)$ denoted by $D_s F(x)$, provided that for each $v$, the following holds
\begin{equation*}
    \lim\limits_{x'\to x,\, t\downarrow 0} \frac{F(x'+tv) - F(x')}{t} = \langle D_s F(x),v\rangle
\end{equation*}
and provided the convergence is uniform for $v$ in compact sets. (This last condition is automatic if $F$ is Lipschitz near $x$.) In this case, we also say that $F$ is strictly differentiable at $x$.
\end{definition}

\begin{proposition}\label{prop:diff}
If $F$ is continuously differentiable at $x$, then $F$ is strictly differentiable at $x$ and hence is Lipschitz near $x$.
\end{proposition}

\begin{proposition}
If $f$ is strictly differentiable at $x$, then $f$ is Lipschitz near $x$ and $\partial f(x) = \{D_s f(x)\}$. Conversely, if $f$ is Lipschitz near $x$ and $\partial f(x)$ reduces to a singleton $\xi$, then $D_s f(x) = \xi$.
\end{proposition}

\begin{theorem}[Chain rule]\label{thm:Chain rule}
Let $f=g\circ h$, where $h: X \to \mathbb{R}$ and $g:\mathbb{R}\to \mathbb{R}$. Suppose that $h$ is Lipschitz near $x$, and that $g$ is Lipschitz near $h(x)$.  Suppose also that $g$ is strictly differentiable at $h(x)$. Then we have
\begin{equation*}
    \partial f(x) = 
    \{D_s g(h(x))\,\xi \,: \, \; \xi \in \partial h(x)\}.
\end{equation*}
\end{theorem}

\begin{proposition}[Product]\label{prop:product}
Let $f_1, f_2$ be Lipschitz near $x$. Then $f_{1}f_{2}$ is Lipschitz near $x$, and one has
\begin{equation*}
    \partial (f_{1}f_{2})(x) \subset f_{2}(x)\,\partial f_{1}(x) + f_{1}(x)\,\partial f_{2}(x).
\end{equation*}
\end{proposition}

\section{Examples of Separable Hilbert Spaces}\label{app:examples of H}

A Hilbert space $H$ is called \textit{separable} if it admits a countable dense subset. Equivalently, $H$ possesses a countable orthonormal basis (see, e.g.,  \cite[Problem 17]{halmos2012hilbert}).

The following are standard examples of separable Hilbert spaces arising in analysis.

\subsection{The sequence space \texorpdfstring{$\ell^2$}{l2}}

Recall the space of square-summable sequences 
\[
    \ell^2(\mathbb{N}) = \left\{ x=(x_n)_{n\in\mathbb{N}} : \sum_{n=1}^\infty |x_n|^2 < \infty \right\}.
\]
The inner product is $\langle x,y\rangle = \sum_{n=1}^\infty x_n \overline{y_n}$.  The induced norm is $\|x\|_{\ell^2} = \left( \sum_{n=1}^\infty |x_n|^2 \right)^{1/2}$. The canonical orthonormal basis is $\mathbf{e}_n=(0,\dots,0,1,0,\dots)$, where the $1$ occurs in the $n$-th position. Hence $\ell^2(\mathbb{N})$ is separable. 

Sequence spaces arise naturally in Fourier coefficients, wavelet expansions, and discrete signal processing.

A fundamental theorem (see, e.g., \cite[Theorem 5.4.8, p.130]{hong2004real}) states the following. 
\begin{theorem}
Every infinite-dimensional separable Hilbert space is isomorphic to $\ell^2(\mathbb{N})$.
\end{theorem}

Thus all infinite-dimensional separable Hilbert spaces share the same Hilbert space structure, even though their concrete realizations and interpretations may differ substantially.

\subsection{The space \texorpdfstring{$L^2(\Omega)$}{L2(Omega)}}

Let $(\Omega,\mathcal F,\mu)$ be a $\sigma$-finite measure space. Define
\[
    L^2(\Omega) = \left\{ f:\Omega\to\mathbb{C} \, \text{ measurable} \, : \, \int_\Omega |f(x)|^2\,\dd\mu(x)<\infty \right\},
\]
where functions are identified almost everywhere. The inner product is $\langle f,g\rangle_{L^2} = \int_\Omega f(x)\overline{g(x)}\,\dd\mu(x)$.  Some examples are as follows.  
If $\Omega\subset\mathbb{R}^n$ is measurable and $\mu$ is Lebesgue measure, then $L^2(\Omega)$ is separable. For $\Omega=[0,2\pi]$, the Fourier basis $\left\{ (2\pi)^{-1/2}\,e^{inx} \right\}_{n\in\mathbb Z}$ 
is an orthonormal basis of $L^2([0,2\pi])$. 
Another example is  $L^2(\mathbb{R}^n)$ which plays a central role in Fourier analysis, probability theory, quantum mechanics, signal processing, and PDE theory. In dimension $n=1$, the Hermite functions form a complete orthonormal basis.

\subsection{Sobolev spaces \texorpdfstring{$H^k(\Omega)$}{Hk(Omega)}}

Let $\Omega\subset\mathbb{R}^n$ be an open set and let $k\in\mathbb{N}$. 
The Sobolev space $H^k(\Omega) = W^{k,2}(\Omega)$ is defined by
\[
    H^k(\Omega) = \left\{ f\in L^2(\Omega) \,:\, D^\alpha f\in L^2(\Omega) \text{ for all multi-indices } \alpha \text{ with } |\alpha|\le k \right\},
\]
where $D^\alpha f$ denotes the weak derivative. 
The inner product is
\[
    \langle f,g\rangle_{H^k} = \sum_{|\alpha|\le k} \int_\Omega D^\alpha f(x)\, \overline{D^\alpha g(x)} \,\dd x.
\]
The associated norm is $\|f\|_{H^k}^2 = \sum_{|\alpha|\le k} \|D^\alpha f\|_{L^2(\Omega)}^2$. 
If $\Omega\subset\mathbb{R}^n$ is open, then $H^k(\Omega)$ is a separable Hilbert space. 

Sobolev spaces play a central role for example in partial differential equations, variational methods, and functional analysis.  They provide the natural energy spaces for weak solutions of partial differential equations.

\subsection{The Hardy space \texorpdfstring{$H^2(\mathbb{D})$}{H2(D)}}

Let $\mathbb{D} = \{z\in\mathbb{C}:|z|<1\}$ denote the open unit disk. The Hardy space $H^2(\mathbb{D})$ consists of all holomorphic functions $f:\mathbb{D}\to\mathbb{C}$ such that
\[
    \sup_{0<r<1} \frac{1}{2\pi}\int_0^{2\pi} |f(re^{it})|^2\,\dd t < \infty.
\]
Every $f\in H^2(\mathbb{D})$ admits a power series expansion $f(z) = \sum_{n=0}^\infty a_n z^n$  with $\sum_{n=0}^\infty |a_n|^2<\infty$. The inner product is $\langle f,g\rangle = \sum_{n=0}^\infty a_n\overline{b_n}$,  where $g(z)=\sum_{n=0}^\infty b_n z^n$. The monomials $\{z^n\}_{n\ge0}$ form an orthogonal basis, hence $H^2(\mathbb{D})$ is separable.

Hardy spaces play important roles in complex analysis, operator theory, and control theory, where transfer functions of causal and stable systems naturally belong to Hardy spaces.

\subsection{The Bergman space \texorpdfstring{$A^2(\mathbb{D})$}{A2(D)}}

Let $\mathbb{D}\subset\mathbb{C}$ be the open unit disk. The Bergman space is
\[
    A^2(\mathbb{D}) = \left\{ f:\mathbb{D}\to\mathbb{C} \text{ holomorphic} : \int_{\mathbb{D}}|f(z)|^2\,\dd A(z)<\infty \right\},
\]
where $\dd A(z)=\dd x\,\dd y$ denotes planar Lebesgue measure on $\mathbb{C}\simeq\mathbb{R}^2$. The inner product is $\langle f,g\rangle = \int_{\mathbb{D}} f(z)\overline{g(z)} \,\dd A(z)$.  The normalized monomials $\left\{ \sqrt{(n+1)/{\pi}}\,z^n \right\}_{n\ge0}$ form an orthonormal basis. Therefore $A^2(\mathbb{D})$ is separable.

Bergman spaces arise in complex analysis, potential theory, and quantization.

\subsection{Reproducing kernel Hilbert spaces}

The previous examples arise naturally in classical analysis. Modern applications, particularly in approximation theory, statistics, and machine learning, frequently involve reproducing kernel Hilbert spaces.

Let $X$ be a set and let $K:X\times X\to\mathbb{C}$ be a positive-definite kernel. The associated reproducing kernel Hilbert space (RKHS) $\mathcal{H}_K$ is the unique Hilbert space of functions on $X$ satisfying
\begin{enumerate}
    \item For every $x\in X$, the function $K_x(\cdot):=K(\cdot,x)$ belongs to $\mathcal{H}_K$.
    
    \item (\emph{Reproducing property})
    For all $f\in\mathcal{H}_K$ and all $x\in X$, $f(x) = \langle f,K_x\rangle_{\mathcal{H}_K}$. 
\end{enumerate}
Many RKHSs arising in analysis and machine learning are separable whenever $X$ is a separable topological space and the kernel $K$ is continuous. 

General background on reproducing kernel Hilbert spaces can be found in \cite{aronszajn1950theory,berlinet2011reproducing}.  For Sobolev kernels and approximation-theoretic aspects, see \cite{wendland2004scattered}, while spline methods and their RKHS interpretation are discussed in \cite{wahba1990spline}. Gaussian and Mat\'ern kernels, together with their applications in statistical learning and Gaussian processes, are discussed in \cite{williams2006gaussian}.  Background on Sobolev spaces and Fourier analytic techniques can be found in \cite{stein1970singular}.  A discussion on the separability of RKHS can be found in \cite{owhadi2017separability}, and a comparison with the setting of Banach spaces is the object of \cite{fukumizu2011learning}.

We now present some standard examples of RKHS.

\subsubsection{Gaussian kernel RKHS}

Let $X=\mathbb{R}^d$,  and fix a parameter $\sigma>0$. 
The Gaussian kernel is
\[
    K(x,y) = \exp\left( -\frac{\|x-y\|^2}{2\sigma^2} \right), \qquad x,y\in\mathbb{R}^d.
\]
This kernel is continuous, symmetric, and positive definite. The associated RKHS $\mathcal{H}_K$ consists of very smooth functions on $\mathbb{R}^d$.  
Since $\mathbb{R}^d$ is separable and the kernel is continuous, $\mathcal{H}_K$ is separable. 
Moreover, for translation-invariant kernels such spaces admit a Fourier characterization through Bochner's theorem. In particular, the RKHS norm can often be expressed as a weighted Fourier norm of the form
\[
    \|f\|_{\mathcal H_K}^2 = \int_{\mathbb{R}^d} w(\xi) |\widehat{f}(\xi)|^2\,\dd\xi,
\]
where the weight $w$ depends on the Fourier transform of the kernel.
Gaussian RKHSs are fundamental for example in kernel methods in machine learning, Gaussian process regression, and approximation theory.

\subsubsection{Sobolev RKHSs}

Sobolev spaces can often be realized as RKHSs.  Let $\Omega\subset\mathbb{R}^d$ be an open set and let $s>d/2$. The Sobolev space $H^s(\Omega)$ embeds continuously into the space of continuous functions $H^s(\Omega)\hookrightarrow C^0(\Omega)$.  Consequently, point evaluation $f\mapsto f(x)$ is a bounded linear functional for every $x\in\Omega$. 
By the Riesz representation theorem, there exists a kernel $K(x,y)$ such that $f(x) = \langle f,K_x\rangle_{H^s(\Omega)}$. Hence $H^s(\Omega)$ becomes an RKHS. For $\Omega=\mathbb{R}^d$, one common Sobolev inner product is
\[
    \langle f,g\rangle_{H^s} =  \int_{\mathbb{R}^d} (1+\|\xi\|^2)^s \widehat{f}(\xi) \overline{\widehat{g}(\xi)} \,\dd\xi
\]
where $\widehat{f}$ denotes the Fourier transform of $f$. 
The corresponding reproducing kernel is translation invariant such as $K(x,y)=\Phi(x-y)$, where $\Phi$ is the inverse Fourier transform of $(1+\|\xi\|^2)^{-s}$. 

Sobolev RKHSs are central for example in elliptic PDE theory, scattered data interpolation, approximation theory, and numerical analysis. 

Since open subsets of $\mathbb{R}^d$ are separable metric spaces, $H^s(\Omega)$ is separable. 
Hence, under suitable regularity assumptions on $\Omega$, the space $H^s(\Omega)$ can be realized as an RKHS.

\subsubsection{Mat\'ern kernel RKHSs}

The Mat\'ern class provides an important family of kernels interpolating between rough and very smooth function spaces.

Let $X=\mathbb{R}^d$. Let $\nu>0$ denote the smoothness parameter, and $\ell>0$ denote the length scale. 
The Mat\'ern kernel is defined by
\[
    K_\nu(x,y) = \frac{2^{1-\nu}}{\Gamma(\nu)} \left( \frac{\sqrt{2\nu}\,\|x-y\|}{\ell} \right)^\nu K_\nu^{(B)} \left( \frac{\sqrt{2\nu}\,\|x-y\|}{\ell} \right),
\]
where $\Gamma$ is the Gamma function, and  $K_\nu^{(B)}$ denotes the modified Bessel function of the second kind. Such a kernel is symmetric, positive definite, and  translation invariant. Hence it defines a reproducing kernel Hilbert space $\mathcal{H}_{K_\nu}$. 

Its smoothness properties depends on the parameter $\nu$: small $\nu$ produces rough functions, while large $\nu$ produces smoother functions. Indeed, as $\nu\to\infty$, the Mat\'ern kernel converges to the Gaussian kernel. Special cases include 
\[
    \nu=\frac{1}{2}\qquad \text{ and } \qquad K(x,y)=e^{-\|x-y\|/\ell},
\]
which gives the exponential kernel, and $\nu=\frac{3}{2}$ or $\nu=\frac{5}{2}$, which yield commonly used kernels with simple closed forms. 

It can be shown that the RKHS norm is equivalent to a Sobolev norm
\[
    \|f\|_{\mathcal{H}_{K_\nu}}^2 \sim \int_{\mathbb{R}^d} |\widehat{f}(\xi)|^2 \left( 1+\|\xi\|^2 \right)^{\nu+d/2} \,\dd\xi.
\]
Thus the RKHS norm is equivalent to a Sobolev norm of order $\nu+d/2$, so Mat\'ern RKHSs exhibit Sobolev-type smoothness properties.

Mat\'ern kernels are fundamental for example in Gaussian process regression, spatial statistics, uncertainty quantification, and Bayesian inverse problems. 
Because $\mathbb{R}^d$ is separable and the kernel is continuous, Mat\'ern RKHSs are separable Hilbert spaces.

\subsubsection{Spline spaces}

Spline spaces provide another important class of RKHSs.

Let $a=x_0<x_1<\cdots<x_n=b$ be a partition of an interval $[a,b]\subset\mathbb{R}$. 
A spline of degree $m$ is a polynomial function of degree at most $m$ on each interval $[x_i,x_{i+1}]$, and is moreover globally of class $C^{m-1}$.  A classical example is the cubic spline space corresponding to $m=3$. Many spline spaces arise as RKHSs associated with variational problems. For example, consider the Sobolev space $H^2([a,b])$.  Define the seminorm
\[
    J(f) = \int_a^b |f''(x)|^2\,\dd x.
\]
The minimizers of $J(f)$ under interpolation constraints are cubic splines. 
The associated reproducing kernel generates a spline RKHS. 
Explicit reproducing kernels can be written for many spline spaces, although their precise form depends on the imposed boundary conditions and normalization.

Spline RKHSs are widely used for example in interpolation theory, finite element methods, and  statistical regression. 
These spaces are typically separable because they are subspaces of separable Sobolev spaces.

\bibliography{bibliography}
\end{document}